\documentclass[11pt]{article}
\usepackage{latexsym}
\usepackage{indentfirst}
\usepackage{amssymb,amsfonts,amsmath,amsthm}
\usepackage{graphicx}
\usepackage{mathrsfs}
\usepackage{array}
\usepackage{color}
\usepackage{epstopdf}
\usepackage[all]{xy}
\usepackage{hyperref}
\usepackage{url}
\usepackage{enumerate}
\usepackage{geometry}
\newtheorem{thm}{Theorem}[section]
\newtheorem{prop}[thm]{Proposition}
\newtheorem{construction}[thm]{Construction}
\newtheorem{notation}[thm]{Notation}
\newtheorem{ques}[thm]{Question}
\newtheorem{convention}[thm]{Convention}
\newtheorem{lem}[thm]{Lemma}
\newtheorem{conj}[thm]{Conjecture}

\newtheorem{lem-def}[thm]{Lemma-Definition}
\newtheorem{cor}[thm]{Corollary}
\theoremstyle{remark}
\newtheorem{exam}[thm]{Example}

\newtheorem{rmk}{Remark}[section]
\theoremstyle{definition}
\newtheorem{dfn}[thm]{Definition}

\numberwithin{equation}{section}

\newcommand{\frakm}{{\mathfrak m}}

\newcommand{\frakU}{{\mathfrak U}}

\newcommand{\frakX}{{\mathfrak X}}

\newcommand{\bC}{{\mathbb C}}

\newcommand{\bL}{{\mathbb L}}
\newcommand{\bM}{{\mathbb M}}
\newcommand{\bN}{{\mathbb N}}

\newcommand{\bQ}{{\mathbb Q}}

\newcommand{\bZ}{{\mathbb Z}}

\newcommand{\calA}{{\mathcal A}}
\newcommand{\calB}{{\mathcal B}}

\newcommand{\calD}{{\mathcal D}}
\newcommand{\calE}{{\mathcal E}}
\newcommand{\calF}{{\mathcal F}}

\newcommand{\calH}{{\mathcal H}}
\newcommand{\calI}{{\mathcal I}}

\newcommand{\calL}{{\mathcal L}}
\newcommand{\calM}{{\mathcal M}}

\newcommand{\calO}{{\mathcal O}}
\newcommand{\calP}{{\mathcal P}}

\newcommand{\rA}{{\mathrm A}}

\newcommand{\rd}{{\mathrm d}}

\newcommand{\rH}{{\mathrm H}}
\newcommand{\rI}{{\mathrm I}}

\newcommand{\rK}{{\mathrm K}}
\newcommand{\rL}{{\mathrm L}}
\newcommand{\rM}{{\mathrm M}}

\newcommand{\rR}{{\mathrm R}}

\newcommand{\rW}{{\mathrm W}}

\newcommand{\Zp}{{\bZ_p}}

\newcommand{\Qp}{{\bQ_p}}

\newcommand{\Ainf}{{\mathrm{A_{inf}}}}

\newcommand{\OX}{{\widehat \calO_X}}    
\newcommand{\OXp}{{\widehat \calO^+_X}}

\newcommand{\Gal}{{\mathrm{Gal}}}           %Galois groups
\newcommand{\Spa}{{\mathrm{Spa}}}           %Affinoid adic spaces
\newcommand{\Spf}{{\mathrm{Spf}}}           %formal affine schemes
\newcommand{\Spec}{{\mathrm{Spec}}}         %affine schemes

\newcommand{\HIG}{{\mathrm{HIG}}}
\newcommand{\Ker}{{\mathrm{Ker}}}           %kernal
\newcommand{\Ima}{{\mathrm{Im}}}            %image
\newcommand{\Hom}{{\mathrm{Hom}}}           %Hom-functor
\newcommand{\Fil}{{\mathrm{Fil}}}           %Filtration
\newcommand{\Sym}{{\mathrm{Sym}}}           %Symmetic algebra
\newcommand{\Rep}{{\mathrm{Rep}}}
\newcommand{\Vect}{{\mathrm{Vect}}}
\newcommand{\GL}{{\mathrm{GL}}}             %genetal linear group
\newcommand{\cts}{{\mathrm{cts}}}
\newcommand{\et}{{\mathrm{\acute{e}t}}}    %etale

\newcommand{\nil}{{\mathrm{nil}}}
\newcommand{\pd}{{\mathrm{pd}}}
\newcommand{\proet}{{\mathrm{pro\acute{e}t}}}%pro-etale
\usepackage{relsize}
\usepackage[bbgreekl]{mathbbol}
\DeclareSymbolFontAlphabet{\mathbb}{AMSb} %to ensure that the meaning of \mathbb does not change
\DeclareSymbolFontAlphabet{\mathbbl}{bbold}
\newcommand{\Prism}{{\mathlarger{\mathbbl{\Delta}}}}

\newcommand{\za}{\langle}
\newcommand{\ya}{\rangle}

\newcommand{\id}{{\mathrm{id}}}             %Identity
\newcommand{\dlog}{{\mathrm{dlog}}}

\newcommand{\can}{{\mathrm{can}}}

\begin{document}

\title{Integral $p$-adic non-abelian Hodge theory for small representations}

\author{Yu Min\footnote{Y. M.: Department of Mathematics, Imperial College London, London SW7 2RH, UK. {\bf Email:} y.min@imperial.ac.uk} and Yupeng Wang\footnote{Y. W.: Beijing International Center of Mathematical Research, Peking University, Yiheyuan Road 5, Beijing, 100190, China. {\bf Email:} 1501110016@pku.edu.cn}}

\date{}

\maketitle
	
\begin{abstract}
    Let $\frakX$ be a smooth $p$-adic formal scheme over $\calO_C$ with rigid generic fiber $X$.
    In this paper, we construct a new integral period sheaf $\calO\widehat \bC_{\pd}^+$ on $X_{\proet}$ and use it to establish an integral $p$-adic Simpson correspondence for small $\OXp$-representations on $X_{\proet}$ and small Higgs bundles on $\frakX_{\et}$, which recovers rational $p$-adic Simpson correspondence for small coefficients after inverting $p$ (at least in the good reduction case). Moreover, for a small $\OXp$-representation $\calL$ with induced Higgs bundle $(\calH,\theta_{\calH})$, we provide a natural morphism $\HIG(\calH,\theta_{\calH})\to\rR\nu_*\calL$ with a bounded $p^{\infty}$-torsion cofiber. Finally, we shall use this natural map to study an analogue of Deligne--Illusie decomposition with coefficients in small $\OXp$-representations. %which is not known in the classical $p$-adic non-abelian Hodge theory.
\end{abstract}



\tableofcontents

\section{Introduction} 

   Let $C$ be a complete algebraically closed extension of $\Qp$ with the ring of integers $\calO_C$. The $p$-adic Simpson correspondence aims to establish a $p$-adic analogue of the classical Simpson correspondence \cite{Sim} in complex geometry. Namely, for a smooth variety $X$ over $C$, it hopes to clarify the relationship between the category of representations of the \'etale fundamental group of $X$ and the category of Higgs bundles on $X_{\et}$.

   In \cite{Fal05}, assuming $X$ admits a good integral model $\frakX$ over $\calO_C$, Faltings described an equivalence between the category of small generalised representations and that of small Higgs bundles over $X$. His idea was realised and generalised systematically by Abbes--Gros--Tsuji \cite{AGT}, Tsuji \cite{Tsu18} and Abbes--Gros \cite{AG22}% by using Faltings' topos
   . Their constructions turn out to be coincident with the previous work \cite{DW05} of Deninger--Werner, proved by Xu \cite{Xu17}. One of the key ingredients in the work of Abbes--Gros in \cite{AGT} is a period sheaf in the spirit of Fontaine's $p$-adic Hodge theory. On the other hand, in \cite{AGT}, Tsuji introduced Higgs isocrystals on his Higgs site as an interpretation of Higgs bundles, which is similar to interpreting $D$-modules as crystals on the infinitesimal/crystalline site. The Higgs site of Tsuji also naturally produces a period ring, which is closely related to the period ring of Abbes--Gros. Based on \cite{LZ17}, the second author  constructed an ``overconvergent pro-\'etale period sheaf" $\calO\bC^{\dagger}$ to establish a $p$-adic Simpson correspondence for rigid varieties with good reductions in \cite{Wang}, which is compatible with \cite{Fal05} and \cite{AGT}.

 Note that all the works mentioned above are on the rational level. One might wonder what will happen without inverting $p$. The goal of this paper is to try to provide some results in this direction, at least dealing with small generalised representations and small Higgs bundles. %Inspired by the prismatic theory, 
 By adapting the construction in \cite{Wang}, we will define a new integral pro-\'etale period sheaf equipped with a Higgs field $(\calO\widehat\bC_{\pd}^+,\Theta)$ and use it to establish an integral $p$-adic Simpson correspondence for small coefficients. It is worth mentioning that our construction of $\calO\widehat\bC_{\pd}^+$ is largely inspired by the work on Hodge--Tate crystals in \cite{Tia23} and \cite{MW22} (see Remark \ref{motivation-construction} for more details).

Once the integral $p$-adic Simpson correspondence is available, one could try to connect the $p$-adic theory and the mod $p$ non-abelian Hodge theory in \cite{OV}. This is indeed one of our main motivations. It turns out that our integral correspondence shares a lot of similarities with the mod $p$ theory in \cite{OV}. In particular, we will also investigate an analogous Deligne--Illusie decomposition in our setting. To further clarify the mysterious relation between the $p$-adic theory and the mod $p$ theory, we believe one has to investigate the prismatic theory simultaneously.

Considering small objects is still the very first step of Faltings' vision of $p$-adic Simpson correspondence. Very recently, for a liftable proper smooth rigid variety, Heuer has given a construction of an equivalence between the category of generalised representations and that of Higgs bundles in \cite{Heu23} (with rational coefficients on both sides), generalising the result of Faltings in the curve case. However, it is still an open question to describe the Higgs bundles on $X$ which are induced by $C$-local systems on $X_{\et}$. Some cases are known, for example, when $X$ is abeloid \cite{HMW} and when the local system has dimension one \cite{Heu22a}. When $X$ is a curve, Xu made a conjecture in \cite{Xu22} on the essential image of $C$-local systems in the category of Higgs bundles via his (twisted) Higgs functor (described in \cite{Fal05} by Faltings at first). Note that as long as one gets integral Higgs bundles on $\frakX$, one can consider their restrictions to the special fiber of $\frakX$ and then try to apply results in \cite{DW05} to get $C$-local systems. So our integral $p$-adic Simpson correspondence might shed some light on the above open problem and we will focus on this in future work.

\subsection{Main results}
 \subsubsection{Integral $p$-adic Simpson correspondence}

  We first introduce some notations (see also \S \ref{SSec-Notation}). Let $K$ be a complete discrete valuation field over $\Qp$ with ring of integers $\calO_K$ and perfect residue field $\kappa$. Let $C$ be the $p$-adic completion of a fixed algebraic closure $\overline K$ of $K$ and $\calO_C$ be the ring of integers. Then we can define  the ramified Witt ring $\rA_{\inf,K} = \rA_{\inf}(C,\calO_C)\otimes_{\rW(\kappa)}\calO_K$. Let $\rI_K$ be the kernel of the natural surjection $\theta:\rA_{\inf,K}\to\calO_C$ and $\rA_2 = \rA_{\inf,K}/\rI_K^2$. Then there exists an element $\rho_K\in\calO_C$ with valuation $\nu_p(\rho_K) = \nu_p(\calD_K)+\frac{1}{p-1}$ such that $I_K\rA_2\cong \frac{1}{\rho_K}\calO_C(1)$, where $\calD_K$ is the ideal of relative differential of $\calO_K$ over $W(\kappa)$. When $\calO_K = \rW(\kappa)$, one can choose $\rho_K = \zeta_p-1$. When contexts are clear, we also denote $\rho_K$ by $\rho$.
  
  Our main result is the following, which upgrades Faltings' $p$-adic Simpson correspondence (cf. \cite{Fal05,AGT,Tsu18,Wang}, etc.) to the integral level.
 %For our purpose, we will introduce a new period sheaf together with a Higgs field $(\calO\widehat \bC_{\pd}^+,\Theta)$ which is a resolution of $\OXp$ and then apply it to the integral $p$-adic non-abelian Hodge theory. More precisely, we can upgrade Faltings' $p$-adic Simpson correspondence (cf. \cite{Fal05,AGT,Tsu18,Wang}, etc.) to the integral level. Our main result is
 \begin{thm}[Theorem \ref{Thm-IntegralSimpson}]\label{Intro-Main result} 
      Assume $a\geq \frac{1}{p-1}$. Let $\frakX$ be a smooth $p$-adic formal scheme of dimension $d$ over $\calO_C$, which is liftable to $\Spf(A_2)$. Let $\nu:X_{\proet}\to\frakX_{\et}$ be the natural projection of sites. Fix an $\rA_2$-lifting of $\frakX$. There exists a pro-\'etale period sheaf together with a Higgs field $(\calO\widehat \bC_{\pd}^+,\Theta)$ (which is induced from the fixed lifting and thus depends on $K$ as so does $\rA_2$), such that the following assertions are true:
      \begin{enumerate}
          \item[(1)] For any $a$-small $\OXp$-representation of rank $r$ (cf. Definition \ref{Dfn-SmallRep}) on $X_{\proet}$, we have that $\rR\nu_*(\calL\otimes_{\OXp}\calO\widehat \bC_{\pd}^+)$ is concentrated in degree $[0,d]$ and that the complex $\rL\eta_{\rho(\zeta_p-1)}\rR\nu_*(\calL\otimes_{\OXp}\calO\widehat \bC_{\pd}^+)$ is concentrated in degree $0$ and coincides with $\nu_*(\calL\otimes_{\OXp}\calO\widehat \bC_{\pd}^+)[0]$. In particular, $(\nu_*(\calL\otimes_{\OXp}\calO\widehat \bC_{\pd}^+),\nu_*(\Theta_{\calL}))$ is an $a$-small Higgs bundle (cf. Definition {\ref{Dfn-SmallHiggs}}) of rank $r$ on $\frakX_{\et}$.

          \item[(2)] For any $a$-small Higgs bundle $(\calH,\theta_{\calH})$ of rank $r$ on $\frakX_{\et}$, we define $\calL=(\calH\otimes_{\calO_{\frakX}}\calO\widehat \bC_{\pd}^+)^{\Theta_{\calH} = 0}$, where $\Theta_{\calH}:=\theta_{\calH}\otimes id+ id\otimes \Theta$. Then $\calL$ is an $a$-small $\OXp$-representation of rank $r$ on $X_{\proet}$.

          \item[(3)] The functors in Items (1) and (2) are quasi-inverses to each other and hence define an equivalence between categories
          \[\Rep^{\geq a}(\OXp)\simeq \HIG^{\geq a}(\frakX)\]
          of a-small $\OXp$-representations on $X_{\proet}$ and $a$-small Higgs bundles on $\frakX_{\et}$,
          which preserves tensor products and dualities. Moreover, if $\calL$ corresponds to $(\calH,\theta_{\calH})$, then there exists an isomorphism of Higgs fields
          \[(\calL\otimes_{\OXp}\calO\widehat \bC_{\pd}^+,\Theta_{\calL})\cong (\calH\otimes_{\calO_{\frakX}}\calO\widehat \bC_{\pd}^+,\Theta_{\calH}).\]
      \end{enumerate} 
  \end{thm}
 % \begin{rmk}
 %      By Zariski--\'etale comparison (and Remark \ref{Rmk-Small}), one can replace the \'etale site $\frakX_{\et}$ by the Zariski site $\frakX_{\rm Zar}$ in the statements of all results in this paper.
%  \end{rmk}
  \begin{rmk}\label{Rmk-OptimalDecalage}
       %In Theorem \ref{Intro-Main result} (1), we have to apply $\rL\eta_{\rho(\zeta_p-1)}$ rather than $\rL\eta_{\zeta_p-1}$ as in \cite{BMS18}. The reason is that for our period sheaf $\calO\widehat \bC_{\pd}^+$, there exist local sections of $\rR^{\geq 1}\nu_*(\calO\widehat \bC^+_{\pd})$ which are exactly killed by $\rho(\zeta_p-1)$ (cf. Lemma \ref{Lem-Key}).
       The functor $\rL\eta_{\rho(\zeta_p-1)}$ is optimal in the sense that there do exist local sections of $\rR^{\geq 1}\nu_*(\calO\widehat \bC^+_{\pd})$ which are exactly killed by $\rho(\zeta_p-1)$ (cf. Lemma \ref{Lem-Key}). This is different from the trivial coefficient case in \cite{BMS18}.
   \end{rmk}

   \begin{rmk}
       The local version of the equivalence part of Theorem \ref{Intro-Main result} was obtained in \cite{Fal05}, \cite{AGT} and \cite{Tsu18} in a very general setting and was generalised to any rigid group $G$ in \cite[\S 6]{Heu22b}. We thank Ben Heuer for informing us of the latter. Also it should be possible to globalize the local equivalence using the period sheaf $\widehat {\mathscr{C}}^{(r)}$ (with Higgs fields) introduced in \cite{AGT} as pointed out to us by the referee. That being said, it seems hard to compare cohomology on the integral level by using $\widehat {\mathscr{C}}^{(r)}$ (or the overconvergent period sheaf $(\calO\bC^{\dagger,+},\widetilde \Theta)$ appearing in \cite[I.4.7]{AGT} and \cite[\S 2.3]{Wang}) as its induced Higgs complex is not exact. Instead, using our period sheaf $\calO\widehat \bC_{\pd}^+$, one could obtain some results on cohomology comparison. See \S\ref{Intro-CompareCohomology} below.
   \end{rmk}

%   The key point of proving Theorem \ref{Intro-Main result} is that the Higgs complex associated to $(\calO\widehat \bC_{\pd}^+,\Theta)$ is a resolution of the integral structure sheaf $\OXp$, which is one of our major discoveries. In $p$-adic non-abelian Hodge theory, people often consider the ``overconvergent period sheaf'' $(\calO\bC^{\dagger},\widetilde \Theta)$ (cf. \cite[I.4.7]{AGT}, \cite[\S 2.3]{Wang}). Although $\calO\bC^{\dagger}$ has a $\widetilde \Theta$-preserving integral model $\calO\bC^{\dagger,+}$, it is hard to use $\calO\bC^{\dagger,+}$ to study the integral theory as the induced Higgs complex $\HIG(\calO\bC^{\dagger,+},\widetilde \Theta)$ is highly non-exact. More precisely, for any $i\geq 1$, $\rH^i(\HIG(\calO\bC^{\dagger,+},\widetilde \Theta))$ has unbounded $p^{\infty}$-torsions (cf. Remark \ref{Rmk-ComparePeriodSheaf}).
   %So there might be no hope to get an integral version of Faltings' Simpson correspondence and the natural map (\ref{Equ-Intro-CanonicalMap}) by using $\calO\bC^{\dagger,+}$ merely, even in the $\calL=\OXp$ case. Therefore, our result is new and is an improvement of known results in \textit{loc.cit.}.

   The correspondence in Theorem \ref{Intro-Main result} is compatible with \cite[Theorem 1.1]{Wang} (and \cite{AGT}) after inverting $p$. %See Remark \ref{Intro-Rmk-ComparePeriodSheaf} for further explanation.
       More precisely, let $\calO\widehat \bC_{\pd} = \calO\widehat \bC_{\pd}^+[\frac{1}{p}]$ with induced Higgs field $\Theta$. One can regard the overconvergent period sheaf $\calO\bC^{\dagger}$ in \emph{loc.cit.} as a sub-$\OX$-algebra of $\calO\widehat \bC_{\pd}$ by identifying $\widetilde \Theta$ with $(\zeta_p-1)\Theta$ (cf. Proposition \ref{Prop-ComparePeriodSheaf}) and use this to obtain the desired compatibility (cf. Corollary \ref{Cor-CompareWithWang}).
  \begin{rmk}
      The integral period sheaf $(\calO\widehat \bC_{\pd}^+,\Theta)$ has been constructed independently in \cite[Section 8.1]{AHLB} by  Ansch\"utz--Heuer--Le Bras. The key ingredient of their approach is a twisted version, provided by a lift of the formal scheme, of the Hodge--Tate structure map due to Bhatt and Lurie \cite{BL22b}. It is possible to get a derived correspondence among small Higgs bundles, small generalised representations (i.e. $v$-bundles) and small Hodge--Tate crystals by using a derived version of Theorem \ref{Intro-Main result}, similar to the rational case \cite[Theorem 1.6]{AHLB}.
  \end{rmk}

  Note that $\rL\eta_{\rho(\zeta_p-1)}$ is {\bf NOT} an exact functor. This causes a lot of trouble in comparing $\HIG(\calH,\theta_{\calH})$ with $\rL\eta_{\rho(\zeta_p-1)}\rR\nu_*\calL$ for corresponding $\calL$ and $(\calH,\theta_{\calH})$ in Theorem \ref{Intro-Main result}. We will briefly discuss this problem in \S \ref{Intro-CompareCohomology}. For the moment, let us satisfy ourselves with the following important corollary of Theorem \ref{Intro-Main result}.
   \begin{cor}[Corollary \ref{Cor-IntegralSimpson}]\label{Intro-Cor-BMS}
       Keep notations as in Theorem \ref{Intro-Main result}. For any $a$-small $\OXp$-representation $\calL$ on $X_{\proet}$ with induced Higgs bundle $(\calH,\theta_{\calH})$ via the equivalence in Theorem \ref{Intro-Main result} (3), we have a natural morphism
       \begin{equation}\label{Equ-Intro-CanonicalMap}
       \HIG(\calH,\theta_{\calH})\to \rR\nu_*\calL
       \end{equation}
       with cofiber killed by $(\rho(\zeta_p-1))^{\max\{d+1,2(d-1)\}}$, which is functorial in $\calL$.
   \end{cor}

\subsubsection{Compare cohomologies}\label{Intro-CompareCohomology}
  Keep assumptions as in Theorem \ref{Intro-Main result}. The following question is natural: 
  \begin{ques}\label{Intro-Ques-CompareCohomology}
      Can we compute $\rL\eta_{\rho(\zeta_p-1)}\rR\nu_*\calL$ via the Higgs complex $\HIG(\calH,\theta_{\calH})$? More precisely, do we have a quasi-isomorphism
      \[\HIG(\calH,\theta_{\calH})\simeq \rL\eta_{\rho(\zeta_p-1)}\rR\nu_*\calL ?\]
  \end{ques} 
  The answer can not be positive if $\rL\eta_{\rho(\zeta_p-1)}\rR\nu_*\calL$ is not a perfect complex in $D(\frakX)$. However, the next theorem shows that this phenomenon never happens.
  \begin{thm}[Theorem \ref{Thm-PerfectComplex}]\label{Intro-PerfectComplex}
    Let $\frakX$ be a smooth $p$-adic formal scheme over $\calO_C$ and not necessarily liftable. 
      Assume $a\geq \frac{1}{p-1}$ and $\lambda\in\calO_C$ with $\nu_p(\lambda)\leq\nu_p(\rho)$. Then for any $a$-small $\OXp$-representation $\calL$, the $\rL\eta_{(\zeta_p-1)\lambda}\rR\nu_*\calL$ is a perfect complex in $D(\frakX)$ concentrated in degree $[0,d]$ with $p$-torsion free $\rH^0$ and there is a natural map $\rL\eta_{(\zeta_p-1)\lambda}\rR\nu_*\calL\to\rR\nu_*\calL$.
  \end{thm}
  \begin{rmk}
      When $\lambda = 1$ and $\calL = \OXp$, the result was obtained in \cite[\S 8]{BMS18}.
  \end{rmk}
  
  As a consequence, we know that $\rL\eta_{(\zeta_p-1)\rho}\rR\nu_*\calL$ is a perfect complex concentrated in degree $[0,d]$ as expected. So it is reasonable to compare $\HIG(\calH,\theta_{\calH})$ and $\rL\eta_{(\zeta_p-1)\rho}\rR\nu_*\calL$. The following theorem asserts that the first truncations of the two are quasi-isomorphic to each other.
 % \begin{thm}\label{Intro-Curve} 
 %     Let $\frakX$ be a smooth curve over $\calO_C$ and fix a lifting to $\rA_2$. Let $a\geq \frac{1}{p-1}$. Then for any $a$-small $\OXp$-representation $\calL$ with induced Higgs bundle $(\calH,\theta_{\calH})$, we have
 %     \[\HIG(\calH,\theta_{\calH})\simeq\rL\eta_{\rho(\zeta_p-1)}\rR\nu_*\calL.\]
 % \end{thm}
 % Indeed we can get a more general result such that Theorem \ref{Intro-Curve} follows immediately.
  \begin{thm}[Theorem \ref{Thm-TruncationOne}]\label{Intro-Truncation}
      Let $\frakX$ be a liftable smooth $p$-adic formal scheme over $\calO_C$ with a fixed lifting to $\rA_2$ as in Theorem \ref{Intro-Main result}. Assume $a\geq\frac{1}{p-1}$. Then for any $a$-small $\OXp$-representation $\calL$ with induced Higgs bundle $(\calH,\theta_{\calH})$, the composite
    $\tau^{\leq 1}\HIG(\calH,\theta_{\calH})\to\HIG(\calH,\theta_{\calH})\xrightarrow{{\rm Cor.} \ref{Intro-Cor-BMS}}\rR\nu_*\calL$ uniquely factors through $\rL\eta_{\rho(\zeta_p-1)}\rR\nu_*\calL\xrightarrow{{\rm Thm.} \ref{Intro-PerfectComplex}}\rR\nu_*\calL$ and induces a quasi-isomorphism 
    \[\tau^{\leq 1}\HIG(\calH,\theta_{\calH})\xrightarrow{\simeq} \tau^{\leq 1}\rL\eta_{\rho(\zeta_p-1)}\rR\nu_*\calL.\]
    In particular, if $\frakX$ is a smooth curve over $\calO_C$, then we have
    \[
   \HIG(\calH,\theta_{\calH})\xrightarrow{\simeq} \rL\eta_{\rho(\zeta_p-1)}\rR\nu_*\calL.
    \]
  \end{thm}
It is not clear to us how to extend this comparison to higher degrees. But inspired by the mod $p$ non-abelian Hodge theory in \cite{OV}, we have the following conjecture.

 \begin{conj}[Conjecture \ref{Conj-CompareCohomology}]\label{Intro-Conj-CompareCohomology}
      Keep assumptions in Theorem \ref{Intro-Main result}. Then for any $a$-small $\OXp$-representation $\calL$ with induced Higgs bundle $(\calH,\theta_{\calH})$, denote by $r$ the nilpotency length of $(\zeta_p-1)\theta_{\calH}$ modulo $p$, and then the natural morphism in Corollary \ref{Cor-IntegralSimpson} induces a quasi-isomorphism
     \[\tau^{\leq p-r}\HIG(\calH,\theta_{\calH})\simeq \tau^{\leq p-r}\rL\eta_{\rho(\zeta_p-1)}\rR\nu_*\calL.\]
  \end{conj}
  This should be considered as an analogue of the Deligne--Illusie decomposition with coefficients in mixed characteristic. Although it is difficult to study the general case, the trivial coefficient case is quite clear.
%\subsubsection{A conjectural analogue of Deligne--Illusie decomposition with coefficients} 
  In \cite{Min21}, using a quasi-isomorphism in \cite[Prop. 15]{BMS18} together with a standard argument of Deligne--Illusie \cite[Thm. 2.1]{DI}, the first author proved the following decomposition of Deligne--Illusie type:
  \begin{thm}[\emph{\cite[Thm. 4.1]{Min21}}]\label{Intro-Min}
      Let $\frakX$ be the base change of a smooth $p$-adic formal scheme over $\calO_K=W(\kappa)$ along the inclusion $W(\kappa)\to \calO_C$. Then the natural lifting of $\frakX$ induces a quasi-isomorphism $\gamma:\tau^{\leq p-1}\HIG(\calO_{\frakX},0)\to\tau^{\leq p-1}\rL\eta_{\zeta_p-1}\rR\nu_*\OXp$.
  \end{thm}
  In particular, there exists a quasi-isomorphism $\gamma_1:\tau^{\leq 1}\HIG(\calO_{\frakX},0)\to\tau^{\leq 1}\rL\eta_{\zeta_p-1}\rR\nu_*\OXp$.  So it is natural to compare $\gamma_1$ with the quasi-isomorphism in Theorem \ref{Intro-Truncation} for $\calL = \OXp$. Using the same argument in \cite[Thm. 4.1]{Min21} together with some calculations, we have
  \begin{thm}[Theorem \ref{Thm-DI} and Theorem \ref{Thm-compare DI with Min}]\label{Intro-DI}
     Let $\frakX$ be a liftable smooth $p$-adic formal scheme over $\calO_C$  and fix a lifting. 
      
      \begin{enumerate}
          \item[(1)] There exists a quasi-isomorphism $\gamma':\tau^{\leq p-1}\HIG(\calO_{\frakX},0)\to\tau^{\leq p-1}\rL\eta_{(\zeta_p-1)\rho}\rR\nu_*\OXp$. 
%If $\frakX$ is the base change of a smooth $p$-adic formal scheme over $W(\kappa)$ along the inclusion $W(\kappa)\to \calO_C$ (and thus $\rho = \zeta_p-1$), 

          \item[(2)] If $\frakX$ is the base change of a smooth $p$-adic formal scheme over $\calO_K = \rW(\kappa)$ (and thus $\rho = \zeta_p-1$) and the fixed lifting is the natural one induced by $W(\kappa)\to A_2 = \Ainf/(\xi^2)$, then the above $\gamma'$ is compatible with $\gamma$ in Theorem \ref{Intro-Min} in the following sense: There is a commutative diagram
          \begin{equation*}
              \xymatrix@C=0.5cm{
              \bigoplus_{i=0}^{p-1}(\zeta_p-1)^i\widehat \Omega_{\frakX}^i(-i)[-i]\ar[d]^{\iota_{\zeta_p-1}}\ar[rr]^{\gamma'}&&\tau^{\leq p-1}\rL\eta_{(\zeta_p-1)^2}\rR\nu_*\OXp\ar[r]\ar[d]&\rR\nu_*\OXp\ar@{=}[d]\\
              \bigoplus_{i=0}^{p-1}(\zeta_p-1)^i\widehat \Omega_{\frakX}^i(-i)[-i]\ar[rr]^{\gamma}&&\tau^{\leq p-1}\rL\eta_{\zeta_p-1}\rR\nu_*\OXp\ar[r]&\rR\nu_*\OXp,
              }
          \end{equation*}
          where $\iota_{\zeta_p-1}$ is induced by multiplication $(\zeta_p-1)^i$ at each degree $0\leq i\leq p-1$.
      \end{enumerate}
  \end{thm}
%  Note that Theorem \ref{Intro-DI}(2) provides another evidence to Question \ref{Intro-Ques-CompareCohomology}. Inspired by this, we make the following conjecture about an analogue of Deligne--Illusie decomposition with coefficients in small $\OXp$-representations:
 For the trivial coefficient, the nilpotency degree $r$ is $1$. So Theorem \ref{Intro-DI} shows that Conjecture \ref{Intro-Conj-CompareCohomology} holds in this case.
 We also remark the quasi-isomorphism $\gamma$ in Theorem \ref{Intro-Min} can not extend to the whole complex. This is a consequence of a recent work \cite{Pet23} of Petrov and one can see Remark \ref{Rmk-Min} for a rough explanation. Note that by Theorem \ref{Intro-Truncation}, we know that Conjecture \ref{Intro-Conj-CompareCohomology} also holds true for curves as well as when $p=2$. But the general case still remains mysterious to us.

\subsection{Organizations}
  The paper is organized as follows: In \S\ref{Sec-PeriodSheaf}, we construct the desired period sheaf $\calO\widehat \bC_{\pd}^+$ together with Higgs field $\Theta$. In \S\ref{Sec-GammaCohomology}, we consider $\Gamma$-cohomology of certain representations in a more general setting. This will help to study the local version of Theorem \ref{Intro-Main result} in \S\ref{SSec-Local Simpson}. In \S\ref{Sec-MainResult}, we first prove the local integral Simpson correspondence and then reduce Theorem \ref{Intro-Main result} to this case. The last section \S \ref{Sec-PerfectComplex} is devoted to proving Theorems \ref{Intro-PerfectComplex} and \ref{Intro-Truncation}, and discussing the analogue of Deligne--Illusie decomposition for small $\OXp$-representations.

\subsection{Notations}\label{SSec-Notation}
  Let $K$ be a complete discretely valued $p$-adic field of mixed characteristic $(0,p)$ with perfect residue field $\kappa$. Let $C$ be the $p$-adic completion of a fixed algebraic closure $\overline K$ of $K$. Denote by $\calD_K$ the ideal of relative differential of $\calO_K$ over $\rW(\kappa)$. We fix a choice of compatible system $\{\zeta_{p^n}\}_{n\geq 1}$ of primitive $p$-roots of unit and for any $\alpha = \frac{q}{p^r}\in \bN[\frac{1}{p}]$ with $\gcd(q,p) = 1$ and $r\geq 0$, define $\zeta^{\alpha} = \zeta_{p^r}^q$. Let $t$ be Fontaine's $p$-adic analogue of ``$2\pi i$'', and then it stands for the basis of the Tate twist $\Zp(1) = \Zp\cdot t$ such that $g(t) = \chi(g)t$ for any $g\in G_K:=\Gal(\overline K/K)$, where $\chi:G_K\to\bZ_p^{\times}$ is the cyclotomic character corresponding to the given $\{\zeta_{p^n}\}_{n\geq 0}$; that is, we have $g(\zeta_{p^n}) = \zeta_{p^n}^{\chi(g)}$ for all $n\geq 0$ and $g\in G_K$. For any $\Zp$-module $M$, we set $M(1):=M\otimes_{\Zp}\Zp(1) = M\cdot t$.
  
  Let $\rA_{\inf,K} = \rA_{\inf}(C,\calO_C)\otimes_{\rW(\kappa)}\calO_K$ be the ramified Witt ring over $\calO_C^{\flat}$, $\rI_K$ the kernel of the natural surjection $\theta:\rA_{\inf,K}\to\calO_C$ and $\rA_2 = \rA_{\inf,K}/\rI_K^2$. Then there exists an element $\rho_K\in\calO_C$ with valuation $\nu_p(\rho_K) = \nu_p(\calD_K)+\frac{1}{p-1}$ such that $I_K\rA_2\cong \frac{1}{\rho_K}\calO_C(1)$. When $\calO_K = \rW(\kappa)$, one can choose $\rho_K = \zeta_p-1$. When contexts are clear, we also denote $\rho_K$ by $\rho$.

  By a {\bf smooth formal scheme} $\frakX$ over $\calO_C$, we mean a separated $p$-adic formally smooth formal scheme locally of topologically finite type over $\calO_C$. Such an $\frakX$ is called {\bf liftable}, if there is a smooth formal scheme $\widetilde \frakX$ over $\Spf(\rA_2)$ (viewed as a $p$-adic formal scheme) which is a lifting of $\frakX$ along $\theta:\rA_2\to\calO_C$. In what follows, by a lifting of $\frakX$, we always mean a smooth lifting as above. Let $\nu:X_{\proet}\to\frakX_{\et}$ be the natural projection from the pro-\'etale site of $X$ to the \'etale site of $\frakX$. By abuse of notations, for any sheaf $\calF$ on $\frakX_{\et}$, we also denote its pull-back $\nu^{-1}\calF$ on $X_{\proet}$ by $\calF$.

  For any ring $A$ with an element $a$ admitting pd-powers, we denote by $a^{[n]}$ the $n$-th pd-power of $a$ (e.g. for $\bZ$-flat $A$, $a^{[n]} = \frac{a^n}{n!}$ as an element in $A\otimes_{\bZ}\bQ$.)
  For any $\underline n = (n_1,\dots,n_d)\in\bN^d$, we define $|\underline n| = \sum_{i=1}^dn_i$. Assume $a_1,\dots,a_d\in A$ which admits pd-powers. For any $\underline n = (n_1,\dots,n_d)\in\bN^d$, we often denote $\prod_{i=1}^da_i^{[n_i]}$ by $\underline a^{[\underline n]}$ for simplicity. For any $1\leq i\leq d$, let $\underline 1_i=(0,\dots,0,1,0\dots,0)\in\bN^d$ be the generator of $i$-th component of $\bN^d$.

\subsection{Acknowledgement}
  The authors want to thank Johannes  Ansch\"utz, Ben Heuer and Arthur-C\'esar Le Bras for their interest and  valuable comments on a very early draft of this paper. We also thank Ruochuan Liu for his interest. The work was done when the second author was a postdoc at Morningside Center of Mathematics, and he would like to thank the institute for the great research condition there. The authors would also like to thank
 the referees for their careful reading and valuable comments.
  
  Y.M. has received funding from the European Research Council (ERC) under the European Union's Horizon 2020 research and innovation programme (grant agreement No. 884596). Y.W. is partially supported by CAS Project for Young Scientists in Basic Research, Grant No. YSBR-032.

\section{The period sheaf $\calO\widehat \bC_{\pd}^+$}\label{Sec-PeriodSheaf}
  In this section, let $\frakX$ be a liftable smooth formal scheme of dimension $d$ over $\calO_C$ with a fixed lifting $\widetilde \frakX$ over $\rA_2$. Let $X$ be the rigid analytic generic fiber of $\frakX$. We are going to construct the desired period sheaf $(\calO \widehat \bC_{\pd}^+,\Theta)$.% We are going to construct a sheaf $\calO\widehat \bC_{\pd}^+$ of $\OXp$-algebras (together with a Higgs field) on $X_{\proet}$ and define $\OC_{\pd}$ as $\calO\widehat \bC_{\pd}^+[\frac{1}{p}]$. Throughout this paper, we always put $\rho = \rho_K$.
\subsection{The construction of $\calO\widehat \bC_{\pd}^+$}
  The key ingredient of constructing $\calO\widehat \bC_{\pd}^+$ is the following well-known result, of which (the idea for) the proof is due to Quillen and Illusie:
  \begin{lem}\label{Lem-SZ}
      Let $A$ be a commutative ring. For any short exact sequence of flat $A$-modules
      \[0\to E\xrightarrow{u} F\xrightarrow{v} G\to 0\]
      and any $n\geq 0$, there exists an exact sequence of $A$-modules:
      \begin{equation}\label{ES-SZ-I}
          0\to\Gamma^n(E)\to\Gamma^n(F)\xrightarrow{\partial}\Gamma^{n-1}(F)\otimes G\xrightarrow{\partial}\cdots\xrightarrow{\partial}\Gamma^{n-i}(F)\otimes\wedge^iG\xrightarrow{\partial}\cdots \xrightarrow{\partial}\wedge^nG\to 0,
      \end{equation}
      where the differentials $\partial$ are induced by sending each 
      \[f_1^{[m_1]}\cdots f_r^{[m_r]}\otimes \omega\in\Gamma^m(F)\otimes\wedge^lG\]
      with $f_i\in F$, $m_i\geq 1$ satisfying $m_1+\cdots+m_r = m$ and $\omega\in\wedge^lG$ to 
      \[\sum_{i=1}^rf_1^{[m_1]}\cdots f_i^{[m_i-1]}\cdots f_r^{[m_r]}\otimes v(f_i)\wedge\omega\in\Gamma^{m-1}(F)\otimes\wedge^{l+1}G.
      \]
      Moreover, there exists an exact sequence
      \begin{equation}\label{ES-SZ-II}
          0\to\Gamma(E)\to\Gamma(F)\xrightarrow{\partial}\Gamma(F)\otimes G\xrightarrow{\partial}\Gamma(F)\otimes\wedge^2G\xrightarrow{\partial}\cdots,
      \end{equation}
      where the differentials $\partial$ are all $\Gamma(E)$-linear.
  \end{lem}
  \begin{proof}
      The exactness of (\ref{ES-SZ-I}) follows from \cite[Lem. A.28]{SZ}. By taking the direct sum of all $n\geq 0$, we know (\ref{ES-SZ-II}) is exact and the $\Gamma(E)$-linearity follows from the definition of $\partial$'s.
  \end{proof}
  We want to apply the above lemma to the integral Faltings' extension
  \begin{equation}\label{ES-FaltingsExt}
      0\to\OXp\to\calE_{\rho}^+\to\rho\OXp\otimes_{\calO_{\frakX}}\widehat \Omega_{\frakX}^1(-1)\to 0,
  \end{equation}
  which was introduced in \cite{Wang}. Let us give a quick review of the definition of $\calE_{\rho}^+$ as follows.
  
  Let $\widehat \rL_{\OXp/\calO_{\widetilde \frakX}}$ be the $p$-complete cotangent complex induced by the morphism of sheaves of rings $ \calO_{\widetilde\frakX}\to\OXp$ on $X_{\proet}$. By \cite[Thm. 2.9]{Wang}, there exists a short exact sequence of $\OXp$-modules
  \[0\to\frac{1}{\rho_K}\OXp(1)\to\rH^{-1}(\widehat \rL_{\OXp/\calO_{\widetilde \frakX}})\to\OXp\otimes_{\calO_{\frakX}}\widehat \Omega_{\frakX}^1\to 0.\]
  Then we define 
  \[\calE_{\rho}^+ = \rho_K\rH^{-1}(\widehat \rL_{\OXp/\calO_{\widetilde \frakX}})(1),\]
  and it fits into the exact sequence \ref{ES-FaltingsExt}. By \cite[Prop. 2.6]{Wang}, we know that $\calE_{\rho}^+$ is a locally finite free $\OXp$-module such that locally on $X_{\proet}$, $\calE_{\rho}^+ \cong \OXp\oplus\OXp\otimes_{\calO_{\frakX}}\rho\widehat \Omega_{\frakX}^1(-1)$.

  Now apply Lemma \ref{Lem-SZ} to the short exact sequence (\ref{ES-FaltingsExt}) of locally finite free $\OXp$-modules. We get an exact sequence
  \begin{equation}\label{ES-Preparation}
      0\to\Gamma(\OXp)\to\Gamma(\calE_{\rho}^+)\xrightarrow{\partial}\Gamma(\calE_{\rho}^+)\otimes_{\calO_{\frakX}}\rho\widehat \Omega^1_{\frakX}(-1)\xrightarrow{\partial}\cdots\xrightarrow{\partial}\Gamma(\calE_{\rho}^+)\otimes_{\calO_{\frakX}}\rho^d\widehat \Omega^d_{\frakX}(-d)\to 0.
  \end{equation}
  
  Let $e$ be the basis $1$ of $\OXp$ as a finite free $\OXp$-module. Then $\Gamma(\OXp) = \OXp[e]_{\pd}$ is the free pd-algebra over $\OXp$ generated by $e$. Noting that for any $n\geq 1$, $(n-1)\nu_p(\zeta_p-1)\geq \nu_p(n!)$, we know that $e-(\zeta_p-1)$ admits $n$-th divided powers in $\Gamma(\OXp)$ for any $n\geq 0$. Denote by $\calI_{\pd}$ the pd-ideal of $\Gamma(\OXp)$ principally generated by $e-(\zeta_p-1)$ and then we have $\OXp\cong \Gamma(\OXp)/\calI_{\pd}$.

  \begin{dfn}\label{Dfn-PeriodSheaf}
      \begin{enumerate}
          \item[(1)] Define $\calO\bC_{\pd}^+:=\Gamma(\calE_{\rho}^+)\otimes_{\Gamma(\OXp)}\OXp$, where we regard $\OXp$ as a $\Gamma(\OXp)$-algebra via the isomorphism $\Gamma(\OXp)/\calI_{\pd}\cong\OXp$.

          \item[(2)] Define $\calO\widehat \bC_{\pd}^{+} = \varprojlim_n\calO\bC_{\pd}^+/p^n$ as the $p$-adic completion of $\calO\bC_{\pd}^+$ and define $\calO\widehat \bC_{\pd} = \calO\widehat \bC_{\pd}^{+}[\frac{1}{p}]$. 

          \item[(3)] For any $\calA\in\{\calO\bC_{\pd}^{+},\calO\widehat \bC_{\pd}^{+}\}$, denote by $\Theta:\calA\to\calA\otimes_{\calO_{\frakX}}\rho\widehat \Omega^1_{\frakX}(-1)$ the $\OXp$-linear morphism induced by $\partial$ in (\ref{ES-Preparation}).
      \end{enumerate}    
  \end{dfn}

  \begin{rmk}\label{motivation-construction}
  The $p$-complete PD-polynomial rings naturally appear in the prismatic theory, which inspires our construction of $\calO\widehat \bC_{\pd}^{+}$. Let us briefly explain the motivation behind this construction. 
     In fact, the close relationship between prismatic theory and $p$-adic Simpson correspondence was discovered in \cite{MT} and \cite{Tia23}, which show that crystals on the prismatic site can be locally linked to Higgs bundles, similar to the crystals on Tsuji's Higgs site. In \cite{MW22}, we established a global correspondence of rational coefficients on smooth $p$-adic formal schemes over the ring of integers of a $p$-adic field. The key ingredient of \cite{MW22} is the global pro-\'etale period sheaf $\calO\bC$. Similar to the Higgs envelope considered by Tsuji in \cite{AGT}, the prismatic envelope can also produce ``local period rings" (to justify this, see \cite[Proposition 3.17]{MT}), which however are hard to be glued together. In some sense, the global period sheaf $\calO\bC$ forces these ``local period rings" to be glued together in a unique way. But this only works in the arithmetic case (i.e. over $p$-adic fields), where the Higgs fields involved are all nilpotent. In the geometric case (i.e. over $C$), there seems no possible way to glue prismatic ``local period rings" due to the convergence radius of $p$-adic exponential function. Despite this difficulty, the prismatic theory still has the advantage of yielding ``integral" period rings. Based on the $p$-complete PD-polynomial ring structure of the (Hodge--Tate) prismatic envelope showed in \cite{Tia23}, \cite{MW22}, one might try to modify the construction in \cite{Wang} by taking into account the convergence radius. This leads us to the above definition of $\calO\widehat\bC_{\pd}^+$, which turns out to be good enough to deal with small coefficients.
   
  \end{rmk}

  \begin{prop}\label{Prop-PeriodSheaf}
      For any $\calA\in\{\calO\bC_{\pd}^{+},\calO\widehat \bC_{\pd}^{+}\}$, the following sequence is exact:
      \begin{equation}\label{ES-PeriodSheaf}
          0\to\OXp\to\calA\xrightarrow{\Theta}\calA\otimes_{\calO_{\frakX}}\rho\widehat \Omega^1_{\frakX}(-1)\xrightarrow{\Theta}\cdots\xrightarrow{\Theta}\calA\otimes_{\calO_{\frakX}}\rho^d\widehat \Omega^d_{\frakX}(-d)\to 0.
      \end{equation}
      In particular, $\Theta$ defines a Higgs field on $\calA$.
  \end{prop}
  \begin{proof}
      As a $\Gamma(\OXp)$-algebra, $\Gamma(\calE_{\rho}^+)$ is (pro-\'etale) locally isomorphic to the free pd-algebra over $\Gamma(\OXp)$ on the free $\Gamma(\OXp)$-module $\Gamma(\OXp)\otimes_{\calO_{\frakX}}\rho\widehat \Omega^1_{\frakX}(-1)$ since locally on $X_{\proet}$, we have $\calE_{\rho}^+ \cong \OXp\oplus\rho\OXp\otimes_{\calO_{\frakX}}\widehat \Omega_{\frakX}^1(-1)$. As a consequence, for any $\Gamma(\OXp)$-module $\calM$, we get an exact sequence of $\Gamma(\OXp)$-modules
  \begin{equation}\label{Equ-PeriodSheaf}
      0\to\calM\to\Gamma(\calE_{\rho}^+)\otimes_{\Gamma(\OXp)}\calM\xrightarrow{\id_{\calM}\otimes\partial}\cdots\xrightarrow{\id_{\calM}\otimes\partial}\Gamma(\calE_{\rho}^+)\otimes_{\Gamma(\OXp)}\calM\otimes_{\calO_{\frakX}}\rho^d\widehat \Omega^d_{\frakX}(-d)\to 0
  \end{equation}  
  by applying $(-)\otimes^{\rL}_{\Gamma(\OXp)}\calM$ to the exact sequence \ref{ES-Preparation}.
      By letting $\calM = \OXp$ above, we conclude the result for $\calA = \calO\bC_{\pd}^{+}$. Moreover, we get that $\calO\bC_{\pd}^+$ is locally isomorphic to the free pd-algebra over $\OXp$ on the free $\OXp$-module $\OXp\otimes_{\calO_{\frakX}}\rho\widehat \Omega^1_{\frakX}(-1)$. In particular, it is $p$-torsion free. 

      By applying $(-)\otimes^{\rL}_{\Zp}\Zp/p^n$ to (\ref{Equ-PeriodSheaf}) for $\calA = \calO\bC_{\pd}^+$ and then taking $\rR\varprojlim_n$, we know the result holds true for $\calA = \calO\widehat \bC_{\pd}^{+}$.
  \end{proof}

\subsection{Local description of $\calO\widehat \bC_{\pd}^+$}
  \begin{convention}\label{Convention-small}
      An affine formal scheme $\frakU = \Spf(R^+)$ over $\calO_C$ of dimension $d$ is called {\bf small} if there is an \'etale morphism $
      \Box:\calO_C\za T_1^{\pm 1},\dots,T_d^{\pm 1}\ya\to R^+$. Such an \'etale morphism $\Box$ is called a chart on $\frakU$. In this case, we can deduce from the smoothness of $R^+$ that, up to isomorphisms, there is a unique $\rA_2$-lifting $\widetilde \frakU= \Spf(\widetilde R^+)$ of $\frakU$. By the \'etaleness of $\Box$, there exists a unique $\rA_2$-morphism $\rA_2\za T_1^{\pm 1},\dots,T_d^{\pm 1}\ya\to\widetilde R^+$ lifting $\Box$. 
      
      Let $U = \Spa(R,R^+)$ be the rigid analytic generic fiber of $\frakU$ and $U_{\infty} = \Spa(\widehat R_{\infty},\widehat R_{\infty}^+)$ be the base-change of $U$ along the morphism
      \[\Spa(C\za T_1^{\pm 1/p^{\infty}},\dots,T_d^{\pm 1/p^{\infty}}\ya,\calO_C\za T_1^{\pm 1/p^{\infty}},\dots,T_d^{\pm 1/p^{\infty}}\ya)\to \Spa(C\za T_1^{\pm 1},\dots,T_d^{\pm 1}\ya,\calO_C\za T_1^{\pm 1},\dots,T_d^{\pm 1}\ya).\]
      Then $U_{\infty}$ is a perfectoid space in $U_{\proet}$ such that $U_{\infty}\to U$ is a Galois cover with Galois group 
      \[\Gamma\cong \Zp\gamma_1\oplus\cdots\oplus\Zp\gamma_d,\]
      where for any $1\leq i,j\leq d$ and any $n\geq 1$, $\gamma_i$ is determined by sending $T_j^{1/p^n}$ to $\zeta_{p^n}^{\delta_{ij}}T_j^{1/p^n}$ and $\delta_{ij}$ denotes Kronecker's $\delta$-function. Moreover, $\widehat R_{\infty}^+$ admits a $\Gamma$-equivariant decomposition
      \begin{equation}\label{R-decomp}
      \widehat R_{\infty}^+\cong \widehat{\bigoplus_{\alpha_1,\dots,\alpha_d\in \bZ[1/p]\cap[0,1)}}R^+\cdot T_1^{\alpha_1}\cdots T_d^{\alpha_d}
      \end{equation}
      where ``$\widehat \oplus$'' denotes the $p$-adic topological direct sum.
  \end{convention}
  \begin{rmk}\label{Rmk-Small}
      It is clear that any smooth formal scheme is \'etale locally small affine. However, it can be even proved that any smooth formal scheme over $\calO_C$ is Zariski locally small affine (cf. \cite[Lem. 4.9]{Bha}).
  \end{rmk}

  In this subsection, we always assume $\frakX = \frakU = \Spf(R^+)$ is affine small as above. Let $E_{\rho}^+$ be the evaluation of $\calE_{\rho}^+$ at $U_{\infty}$. Then $E_{\rho}^+$ is endowed with a $\Gamma$-action and by \cite[Prop. 2.6]{Wang} fitting into the following short exact sequence of $\widehat R_{\infty}^+$-modules
  \[0\to\widehat R_{\infty}^+\to E_{\rho}^+\to \widehat R_{\infty}^+\otimes_{R^+}\rho\widehat \Omega^1_{R^+}(-1)\to 0.\]
  Comparing (\ref{ES-FaltingsExt}) with \cite[(2-3) and (2-6)]{Wang}, we see that $E_{\rho}^+ = \rho E^+(-1) = \frac{\rho}{t}E^+$ with $E^+$ appearing in \cite[Prop. 2.8]{Wang} (where we identify the Tate twist $\Zp(-1)$ with $\Zp\cdot t^{-1}$, cf. \S\ref{SSec-Notation}). By \cite[Prop. 2.6]{Wang}, $E^+$ admits an action of $\Gamma$ and fits into a $\Gamma$-equivariant short exact sequence
  \[0\to \rho^{-1}\widehat R_{\infty}^+(1)\to E^+\to\widehat R_{\infty}^+\otimes_{R^+}\widehat \Omega^1_{R^+}\to 0.\]
  Let $\frac{te}{\rho}\in E^+$ be the image of the morphism of $\widehat R_{\infty}^+$-modules
  $\xymatrix@C=0.5cm{\widehat R_{\infty}^+\ar[rrr]^{1\mapsto t\rho^{-1}}_{\cong}&&&}\rho^{-1}\widehat R_{\infty}^+(1)\hookrightarrow E^+$, and then $e \in E^+\cdot \rho t^{-1} = E_{\rho}^+$.
  \begin{lem}\label{Lem-LocalFaltings}
      There exists $\rho y_1,\dots,\rho y_d$ in $E_{\rho}^+$ lifting $\frac{\rho}{t}\rd\log T_1,\dots\frac{\rho}{t}\rd\log T_d$ via the projection
      \[E_{\rho}^+\to \widehat R_{\infty}^+\otimes_{R^+}\rho\widehat \Omega^1_{R^+}(-1)\]
      such that as an $\widehat R_{\infty}^+$-module, $E_{\rho}^+\cong \widehat R_{\infty}^+e\oplus(\bigoplus_{i=1}^d\widehat R_{\infty}^+\rho y_i)$ and for any $1\leq i,j\leq d$, $\gamma_i(\rho y_j) = \rho y_j+\rho \delta_{ij}e$.
  \end{lem}
  \begin{proof}
      Let $E^+$ be as in \cite[Prop. 2.8]{Wang} and $e$ be as above. Then by \cite[Prop. 2.8]{Wang}, $E^+$ admits a basis $\frac{t}{\rho}e,x_1,\dots,x_d$ such that $x_i$ lifts $\dlog T_i$ via the projection $E^+\to\widehat R_{\infty}^+\otimes_{R^+}\widehat \Omega^1_{R^+}$ such that for any $1\leq i,j\leq d$, we have
      \[\gamma_i(x_j) = x_j+\delta_{ij}te.\]
      As $E_{\rho}^+ = E^+\cdot \rho t^{-1}$, one can check $\rho y_i = \frac{\rho}{t}x_i$ satisfies the desired conditions.
  \end{proof}
  \begin{cor}\label{Cor-LocalPeriodSheaf}
      There is an isomorphism of $\OXp$-algebras
      \[\iota:\OXp[\rho Y_1,\dots,\rho Y_d]_{\pd}\big|_{U_{\infty}}\to \calO\bC_{\pd}^+\big|_{U_{\infty}}\]
      by identifying $\rho Y_i$'s with the images of $\rho y_i$'s via the composition $\calE_{\rho}^+\to\Gamma(\calE_{\rho}^+)\to\calO\bC_{\pd}^+$, which induces isomorphisms $\OXp[\rho Y_1,\dots,\rho Y_d]^{\wedge}_{\pd}\big|_{U_{\infty}}\to \calO\widehat \bC_{\pd}^+\big|_{U_{\infty}}$ and $\OXp[\rho Y_1,\dots,\rho Y_d]^{\wedge}_{\pd}[\frac{1}{p}]\big|_{U_{\infty}}\to \calO\widehat \bC_{\pd}\big|_{U_{\infty}}$. Via these isomorphisms, the Higgs field $\Theta$ on $\calP\in\{\calO\bC_{\pd}^+\big|_{U_{\infty}},\calO\widehat \bC_{\pd}^+\big|_{U_{\infty}},\calO\widehat \bC_{\pd}\big|_{U_{\infty}}\}$ is given by 
      \[\Theta = \sum_{i=1}^d\frac{\partial }{\partial Y_i}\otimes\frac{\rd\log T_i}{t}:\calP\to\calP\otimes\rho
      \widehat \Omega^1_{R^+}(-1) = \bigoplus_{i=1}^d\calP\cdot \rho\frac{\dlog T_i}{t}\]
      by identifying $\widehat\Omega^1_{R^+}(-1)$ with $\bigoplus_{i=1}^dR^+\cdot\frac{\dlog T_i}{t}$.
  \end{cor}
  \begin{proof}
      It is enough to prove the result for $\calO\bC_{\pd}^+$ while the rests follow from the constructions of $\calO\widehat \bC_{\pd}^+$ and $\calO\widehat \bC_{\pd}$. By \cite[Lem. 3.14]{BMS18}, for any affinoid perfectoid $V = \Spa(S,S^+)\in U_{\proet}/U_{\infty}$, we have 
      \[\widehat \rL_{S^+/\widetilde R^+}\cong \widehat \rL_{\widehat R_{\infty}^+/\widetilde R^+}\otimes^{\bL}_{\widehat R_{\infty}^+}S^+.\]
      By the definition of $\calE_{\rho}^+$, we get an isomorphism of $S^+$-modules
      \[\calE_{\rho}^+(V) = E_{\rho}^+\otimes_{\widehat R_{\infty}^+}S^+.\]
      By the construction of $\calO\bC_{\pd}^+$, we know that $\calO\bC_{\pd}^+(V) \cong \calO\bC_{\pd}^+(U_{\infty})\otimes_{\widehat R_{\infty}^+}S^+$. So we are reduced to the case to show that 
      \[\widehat R_{\infty}^+[\rho Y_1,\dots,\rho Y_d]_{\pd}\cong \calO\bC_{\pd}^+(U_{\infty}).\]
      By Lemma \ref{Lem-LocalFaltings}, $E_{\rho}^+$ is a free $\widehat R_{\infty}^+$-module with a basis $e,\rho y_1,\dots,\rho y_d$ fitting into the short exact sequence
      \[0\to\widehat R_{\infty}^+e\to E_{\rho}^+\to \widehat R_{\infty}^+\otimes_{R^+}\rho\widehat \Omega^1_{R^+}(-1)\to 0.\]
      Using (\ref{ES-SZ-II}), this yields a canonical exact sequence
      \[0\to\widehat R_{\infty}^+[e]_{\pd}\to\widehat R_{\infty}^+[e,\rho y_1,\dots,\rho y_d]_{\pd}\xrightarrow{\partial}\widehat R_{\infty}^+[e,\rho y_1,\dots,\rho y_d]_{\pd}\otimes_{R^+} \rho\widehat \Omega^1_{R^+}(-1)\to\cdots\]
      with the $\widehat R_{\infty}^+[e]_{\pd}$-linear $\partial$ given by 
      \[\partial(\prod_{i=1}^d(\rho y_i)^{[n_i]}) = \sum_{j=1}^d(\rho y_1)^{[n_1]}\cdots(\rho y_{j-1})^{[n_{j-1}]} (\rho y_j)^{[n_j-1]}(\rho y_{j+1})^{[n_{j+1}]}\cdots(\rho y_d)^{[n_d]}\otimes\frac{\rho}{t}\dlog T_j.\]
      Here, we use $\rho y_i$ is the lifting of $\frac{\rho}{t}\dlog T_i$, by Lemma \ref{Lem-LocalFaltings}) again. Modulo the pd-ideal generated by $(e-(\zeta_p-1))$ and denote by $\rho Y_i$ the image of $\rho y_i$, we then obtain the desired isomorphism
      \[\widehat R_{\infty}^+[\rho Y_1,\dots,\rho Y_d]_{\pd}\cong \calO\bC_{\pd}^+(U_{\infty})\]
      such that the $\partial$ is induced by 
      \[\sum_{i=1}^d\frac{\partial}{\partial (\rho Y_i)}\otimes\frac{\rho}{t}\dlog T_i = \frac{\partial}{\partial Y_i}\otimes\frac{\dlog T_i}{t}\]
      as desired. This completes the proof.
  \end{proof}
  \begin{notation}\label{Notation-LocalPeriodSheaf}
      Let $S_{\pd}^+$ (resp. $\widehat S_{\pd}^+$, $\widehat S_{\pd}$) be the evaluation of $\calO\bC_{\pd}^{+}$ (resp. $\calO\widehat \bC_{\pd}^{+}$, $\calO\widehat \bC_{\pd}$) at $U_{\infty}$. Then we have a $\Gamma$-equivariant isomorphism
      \[\widehat R_{\infty}^+[\rho Y_1,\dots,\rho Y_d]_{\pd}\cong S_{\pd}^+~({\rm resp.}~\widehat R_{\infty}^+[\rho Y_1,\dots,\rho Y_d]_{\pd}^{\wedge}\cong \widehat S_{\pd}^+,~\widehat R_{\infty}^+[\rho Y_1,\dots,\rho Y_d]_{\pd}^{\wedge}[\frac{1}{p}]\cong \widehat S_{\pd}).\]
      By Lemma \ref{Lem-LocalFaltings}, we know that the $\Gamma$-action on $S_{\pd}^+$ (resp. $\widehat S_{\pd}^+$, $\widehat S_{\pd}^+$) is determined by $\gamma_i(Y_j) = Y_j+\delta_{ij}(\zeta_p-1)$ for any $1\leq i,j\leq d$. By Corollary \ref{Cor-LocalPeriodSheaf}, the Higgs field on $S_{\pd}^+$ (resp. $\widehat S_{\pd}^+$, $\widehat S_{\pd}^+$) is given by $\Theta = \sum_{i=1}^d\frac{\partial}{\partial Y_i}\otimes\frac{\dlog T_i}{t}$.
  \end{notation}
\subsection{Comparison with $\calO\bC^{\dagger}$}

  We want to compare our period sheaves $\calO\widehat \bC_{\pd}^+$ and $\calO\widehat \bC_{\pd}$ with period sheaves $\calO\widehat \bC_{\rho}^+$ and $\calO\bC^{\dagger}$ introduced in \cite[\S 2.3]{Wang}. Let us first review some constructions in \textit{loc.cit.}.

  For any $\rho^{\prime}\in \calO_C$ with $\nu_p(\rho^{\prime})\geq \nu_p(\rho)$, let $\calE_{\rho^{\prime}}^+$ be the pull-back of $\calE_{\rho}^+$ along the natural inclusion $\OXp\otimes_{\calO_{\frakX}}\rho^{\prime}\widehat \Omega^1_{\frakX}(-1)\hookrightarrow\OXp\otimes_{\calO_{\frakX}}\rho\widehat \Omega^1_{\frakX}(-1)$. Then $\calE_{\rho^{\prime}}^+$ fits into an analogue of the short exact sequence (\ref{ES-FaltingsExt}) by replacing $\rho$ by $\rho^{\prime}$ there. For any such a $\rho^{\prime}$, we define $\calO\bC_{\rho^{\prime}}^+ = \varinjlim_n\Sym^n\calE_{\rho^{\prime}}^+$ and $\calO\widehat \bC_{\rho^{\prime}}^+ = \varprojlim_n\calO\bC_{\rho^{\prime}}^+/p^n$. By construction, for any $\rho_1,\rho_2\in\rho\calO_C$ with $\nu_p(\rho_1)\geq \nu_p(\rho_2)$, $\calO\bC_{\rho_1}^+$ (resp, $\calO\widehat \bC_{\rho_1}^+$) is a sub-sheaf of  $\calO\bC_{\rho_2}^+$ (resp, $\calO\widehat \bC_{\rho_2}^+$). Define $\calO\bC^{\dagger,+} = \varinjlim_{\rho^{\prime}\in\rho\frakm_C}\calO\widehat \bC_{\rho^{\prime}}^+$ and $\calO\bC^{\dagger} = \calO\bC^{\dagger,+}[\frac{1}{p}]$. Moreover, there are Higgs fields $\widetilde \Theta$ on $\calO\bC_{\rho^{\prime}}^+$, $\calO\widehat \bC_{\rho^{\prime}}^+$ and $\calO\bC^{\dagger,+}$ which are compatible with each other. 

  Assume $\frakU = \Spf(R^+)$ is affine small and keep notations in Notation \ref{Notation-LocalPeriodSheaf}. For any $\rho^{\prime}\in\rho\calO_C$, let $E_{\rho^{\prime}}^+$ be the evaluation of $\calE_{\rho^{\prime}}^+$ at $U_{\infty}$. Let $e$ and $\rho y_i$'s be as in Lemma \ref{Lem-LocalFaltings} and then we know that $E_{\rho^{\prime}}^+ \cong \widehat R_{\infty}^+e\oplus(\bigoplus_{i=1}^d\widehat R_{\infty}^+\rho^{\prime}y_i)$. 
  As Corollary \ref{Cor-LocalPeriodSheaf}, we have a local description of these period sheaves.
  \begin{lem}\label{Lem-LocalPeriodSheafWang}
      Assume $\frakU = \Spf(R^+)$ is affine small and keep notations in Convention \ref{Convention-small}.
      For any $\rho^{\prime}\in\rho\calO_C$, there is an isomorphism of $\OXp$-algebras
      \[\iota:\OXp[\rho^{\prime}Y_1,\dots,\rho^{\prime}Y_d]\big|_{U_{\infty}}\to\calO\bC_{\rho^{\prime}}^+\big|_{U_{\infty}}\]
      by identifying $\rho^{\prime} Y_i$'s with the images of $\rho^{\prime} y_i$'s via the map $\calE_{\rho^{\prime}}^+\to \calO\bC_{\rho^{\prime}}^+$, which induces isomorphisms $\OXp[\rho^{\prime}Y_1,\dots,\rho^{\prime}Y_d]^{\wedge}\big|_{U_{\infty}}\cong\calO\widehat \bC_{\rho^{\prime}}^+\big|_{U_{\infty}}$ and $\varinjlim_{\rho^{\prime}\in\rho\frakm_C}\OXp[\rho^{\prime}Y_1,\dots,\rho^{\prime}Y_d]\big|_{U_{\infty}}\cong\calO\bC^{\dagger,+}\big|_{U_{\infty}}$.
      Via these isomorphisms, the Higgs field $\widetilde \Theta$ on $\calO\bC_{\rho^{\prime}}^+\big|_{U_{\infty}}$ (resp, $\calO\widehat \bC_{\rho^{\prime}}^+\big|_{U_{\infty}}$, $\calO\bC^{\dagger,+}\big|_{U_{\infty}}$) is given by $\widetilde \Theta = \sum_{i=1}^d\frac{\partial}{\partial Y_i}\otimes\frac{\rd\log T_i}{t}$.
  \end{lem}
  \begin{proof}
      This is essentially proved in \cite[Cor. 2.22, 2.23 and Thm. 2.28]{Wang} and can be checked by the same argument used in the proof of Corollary \ref{Cor-LocalPeriodSheaf}.
  \end{proof}
  In particular, for any $\rho^{\prime}\in\rho\frakm_C$, we have inclusions of period sheaves 
  \[\calO\widehat \bC_{\rho^{\prime}}^+\subset\calO\bC^{\dagger,+}\subset\calO\widehat \bC^+_{\rho}.\]
  
  Now, we are going to construct a canonical inclusion $\calO\widehat \bC_{\pd}^{+}\to \calO\widehat \bC_{\rho}^+$ which is compatible with Higgs fields such that the inclusion $\calO\bC^{\dagger}\subset\calO\widehat \bC_{\rho}^+[\frac{1}{p}]$ factors through $\calO\widehat \bC_{\pd}$.
  \begin{construction}\label{Construction-ComparePeriodSheaf}
      Consider the composition $f:\calE_{\rho}^+\xrightarrow{\times(\zeta_p-1)}(\zeta_p-1)\calE_{\rho}^+\to\calO\bC_{\rho}^+$ of $\OXp$-modules. As $(\zeta_p-1)$ admits pd-powers in $\calO\bC_{\rho}^+$, we know that $f$ induces a natural morphism of $\OXp$-algebras
      \[\Gamma(\calE_{\rho}^+)\to\calO\bC^+_{\rho}.\]
      Consider the sub-$\OXp$-algebra $\Gamma(\OXp)\cong \OXp[e]_{\pd}$. By the constructions of $\calO\bC_{\rho}^+$ and $f$, we know that $f(e) = \zeta_p-1$ and thus get a natural morphism 
      \[\calO\bC_{\pd}^+\to \calO\bC_{\rho}^+.\]
      Taking $p$-adic completions on both sides, we get a natural morphism of $\OXp$-algebras 
      \[\iota_{PS}:\calO\widehat \bC_{\pd}^+\to\calO\widehat \bC_{\rho}^+.\]
  \end{construction}
  \begin{prop}\label{Prop-ComparePeriodSheaf}
      \begin{enumerate}
          \item[(1)] The natural morphism $\iota_{PS}:\calO\widehat \bC_{\pd}^+\to\calO\widehat \bC_{\rho}^+$ constructed above is an inclusion such that the Higgs fields are compatible in the sense that $\widetilde \Theta = (\zeta_p-1)\Theta$. 

          \item[(2)] The inclusion $\calO\bC^{\dagger}\subset\calO\widehat \bC_{\rho}^+[\frac{1}{p}]$ factors through $\calO\widehat \bC_{\pd}$.

          \item[(3)] Assume $\frakU=\Spf(R^+)$ is affine small and keep notations as in Convention \ref{Convention-small}. Via isomorphisms in Corollary \ref{Cor-LocalPeriodSheaf} and Lemma \ref{Lem-LocalPeriodSheafWang}, we have the following commutative diagram
          \begin{equation}\label{Diag-ComparePeriodSheaf}
              \xymatrix@C=0.5cm{
                \calO\widehat \bC_{\pd}^+\big|_{U_{\infty}}\ar[d]^{\cong}\ar[rrrrrr]&&&&&&\calO\widehat \bC_{\rho}^+\big|_{U_{\infty}}\ar[d]^{\cong}\\
                \OXp[\rho Y_1,\dots,\rho Y_d]^{\wedge}_{\pd}\big|_{U_{\infty}}\ar[rrrrrr]^{(\rho Y_i)^{[n]}\mapsto (\zeta_p-1)^{[n]}(\rho Y_i)^n}&&&&&&\OXp[\rho Y_1,\dots,\rho Y_d]^{\wedge}\big|_{U_{\infty}}.
              }
          \end{equation}
      \end{enumerate}
  \end{prop}
  \begin{proof}
      We first prove Item (3). By the definition of $f$ in Construction \ref{Construction-ComparePeriodSheaf}, for any $\rho y_i\in E_{\rho}^+$ given in Lemma \ref{Lem-LocalFaltings}, we have $f(\rho y_i) = (\zeta_p-1)\rho Y_i$. Then we can conclude by using \ref{Cor-LocalPeriodSheaf}. By Corollary \ref{Cor-LocalPeriodSheaf} and Lemma \ref{Lem-LocalPeriodSheafWang}, for any $\underline n = (n_1,\dots,n_d)\in\bZ_{\geq 0}^d$, we have 
      \[\begin{split}
          &\iota_{PS}((\zeta_p-1)\Theta(\prod_{l=1}^d(\rho Y_l)^{[n_l]}))\\
          =&\iota_{PS}((\zeta_p-1)\sum_{i=1}^d(\rho Y_1)^{[n_1]}\cdots(\rho Y_i)^{[n_i-1]}\cdots(\rho Y_d)^{[n_d]}\frac{\rho\rd\log T_i}{t})\\
          =&(\zeta_p-1)\sum_{i=1}^d(\zeta_p-1)^{[n_1]}\cdots(\zeta_p-1)^{[n_i-1]}\cdots(\zeta_p-1)^{[n_d]}(\rho Y_1)^{n_1}\cdots(\rho Y_i)^{n_i-1}\cdots(\rho Y_d)^{n_d}\frac{\rho\rd\log T_i}{t}\\
          =&\sum_{i=1}^dn_i(\zeta_p-1)^{[n_1]}\cdots(\zeta_p-1)^{[n_d]}(\rho Y_1)^{n_1}\cdots(\rho Y_i)^{n_i-1}\cdots(\rho Y_d)^{n_d}\frac{\rho\rd\log T_i}{t}\\
          =&\prod_{j=1}^d(\zeta_p-1)^{[n_j]}\widetilde \Theta(\prod_{l=1}^d(\rho Y_l)^{n_l})\\
          =&\widetilde \Theta(\iota_{PS}(\prod_{l=1}^d(\rho Y_l)^{[n_l]})).
      \end{split}\]
      In other words, we have $\widetilde \Theta = (\zeta_p-1)\Theta$ in this case. On the other hand, as $\OXp$-modules, the diagram
      \[\xymatrix@C=0.5cm{
        \OXp[\rho Y_1,\dots,\rho Y_d]^{\wedge}_{\pd}\big|_{U_{\infty}}\ar[d]^{\cong}\ar[rrrrrr]^{(\rho Y_i)^{[n]}\mapsto (\zeta_p-1)^{[n]}(\rho Y_i)^n}&&&&&&\OXp[\rho Y_1,\dots,\rho Y_d]^{\wedge}\big|_{U_{\infty}}\ar[d]^{\cong}\\
        \widehat \bigoplus_{\underline n = (n_1,\dots,n_d)\in\bN^d}\OXp\big|_{U_{\infty}} e_{\underline n}\ar[rrrrrr]^{e_{\underline n}\mapsto \prod_{i=1}^d(\zeta_p-1)^{[n_i]}e_{\underline n}}&&&&&&\widehat \bigoplus_{\underline n = (n_1,\dots,n_d)\in\bN^d}\OXp\big|_{U_{\infty}} e_{\underline n}
      }\]
      commutes. So we know that $\iota_{PS}$ is injective in this case as $\OXp$ is $p$-torsion free.

      For Item (1): Since we already have a globally defined morphism $\iota_{PS}$, all results in statement can be check locally on both $\frakX_{\et}$ and $X_{\proet}$ and thus reduces to the case in Item (3) that we have proved above.

      To prove (2), by the construction of $\calO\bC^{\dagger}$, we are reduced to showing that for any $\rho^{\prime}\in\rho\frakm_C$, there exists an $N\geq 0$ such that the inclusion $p^N\calO\widehat \bC_{\rho^{\prime}}^+\subset \calO\widehat \bC_{\rho}^+$ factors through $\calO\widehat \bC_{\pd}^+$. Since the problem is local, we may assume $\frakX = \frakU$ is affine small as above. Thanks to Item (3), we are reduced to showing that there exists an $N\geq 0$ such that
      \[p^N\OXp[\rho^{\prime}Y_1,\dots,\rho^{\prime}Y_d]^{\wedge}\subset \OXp[\rho(\zeta_p-1)Y_1,\dots,\rho(\zeta_p-1)Y_d]^{\wedge}_{\pd}.\]
      Fix an $\underline n = (n_1,\dots,n_d)\in\bN^d$. We know that
      \[(\rho^{\prime}Y_1)^{n_1}\cdots (\rho^{\prime}Y_d)^{n_d} =(\frac{\rho^{\prime}}{\rho})^{n_1+\cdots+n_d} (\rho Y_1)^{n_1}\cdots (\rho Y_d)^{n_d}.\]
      Recall that for any $n = n_0+n_1p+\cdots+n_rp^r\geq 0$ with $n_1,\dots,n_r\in\{0,1\dots,p-1\}$ $n_r\neq 0$, 
      \[\nu_p(\frac{(\zeta_p-1)^n}{n!}) = \frac{n}{p-1}-\frac{n-\sum_{i=0}^rn_r}{p-1} = \frac{\sum_{i=0}^rn_r}{p-1}\leq (r+1)\leq \log_p(n)+1.\]
      In particular, there exists an $m\geq 0$ such that for any $n\geq m$,
      \[n\nu_p(\frac{\rho^{\prime}}{\rho})\geq \nu_p((\zeta_p-1)^{[n]}).\]
      Choose $N\geq 0$ such that for any $(n_1,\dots,n_d)\in\bN^d$ with $\max_{1\leq i\leq d}\{n_i\}\leq m$, 
      \[p^N\prod_{i=1}^d(\rho^{\prime}Y_i)^{n_i}\in \OXp[\rho(\zeta_p-1)Y_1,\dots,\rho(\zeta_p-1)Y_d]^{\wedge}_{\pd}.\]
      As argued as above, for such an $N$, we have 
      \[p^N\OXp[\rho^{\prime}Y_1,\dots,\rho^{\prime}Y_d]^{\wedge}\subset \OXp[\rho(\zeta_p-1)Y_1,\dots,\rho(\zeta_p-1)Y_d]^{\wedge}_{\pd}\]
      as desired. This completes the proof.
  \end{proof}
  \begin{rmk}[$\calO\bC^{\dagger,+}$ vs $\calO\widehat \bC^+_{\pd}$ vs $\calO\widehat \bC_{\rho}^+$]\label{Rmk-ComparePeriodSheaf}
      \begin{enumerate}
          \item[(1)] Although $(\calO\widehat \bC_{\rho}^+)^{\widetilde \Theta} = \OXp$, even after inverting $p$, the induced Higgs complex ${\rm HIG}(\calO\widehat \bC_{\rho}^+[\frac{1}{p}],\widetilde \Theta)$ is not a resolution of $\OX$ (cf. \cite[Rem. 2.24 (1)]{Wang}).

          \item[(2)] Although $(\calO\bC^{\dagger,+})^{\widetilde \Theta} = \OXp$ and after inverting $p$, the induced Higgs complex ${\rm HIG}(\calO\bC^{\dagger},\widetilde \Theta)$ is a resolution of $\OX$, ${\rm HIG}(\calO\bC^{\dagger,+},\widetilde \Theta)$ is never a resolution of $\OXp$. Indeed, the $p^{\infty}$-torsion part of $\rH^{\geq 1}({\rm HIG}(\calO\bC^{\dagger,+},\widetilde \Theta))$ is unbounded (cf. \cite[Rem. 2.24, Thm. 2.28]{Wang}).

          \item[(3)] The Higgs complex ${\rm HIG}(\calO\widehat \bC^+_{\pd},\Theta)$ is a resolution of $\OXp$.
      \end{enumerate}
  \end{rmk}
  \begin{rmk}
      In \cite[Section 8]{AHLB}, the authors also compared $\calO\widehat\bC_{\pd}$ (which is denoted by $\calB_{\tilde{\frakX}}$ in \textit{loc.cit.}) with $\calO\bC^{\dagger}$ independently.
  \end{rmk}

\section{$\Gamma$-cohomology of log nilpotent representations}\label{Sec-GammaCohomology}
In this section, we focus on the local calculations of some group cohomologies. The global arguments in Section \ref{Sec-MainResult} will reduce to these local calculations. 
For convenience and later applications, we work in a general setting as follows.
  \begin{notation}\label{notation:set-up of A}
    Let $A$ be a $p$-complete $p$-torsion free $\calO_C$-algebra and $\rho\in\calO_C$ with $\nu_p(\rho)\geq\frac{1}{p-1}$. Let $A[\rho X]_{\pd}^{\wedge}$ be the $p$-complete free pd algebra over $A$ generated by $\rho X$. By abuse of notation, we temporarily put $\Gamma\cong \Zp\gamma$ (only in this section). Let $\Gamma$ act on $A[\rho X]_{\pd}^{\wedge}$ by $\gamma(\rho X) = \rho X+\rho e$ for some topologically nilpotent non-zero divisor $e\in A$ such that $A$ is $e$-complete.
  \end{notation}
  We are interested in log-nilpotent $A$-representations of $\Gamma$ defined below.
 \begin{dfn}\label{Dfn-LogNilRep}
    By an {\bf $A$-representation of $\Gamma$ of rank $r$}, we mean a finite free $A$-module $V$ of rank $r$ together with a continuous $A$-linear action of $\Gamma$. An $A$-representation $V$ is called {\bf log nilpotent}, if for a chosen $A$-basis $v_1,\dots, v_r$ of $V$, the matrix of $\gamma$ on $V$ is of the form 
    \[\exp(-(\zeta_p-1)\Theta) = \sum_{n\geq 0}(-1)^n(\zeta_p-1)^{[n]}\Theta^{n},\]
    where $\Theta\in\rM_r(A)$ is topologically nilpotent; that is, $\lim_{n\to+\infty}\Theta^{n} = 0$. Let $\Rep_A(\Gamma)$ (resp. $\Rep^{\nil}_A(\Gamma)$) be the category of $A$-representations (resp. log nilpotent $A$-representations) of $\Gamma$
  \end{dfn}

 The following example exhibits why  the log-nilpotent representations are interesting. Indeed, they appear naturally in the $p$-adic geometry.
  
  \begin{exam}
      We assume $A = \calO_C\za T^{\pm 1}\ya$ is the the ring of regular functions on the $1$-dimensional formal torus and denote by $(A)_{\Prism} = (A/\Ainf)_{\Prism}$ the prismatic site of $A$ (cf. \cite[Def. 4.1]{BS22}). In \cite[Thm. 5.12]{Tia23}, Tian proved that there exists an equivalence between the category of Hodge--Tate crystals (i.e. reduced crystal in \textit{loc.cit.}) and the category of topologically nilpotent Higgs bundles over $A$. Fix such a Higgs bundle $(H,\Theta)$ with $H$ finite free over $A$ and equip $H$ with a $\Gamma$-action by letting $\gamma = \exp(-(\zeta_p-1)\Theta)$. Then we get a log nilpotent $A$-representation $H$. Let $\bM$ be the Hodge--Tate crystal induced by $(H,\Theta)$ via \cite[Thm. 5.12]{Tia23}. Similar to \cite[Lem. 6.5]{MW22}, one can check that the evaluation of $\bM$ at $\Ainf(\calO_C\za T^{\pm \frac{1}{p^{\infty}}}\ya)$ induces a $\Gamma$-equivariant isomorphism
      $\bM(\Ainf(\calO_C\za T^{\pm \frac{1}{p^{\infty}}}\ya),\xi)\cong H\otimes_A\calO_C\za T^{\pm \frac{1}{p^{\infty}}}\ya$.
  \end{exam}

  Now, we state the main result of this section, which shall be used in next sections.
  \begin{prop}\label{Prop-CohomologySummary}
      Let $A$ be a $p$-complete $p$-torsion free $\calO_C$-algebra with the topologically nilpotent element $e = \zeta_p-1\in A$ and $\rho\in\calO_C$ satisfied $\nu_p(\rho)\geq\frac{1}{p-1}$.
      Let $V$ be a log nilpotent $A$-representation of $\Gamma$ with $\Theta\in \rM_r(A)$ as above. For any $\alpha\in\bZ[\frac{1}{p}]\cap[0,1)$, define 
      \[M_{\alpha}(V) = V\otimes_AA[\rho X]^{\wedge}_{\pd}\otimes_AAe_{\alpha},\]
      with the diagonal $\Gamma$-action, where $Ae_{\alpha}$ is the free $A$-module with the basis $e_{\alpha}$ on which $\Gamma$ acts $A$-linearly such that $\gamma(e_{\alpha}) = \zeta^{\alpha}e_{\alpha}$.
      \begin{enumerate}
          \item[(1)] Assume $\alpha\neq 0$. We have $\rH^0(\Gamma,M_{\alpha}(V)) = 0$ and $\rH^1(\Gamma,M_{\alpha}(V)) = M_{\alpha}(V)/(\zeta^{\alpha}-1)$.

          \item[(2)] Assume $\alpha = 0$. We have 
          \[\rH^0(\Gamma,M_{0}(V)) = \{\exp(\Theta X)v\mid v\in V_{\Theta/\rho}\}\cong V_{\Theta/\rho},\]
          where
          \[V_{\Theta/\rho} = \{v\in V\mid \forall~ n\geq 1, \Theta^nv\in\rho^nV~\&~\lim_{n\to+\infty}\rho^{-n}\Theta^nv = 0\}.\]
          If there exists $n\geq 1$ such that $\Theta^n = \rho^n\Theta^{\prime}$ for some topologically nilpotent $\Theta^{\prime}$, then we have
          \[\rH^1(\Gamma,M_0(V))\cong \rH^1(\Gamma,M_0(V))[\rho^{n}(\zeta_p-1)].\]
          If moreover $n=1$, then the natural inclusion $M_0(V)^{\Gamma}\subset M_0(V)$ induces a $\Gamma$-equivariant isomorphism 
          \[M_0(V)^{\Gamma}\otimes_AA[\rho X]^{\wedge}_{\pd} = M_0(V)\]
          such that 
          \[\rH^1(\Gamma,M_0(V)) \cong M_0(V)/\rho (\zeta_p-1)M_0(V).\]
      \end{enumerate}
  \end{prop}
  \begin{proof}
      This is a special case of Proposition \ref{Prop-CohomologySummary-general} we shall prove later. But we would like to give the rough idea here: The key point for the proof is to compute $\rR\Gamma(\Gamma,M_{\alpha}(V))$ via the complex
      \[M_{\alpha}(V)\xrightarrow{\gamma-1}M_{\alpha}(V).\]
      When $\alpha \neq 0$, we have $\zeta^{\alpha}-1$ divides $\zeta_p-1$. As $V$ is log-nilpotent, one can check the action of $\gamma-1$ on $M_{\alpha}(V)$ is ``controlled'' by $\zeta^{\alpha}-1$; that is, it is of the form $(\zeta^{\alpha}-1)\phi_{\alpha}$, where $\phi_{\alpha}:M_{\alpha}(V)\to M_{\alpha}(V)$ is a bijection. So the computation in this setting is quite easy. However, the $\alpha = 0$ case is much more complicated and the corresponding calculations form the most technical part in this section.
  \end{proof}

   Before we go to details, it is worth explaining why Proposition \ref{Prop-CohomologySummary} is useful (at least in a special case): 
   Note that when $R^+$ is small affine of relative dimension $1$ over $\calO_C$, the $\widehat S_{\pd}^+$ in Notation \ref{Notation-LocalPeriodSheaf} admits a $\Gamma$-equivariant ($p$-adically) topological decomposition
   \[\widehat S_{\pd}^+ = \widehat {\bigoplus_{\alpha\in \bN[1/p]\cap[0,1)}}R^+[\rho Y]^{\wedge}_{\pd}T^{\alpha}.\]
   Recall we have $\gamma(\rho Y) = \rho Y+\rho(\zeta_p-1)$ and $\gamma(T^{\alpha}) = \zeta^{\alpha}T^{\alpha}$. So for any log-nilpotent $R^+$-representation $V$ of $\Gamma$, we have (for $(A,X) = (R^+,Y)$ in Proposition \ref{Prop-CohomologySummary})
   \[V\otimes_{R^+}\widehat S_{\pd}^+\cong \widehat  {\bigoplus_{\alpha\in \bN[1/p]\cap[0,1)}} M_{\alpha}(V)\]
   Thus, Proposition \ref{Prop-CohomologySummary} essentially computes $\rR\Gamma(\Gamma,V\otimes_{R^+}\widehat S_{\pd}^+)$. See also Example \ref{exam:apply Section three} below.
  
\subsection{$\Gamma$-cohomology of pd polynomial ring}\label{SSec-Gamma-cohomology of PD-ring}
  From now on, let $A, \rho, e$ and $A[\rho X]^{\wedge}_{\pd}$ with the $\Gamma$-action be as in Notation \ref{notation:set-up of A}.
  
  \begin{notation}
     For any $\alpha\in\bZ[1/p]\cap[0,1)$, we define $M_{\alpha} = A[\rho X]^{\wedge}_{\pd}e_{\alpha}$ and let $\Gamma$ act on $M_{\alpha}$ by $\gamma(\rho X) = \rho X+\rho e$ and $\gamma(e_{\alpha}) = \zeta^{\alpha}e_{\alpha}$.
  \end{notation}

 \begin{exam}\label{exam:apply Section three}
     Keep notations as in Convention \ref{Convention-small} and Notation \ref{Notation-LocalPeriodSheaf}. Let $\Spa(A_0[\frac{1}{p}],A_0)\to U$ be the Galois cover corresponding to the subgroup $\oplus_{i=1}^{d-1}\Zp\gamma_i$ of $\Gamma$ and $A:=A_0[\rho Y_1,\cdots,\rho Y_{d-1}]^{\wedge}_{\rm pd}$. 
      Then as a $\gamma_d$-module, we have 
      \[\widehat S_{\pd}^+ = \widehat {\bigoplus_{\alpha\in\bN[\frac{1}{p}]\cap[0,1)}}A[\rho Y_d]^{\wedge}_{\pd}T_d^{\alpha}.\]
   Putting $X=Y_d$, $e_\alpha=T_d^{\alpha}$ and $e=\zeta_p-1$, this will be the main example interesting to us.
 \end{exam}

      Now as an $A$-module, 
      $M_{\alpha} = \widehat \bigoplus_{n\geq 0}A(\rho X)^{[n]}e_{\alpha}$ for any $\alpha\in \bN[\frac{1}{p}]\cap [0,1)$.
      So for any $n\geq 0$ and for any $N\in \Zp$, we have
      \begin{equation}\label{equ:gamma^N action}
          \gamma^N((\rho X)^{[n]}e_{\alpha}) = \zeta^{\alpha N}(\rho X+N\rho e)^{[n]}e_{\alpha} = \zeta^{\alpha N}\sum_{i=0}^n\rho^{[n-i]}N^{n-i}e^{n-i}(\rho X)^{[i]}e_{\alpha}.
      \end{equation}
      Note that this implies $\Gamma$ acts on $M_{\alpha}$ continuously: Indeed, for any $N\in \Zp$ with $N\alpha\in \bZ$, one can conclude from (\ref{equ:gamma^N action}) that $\gamma^N$ acts on $M_{\alpha}/NM_{\alpha}$ trivially.

      Now, by letting $N = 1$ in (\ref{equ:gamma^N action}), we get
      \[\gamma((\rho X)^{[n]}e_{\alpha}) = \zeta^{\alpha}(\rho X+\rho e)^{[n]}e_{\alpha} = \zeta^{\alpha}\sum_{i=0}^n\rho^{[n-i]}e^{n-i}(\rho X)^{[i]}e_{\alpha}.\]
  \begin{prop}\label{Prop-Gamma-cohomology of pd ring}
    Keep notations as above.
    \begin{enumerate}
        \item[(1)] When $\alpha \neq 0$, we have $\rH^0(\Gamma, M_{\alpha}) = 0$ and $\rH^1(\Gamma,M_{\alpha}) \cong M_{\alpha}/(\zeta^{\alpha}-1)M_{\alpha}$.

        \item[(2)] When $\alpha = 0$, the inclusion $A\to A[\rho X]_{\pd}^{\wedge}$ induces an isomorphism $\rH^0(\Gamma,A[\rho X]_{\pd}^{\wedge}) \cong A$ and $\rH^1(\Gamma,A[\rho X]^{\wedge}_{\pd}) = A[\rho X]^{\wedge}_{\pd}/\rho eA[\rho X]^{\wedge}_{\pd}$.
    \end{enumerate}
  \end{prop}
  \begin{proof}
    Note that $\rH^i(\Gamma,M_{\alpha})$ is calculated by the Koszul complex 
    \[M_{\alpha}\xrightarrow{\gamma-1}M_{\alpha}.\]
    So we have $\rH^0(\Gamma,M_{\alpha}) \cong (M_{\alpha})^{\gamma=1}$ and $\rH^1(\Gamma,M_{\alpha})\cong M_{\alpha}/(\gamma-1)M_{\alpha}$.
      
    For any $x = \sum_{n\geq 0}a_n(\rho X)^{[n]}e_{\alpha}\in M_{\alpha}$, we have that 
    \begin{equation}\label{Equ-(gamma-1)-A}
      \begin{split}
          \gamma(x) - x &= \zeta^{\alpha}\sum_{n\geq 0}a_n\sum_{i=0}^{n}\rho^{[n-i]}e^{n-i}\rho^iX^{[i]}e_{\alpha} - \sum_{n\geq 0}a_n(\rho X)^{[n]}e_{\alpha} \\
          &= \sum_{n\geq 0}(\zeta^{\alpha}\sum_{m\geq 1}a_{n+m}\rho^{[m]}e^m+(\zeta^{\alpha}-1)a_n)(\rho X)^{[n]}e_{\alpha}.
      \end{split}
    \end{equation}

    {\bf Case~1:} $\alpha\neq 0$. 

    We first show that $(M_{\alpha})^{\gamma=1} = 0$. Suppose that $x = \sum_{n\geq 0}a_n(\rho X)^{[n]}e_{\alpha}\in M_{\alpha}$ is fixed by $\gamma$. Then (\ref{Equ-(gamma-1)-A}) tells us that for any $n\geq 0$,
    \[\zeta^{\alpha}\sum_{m\geq 1}a_{n+m}\rho^{[m]}e^m+(\zeta^{\alpha}-1)a_n = 0.\]
    We claim that $a_{n} = 0$ for any $n\geq 0$. To see this, it is enough to show that $a_{n}\in e^mA$ for any $m\geq 1$. We will prove this by induction on $m\geq 1$.
    Since $\nu_p(\rho)\geq \frac{1}{p-1}$, for any $n\geq 1$, we have 
    \[\nu_p(\rho^{[n]})\geq \nu_p(\rho)\geq \frac{1}{p-1}\geq \nu_p(\zeta^{\alpha}-1).\] 
    Since $e$ is regular in $A$, we see that for any $n\geq 0$,
    \[a_{n} = -\zeta^{\alpha}\sum_{m\geq 1}a_{n+m}\frac{\rho^{[m]}}{\rho}\frac{\rho}{\zeta^{\alpha}-1}e^m.\]
    In particular, we have $a_{n}\in eA$ for any $n\geq 0$. Now, assume $a_{n}\in e^kA$ for some $k\geq 1$ and then we see that
    \[a_{n}\in \sum_{m\geq 1}\frac{\rho^{[m]}}{\rho}\frac{\rho}{\zeta^{\alpha}-1}e^{m+k}A\subset e^{k+1}A.\]
    This proves the claim and hence $\rH^0(\Gamma,M_{\alpha}) = 0$.

    To settle down the result in this case, we are reduced to showing that $(\zeta^{\alpha}-1)M_{\alpha}\subset (\gamma-1)M_{\alpha}$. By (\ref{Equ-(gamma-1)-A}), we have $(\gamma-1)M_{\alpha} \subset (\zeta^{\alpha}-1)M_{\alpha}$. On the other hand, for any $y = \sum_{n\geq 0}b_n(\rho X)^{[n]}e_{\alpha}\in M_{\alpha}$, we have
    \[(\gamma-1)y-(\zeta^{\alpha}-1)y = \sum_{n\geq 0}(\zeta^{\alpha}\sum_{m\geq 1}a_{n+m}\rho^{[m]}e^m)(\rho X)^{[n]}e_{\alpha}\subset \rho e M_{\alpha}.\]
    In particular, we have
    \[(\zeta^{\alpha}-1)M_{\alpha}\subset (\gamma-1)M_{\alpha}+\rho eM_{\alpha}\subset (\gamma-1)M_{\alpha}+ e(\zeta^{\alpha}-1) M_{\alpha}.\]
    Since $\nu_p(\rho)\geq \nu_p(\zeta^{\alpha}-1)$ and $e$ is topologically nilpotent, by Nakayama's Lemma, we conclude that $(\zeta^{\alpha}-1)M_{\alpha} = (\gamma-1)M_{\alpha}$ as desired.

    {\bf Case 2:} $\alpha = 0$. 
    
    The argument is similar to that in the above case. We first show that $(A[\rho X]_{\pd}^{\wedge})^{\gamma=1} = A$. Suppose that $x = \sum_{n\geq 0}a_n(\rho X)^{[n]}\in M_{0}$ is fixed by $\gamma$. Then by (\ref{Equ-(gamma-1)-A}), we have that for any $n\geq 0$,
    \[\sum_{m\geq 1}a_{n+m}\rho^{[m]}e^m = 0.\]
    We claim that $a_{n+1} = 0$ for any $n\geq 0$. To see this, it is enough to show that $a_{n+1}\in e^mA$ for any $m\geq 1$. We will prove this by induction on $m\geq 1$.
    Since $\nu_p(\rho)\geq \frac{1}{p-1}$, for any $n\geq 1$, we have $\nu_p(\rho^{[n]})\geq \nu_p(\rho)$. Since $e$ is regular in $A$, we see that for any $n\geq 0$,
    \[a_{n+1} = -\sum_{m\geq 1}a_{n+1+m}\frac{\rho^{[m+1]}}{\rho}e^m.\]
    In particular, we have $a_{n+1}\in eA$ for any $n\geq 0$. Now, assume $a_{n+1}\in e^kA$ for some $k\geq 1$ and then we see that
    \[a_{n+1}\in \sum_{m\geq 1}\frac{\rho^{[m+1]}}{\rho}e^{m+k}A\subset e^{k+1}A.\]
    This proves the claim. In particular, $x = \sum_{n\geq 0}a_n(\rho X)^{[n]}\in (A[\rho X]_{\pd}^{\wedge})^{\gamma=1}$ if and only if $x = a_0\in A$, which implies that $\rH^0(\Gamma,A[\rho X]_{\pd}^{\wedge}) \cong A$.

    To conclude the result, we have to show that $\rH^1(\Gamma, A[\rho X]_{\pd}^{\wedge}) = A[\rho X]^{\wedge}_{\pd}/\rho eA[\rho X]^{\wedge}_{\pd}$. One can easily deduce from (\ref{Equ-(gamma-1)-A}) that $(\gamma-1)A[\rho X]^{\wedge}_{\pd}\subset \rho eA[\rho X]^{\wedge}_{\pd}$. 
    On the other hand, for any $y = \sum_{n\geq 0}b_n(\rho X)^{[n]}$, put $z = \zeta^{-\alpha}\sum_{n\geq 1}b_{n-1}(\rho X)^{[n]}$. Then we have
    \[(\gamma-1)z - \rho ey = \sum_{n\geq 0}(\sum_{m\geq 1}b_{n-1+m}\rho^{[m]}e^m-b_n\rho e)(\rho X)^{[n]} = \rho e^2\sum_{n\geq 0}\sum_{m\geq 1}b_{n+m}\frac{\rho^{[m+1]}}{\rho}e^{m-1}(\rho X)^{[n]}.\]
    As a consequence, we have
    \[\rho eA[\rho X]^{\wedge}_{\pd}\subset (\gamma-1)A[\rho X]^{\wedge}_{\pd}+\rho e^2A[\rho X]^{\wedge}_{\pd}.\]
    Since $e$ is topologically nilpotent, by Nakayama's Lemma, we get $(\gamma-1)A[\rho X]^{\wedge}_{\pd}=\rho eA[\rho X]^{\wedge}_{\pd}$ as desired.
  \end{proof}
\subsection{$\Gamma$-cohomology of log nilpotent representations}
\label{SSec-Gamma-cohomology of LogNilRep}
  We keep notations as in \S\ref{SSec-Gamma-cohomology of PD-ring}. Recall we have introduced the log-nilpotent $A$-representations of $\Gamma$ in Definition \ref{Dfn-LogNilRep}.

  \begin{notation}
    Let $V$ be a log nilpotent $A$-representation of $\Gamma$. For any $\alpha\in \bZ[1/p]\cap [0,1)$, we define 
    \[M_{\alpha}(V):=V\otimes_AM_{\alpha}\]
    on which $\Gamma$ acts diagonally.
  \end{notation}

  We want to compute $\rR\Gamma(\Gamma, M_{\alpha}(V))$ for any $\alpha\in \bZ[1/p]\cap[0,1)$. Let $v_1,\dots, v_r$ be a fixed $A$-basis of $V$ and let $P = \exp(-(\zeta_p-1)\Theta)$ be the matrix of $\gamma$ with respect to the chosen basis with $\Theta$ as mentioned in Definition \ref{Dfn-LogNilRep}. As an $A$-module, we see that
  \[M_{\alpha}(V) = \widehat {\bigoplus_{n\geq 0} }\bigoplus_{i=1}^rAv_i\otimes(\rho X)^{[n]}e_{\alpha}.\]
  For any $n\geq 0$, we have
  \[\gamma(\underline v\otimes(\rho X)^{[n]}e_{\alpha}) = \underline vP\otimes \zeta^{\alpha}\sum_{i=0}^n\rho^{[n-i]}e^{n-i}(\rho X)^{[i]}e_{\alpha}.\]
  Here and in what follows, we put
  \[\underline v\otimes (\rho X)^{[n]}e_{\alpha}:=(v_1,\dots,v_r)\otimes(\rho X)^{[n]}e_{\alpha}\]
  for short.
  So for any $x = \underline v\otimes \sum_{n\geq 0}a_n(\rho X)^{[n]}e_{\alpha}\in M_{\alpha}(V)$ with $a_n$'s lying in $A^r$,
  \begin{equation}\label{Equ-(gamma-1)-B}
      \begin{split}
          (\gamma-1)x & = \underline v\otimes(\sum_{n\geq 0}Pa_n \zeta^{\alpha}\sum_{i=0}^n\rho^{[n-i]}e^{n-i}\rho^iX^{[i]}e_{\alpha}-\sum_{n\geq 0}a_n(\rho X)^{[n]}e_{\alpha})\\
          & = \underline v\otimes\sum_{n\geq 0}(\sum_{m\geq 0}\zeta^{\alpha}Pa_{n+m}\rho^{[m]}e^m-a_n)(\rho X)^{[n]}e_{\alpha}\\
          & = \underline v\otimes\sum_{n\geq 0}(\sum_{m\geq 1}\zeta^{\alpha}Pa_{n+m}\rho^{[m]}e^m+(\zeta^{\alpha}P-I)a_n)(\rho X)^{[n]}e_{\alpha}.
      \end{split}
  \end{equation}

  \begin{lem}\label{Lem-Cogo-Error}
      Assume $\alpha\neq 0$. For any log nilpotent $A$-representation $V$ of $\Gamma$, we have $\rH^0(\Gamma,M_{\alpha}(V)) = 0$ and $\rH^1(\Gamma,M_{\alpha}(V)) = M_{\alpha}(V)/(\zeta^{\alpha}-1)M_{\alpha}(V)$.
  \end{lem}
  \begin{proof}
      The proof is similar to that of Proposition \ref{Prop-Gamma-cohomology of pd ring}. We write $P = \exp(-(\zeta_p-1)\Theta)$ with $\Theta$ topologically nilpotent and define 
      \[\Theta_{\alpha} = I+\zeta^{\alpha}\Theta\sum_{l\geq 0}(-1)^l\frac{(\zeta_p-1)^{[l+1]}}{\zeta^{\alpha}-1}\Theta^l.\]
      Then $\Theta_{\alpha}\in \GL_r(A)$ satisfying $(\zeta^{\alpha}-1)\Theta_{\alpha} = \zeta^{\alpha}P-I$.

      Assume $x = \underline v\otimes \sum_{n\geq 0}a_n(\rho X)^{[n]}e_{\alpha}\in M_{\alpha}(V)^{\gamma=1}$ with $a_n$'s lying in $A^r$. By (\ref{Equ-(gamma-1)-B}), we know that for any $n\geq 0$,
      \[a_n = -e\sum_{m\geq 0}\zeta^{\alpha}\Theta_{\alpha}^{-1}Pa_{n+1+m}\frac{\rho^{[m+1]}}{\zeta^{\alpha}-1}e^m\in eA^r.\]
      Now assume for all $n\geq 0$, $a_n\in e^kA^r$ for some $k\geq 1$. Then the above formulae implies that 
      \[a_n\in e\sum_{m\geq 0}-e\sum_{m\geq 0}\zeta^{\alpha}\frac{\rho^{[m+1]}}{\zeta^{\alpha}-1}e^m\Theta_{\alpha}^{-1}Pe^kA^r\subset e^{k+1}A^r.\]
      Since $e$ is topologically nilpotent, we have that $a_n = 0$ for all $n\geq 0$, which forces $x = 0$ as desired.

      To complete the proof, we are reduced to showing that  $(\gamma-1)M_{\alpha}(V) = (\zeta^{\alpha}-1)M_{\alpha}(V)$. Using (\ref{Equ-(gamma-1)-B}), it is easy to see that $(\gamma-1)M_{\alpha}(V) \subset (\zeta^{\alpha}-1)M_{\alpha}(V)$. On the other hand, for any $y = \underline v\otimes\sum_{n\geq 0}b_n(\rho X)^{[n]}e_{\alpha}\in M_{\alpha}(V)$, by (\ref{Equ-(gamma-1)-B}), we have
      \begin{equation*}
          \begin{split}
              (\gamma-1)\Theta_{\alpha}^{-1}y-(\zeta^{\alpha}-1)y & = \underline v\otimes\sum_{n\geq 0}(\sum_{m\geq 1}\zeta^{\alpha}P\Theta_{\alpha}^{-1}b_{n+m}\rho^{[m]}e^m+(\zeta^{\alpha}P-I)\Theta_{\alpha}^{-1}b_n-(\zeta^{\alpha}-1)b_n)(\rho X)^{[n]}e_{\alpha}\\
              & = \underline v\otimes\sum_{n\geq 0}(\sum_{m\geq 1}\zeta^{\alpha}P\Theta_{\alpha}^{-1}b_{n+m}\rho^{[m]}e^m)(\rho X)^{[n]}e_{\alpha}\quad(\because (\zeta^{\alpha}-1)\Theta_{\alpha} = \zeta^{\alpha}P-I)\\
              & \in \rho eM_{\alpha}(V)\subset e(\zeta^{\alpha}-1)M_{\alpha}(V).
          \end{split}
      \end{equation*}
      So we get $(\zeta^{\alpha}-1)M_{\alpha}(V)\subset (\gamma-1)M_{\alpha}(V)+e(\zeta^{\alpha}-1)M_{\alpha}(V)$. Since $e$ is topologically nilpotent, by Nakayama's Lemma, we conclude that $(\zeta^{\alpha}-1)M_{\alpha}(V)\subset (\gamma-1)M_{\alpha}(V)$ as desired.
  \end{proof}
  \begin{construction}\label{Construction-f_V}
    For any log nilpotent $A$-representation $V$ of $\Gamma$, we define $A$-linear operators $g_V, f_V:M_0(V)\to M_0(V)$ as follows:

    Let $v_1,\dots,v_r$ be a fixed $A$-basis of $V$ such that the corresponding matrix of $\gamma$ is $ P = \exp(-(\zeta_p-1)\Theta)$. Define $\Theta_0 = \sum_{l\geq 0}(-1)^{l+1}\frac{(\zeta_p-1)^{[l+1]}}{\zeta_p-1}\Theta^l$. Then $\Theta_0\in\GL_r(A)$ satisfying $(\zeta_p-1)\Theta\Theta_0 = P-I$.
    Assume $x = \underline v\otimes\sum_{n\geq 0}b_n(\rho X)^{[n]}$ with $b_n$'s belonging to $A^r$. Then for any $n\geq 0$, we can define
    \begin{equation}\label{Equ-f_V-A}
        \begin{split}
            a_{n+1} :=b_n+\sum_{l\geq 1}(-1)^l\sum_{m_1,\dots,m_l\geq 1}P^lb_{n+m_1+\cdots+m_l}\frac{\rho^{[m_1+1]}}{\rho}\cdots\frac{\rho^{[m_l+1]}}{\rho}e^{m_1+\cdots+m_l}\in A^r.
        \end{split}
    \end{equation}
    As $\lim_{n\to+\infty}b_n = 0$, we have $\lim_{n\to+\infty}a_n = 0$. In particular 
    \begin{equation}\label{Equ-f_V-B}
        g_V(x) =\underline v\otimes\sum_{n\geq 1}a_n(\rho X)^{[n]}
    \end{equation}
    is a well-defined element in $M_0(V)$. Define
    \begin{equation}\label{Equ-f_V-C}
        f_V(x) = \underline v\otimes(\rho eb_0+\sum_{n\geq 1}(\rho eb_n+(\zeta_p-1)\Theta_0\Theta a_n)(\rho X)^{[n]}) = \rho ex+(\zeta_p-1)\Theta_0\Theta g_V(x).
    \end{equation}
  \end{construction}
  \begin{lem}\label{Lem-f_V}
    \begin{enumerate}
        \item[(1)] For any $ y = \underline v\otimes\sum_{n\geq 0}b_n(\rho X)^{[n]}\in M_0(V)$, we have $(\gamma-1)g_V(y) = f_V(y)$.
        
        \item[(2)] For any $ x = \underline v\otimes\sum_{n\geq 0}a_n(\rho X)^{[n]}\in M_0(V)$ and any $n\geq 0$, define $b_n=\sum_{m\geq 0}Pa_{n+1+m}\frac{\rho^{[m+1]}}{\rho}e^m$. Then $ y = \underline v\otimes\sum_{n\geq 0}b_n(\rho X)^{[n]} $ is well-defined in $M_0(V)$ and independent of $a_0$ such that 
        \[x = g_V(y)+\underline v\otimes a_0.\]
        In particular, we have $(\gamma-1)x = f_V(y)+\underline v\otimes(\zeta_p-1)\Theta\Theta_0a_0$.

        \item[(3)] Write $M_0^+(V) = V\otimes_A\hat \oplus_{n\geq 1}A(\rho X)^{[n]}$. Then $g_V$ takes values in $M_0^+(V)$ and induces an isomorphism $g_V:M_0(V)\to M_0^+(V)$ of $A$-modules.
    \end{enumerate}
  \end{lem}
  \begin{proof}
      We first prove Item (1).
      We claim that $b_n=\sum_{m\geq 0}Pa_{n+1+m}\frac{\rho^{[m+1]}}{\rho}e^m$. Indeed, by (\ref{Equ-f_V-A}), we have
      \[\begin{split}
        &\sum_{m\geq 1}Pa_{n+1+m}\frac{\rho^{[m+1]}}{\rho}e^m \\
        =& \sum_{m\geq 1}(Pb_{n+m}\frac{\rho^{[m+1]}}{\rho}e^m+\sum_{l\geq 1}(-1)^l\sum_{m_1,\dots,m_l\geq 1}P^{l+1}b_{n+m+m_1+\cdots+m_l}\frac{\rho^{[m+1]}}{\rho}\frac{\rho^{[m_1+1]}}{\rho}\cdots\frac{\rho^{[m_l+1]}}{\rho}e^{m+m_1+\cdots+m_l})\\
        =&\sum_{m_1\geq 1}Pb_{n+m_1}\frac{\rho^{[m_1+1]}}{\rho}e^{m_1+1}-\sum_{l\geq 1}(-1)^{l+1}\sum_{m_1,\dots,m_l,m_{l+1}\geq 1}P^{l+1}b_{n+m_1+\cdots+m_{l+1}}\frac{\rho^{[m_1+1]}}{\rho}\cdots\frac{\rho^{[m_{l+1}+1]}}{\rho}e^{m_1+\cdots+m_{l+1}}\\
        =&-\sum_{l\geq 1}(-1)^l\sum_{m_1,\dots,m_l\geq 1}P^lb_{n+m_1+\cdots+m_l}\frac{\rho^{[m_1+1]}}{\rho}\cdots\frac{\rho^{[m_l+1]}}{\rho}e^{m_1+\cdots+m_l}\\
        =& b_n-a_{n+1}
      \end{split}\]
      as desired, where the last equality follows from (\ref{Equ-f_V-A}) again.

      Combining above formulae with (\ref{Equ-(gamma-1)-B}) and (\ref{Equ-f_V-B}) and letting $a_0=0$, we see that
      \[\begin{split}
        (\gamma-1)(g_V(y)) & = \underline v\otimes\sum_{n\geq 0}(\sum_{m\geq 1}Pa_{n+m}\rho^{[m]}e^m+(P-I)a_n)(\rho X)^{[n]}\\
        & = \underline v\otimes\sum_{n\geq 0}(\rho e\sum_{m\geq 0}Pa_{n+1+m}\frac{\rho^{[m+1]}}{\rho}e^m+(P-I)a_n)(\rho X)^{[n]}\\
        & = \rho ey+(\zeta_p-1)\Theta_0\Theta g_V(y).
      \end{split}\]
      Now Item (1) follows from (\ref{Equ-f_V-C}) directly.

      We now prove Item (2). Since $\lim_{n\to+\infty}a_n = 0$, we have $\lim_{n\to+\infty}b_n = 0$ and thus $y$ is well-defined. It remains to show that $x = g_V(y)+\underline v\otimes a_0$ while the rest follows from this together with Item (1). To do so, it suffices to show that $\{a_{n}\}_{n\geq 1}$ and $\{b_n\}_{n\geq 0}$ fit into the equation (\ref{Equ-f_V-B}) for any $n\geq 0$.

      Since $b_n=\sum_{m\geq 0}Pa_{n+1+m}\frac{\rho^{[m+1]}}{\rho}e^m$ for any $n\geq 0$, we see that for any $k\geq 1$,
      \begin{equation*}
          \begin{split}
              a_{n+1} &= b_n - \sum_{m_1\geq 1}Pa_{n+1+m_1}\frac{\rho^{[m_1+1]}}{\rho}e^{m_1}\\
              =&b_n - \sum_{m_1\geq 1}P(b_{n+m_1} - \sum_{m_2\geq 1}Pa_{n+1+m_1+m_2}\frac{\rho^{[m_2+1]}}{\rho}e^{m_2})\frac{\rho^{[m_1+1]}}{\rho}e^{m_1}\\
              =&b_n - \sum_{m_1\geq 1}Pb_{n+m_1}\frac{\rho^{[m_1+1]}}{\rho}e^{m_1}+\sum_{m_1,m_2\geq 1}P^2a_{n+1+m_1+m_2}\frac{\rho^{[m_1+1]}}{\rho}\frac{\rho^{[m_2+1]}}{\rho}e^{m_1+m_2}\\
              =&\cdots\\
              =&b_n+\sum_{l=1}^k(-1)^l\sum_{m_1,\dots,m_l\geq 1}P^lb_{n+m_1+\cdots+m_l}\frac{\rho^{[m_1+1]}}{\rho}\cdots\frac{\rho^{[m_l+1]}}{\rho}e^{m_1+\cdots+m_l}\\
              &+(-1)^{k+1}\sum_{m_1,\dots,m_{k+1}\geq 1}P^{k+1}a_{n+1+m_1+\cdots+m_{k+1}}\frac{\rho^{[m_1+1]}}{\rho}\cdots\frac{\rho^{[m_{k+1}+1]}}{\rho}e^{m_1+\cdots+m_{k+1}}\\
              \equiv & b_n+\sum_{l=1}^k(-1)^l\sum_{m_1,\dots,m_l\geq 1}P^lb_{n+m_1+\cdots+m_l}\frac{\rho^{[m_1+1]}}{\rho}\cdots\frac{\rho^{[m_l+1]}}{\rho}e^{m_1+\cdots+m_l}\mod e^{k+1}.
          \end{split}
      \end{equation*}
      Since $e$ is topologically nilpotent, we see that (\ref{Equ-f_V-B}) holds true as desired. So Item (2) is true.

      Now we are going to prove Item (3). By (\ref{Equ-f_V-B}), we know that $g_V$ takes values in $M_0^+(V)$ which is indeed a surjection by Item (2). It remains to show $g_V$ is injective.

      To do so, let us fix an $x = \sum_{n\geq 0}b_n(\rho X)^{[n]}$ which is killed by $g_V$. By (\ref{Equ-f_V-A}) and (\ref{Equ-f_V-B}), for any $n\geq 0$, we have 
      \[b_n = -\sum_{l\geq 1}(-1)^l\sum_{m_1,\dots,m_l\geq 1}P^lb_{n+m_1+\cdots+m_l}\frac{\rho^{[m_1+1]}}{\rho}\cdots\frac{\rho^{[m_l+1]}}{\rho}e^{m_1+\cdots+m_l}\in eA^r.\]
      By iteration, we know for any $k\geq 1$, $b_n\in e^kA^r$ and hence $b_n = 0$, by the topological nilpotency of $e$. This implies that $x = 0$.
  \end{proof}
  \begin{lem}\label{Lem-Coho-Principle}
    The morphism $(g_V,\id_V):M_0(V)\oplus V\to M_0(V)$ is an isomorphism fitting into the following commutatative diagram
    \begin{equation}\label{Diag-f_V}
        \xymatrix@C=0.5cm{
          M_0(V)\oplus V\ar[rrrr]^{(f_V,(\zeta_p-1)\Theta\Theta_0)}\ar[drr]_{\cong}^{(g_V,\id_V)}&&&&M_0(V)\\
          &&M_0(V)\ar[urr]^{\gamma-1}.
        }
    \end{equation}
    Moreover, via the isomorphism $(g_V,\id_V)$, we have 
    \[\rH^0(\Gamma,M_0(V))\cong \{(x,v)\in M_0(V)\oplus V\mid f_V(x) = (\zeta_p-1)\Theta(v)\}\]
    and 
    \[\rH^1(\Gamma,V)\cong M_0(V)/(\Ima(f_V)+(\zeta_p-1)\Theta(V)).\]
  \end{lem}
  \begin{proof}
      Since $M_0(V) = V\oplus M_0^+(V)$, we know $(g_V,\id_V)$ is an isomorphism by Lemma \ref{Lem-f_V} (3). Since $\gamma$ acts on $V$ via the matrix $P$, we know $\gamma-1$ acts on $V$ via $P-I = (\zeta_p-1)\Theta\Theta_0$. So the commutativity of (\ref{Diag-f_V}) follows from Lemma \ref{Lem-f_V} (1). Since $\rR\Gamma(\Gamma,M_0(V))$ is computed by the Koszul complex $M_0(V)\xrightarrow{\gamma-1}M_0(V)$, we know it is also calculated by the complex 
      \[ M_0(V)\oplus V\xrightarrow{(f_V,(\zeta_p-1)\Theta\Theta_0)}M_0(V).\]
      Then the ``moreover'' part follows by noticing that $\Theta_0$ is an automorphism of $V$.
  \end{proof}

\subsection{The case for $e = \zeta_p-1$}
  Throughout this subsection, we keep notations as above and furthermore assume that $e = \zeta_p-1$. Then for $x = \underline v\otimes\sum_{n\geq 0}b_n(\rho X)^{[n]}$, the equation (\ref{Equ-f_V-C}) can be reformulated as
  \begin{equation}\label{Equ-f_V-e=zeta-1-I}
      \begin{split}
          &f_V(\underline v\otimes\sum_{n\geq 0}b_n(\rho X)^{[n]})\\ 
          = &\underline v\otimes e\sum_{n\geq 1}(\rho b_n
          +\Theta\Theta_0(b_{n-1}+\sum_{l\geq 1,m_1,\dots,m_l\geq 1}(-1)^lP^lb_{n-1+m_1+\cdots+m_l}\frac{e^{[m_1+1]}}{e}\cdots\frac{e^{[m_l+1]}}{e}\rho^{m_1+\cdots+m_l}))(\rho X)^{[n]}\\
          &+\underline v\otimes e\rho b_0\\ 
          = &\underline v\otimes e\sum_{n\geq 1}(\rho b_n
          +\Theta\Theta_0(b_{n-1}+\sum_{1\leq l\leq m}\sum_{m_1,\dots,m_l\geq 1,m_1+\cdots+m_l =m}(-1)^lP^lb_{n-1+m}\frac{e^{[m_1+1]}\cdots e^{[m_l+1]}}{e^l}\rho^{m}))(\rho X)^{[n]}\\
          &+\underline v\otimes e\rho b_0.
      \end{split}
  \end{equation}
  \begin{notation}
      Put $Q_0=I$. For any $m\geq l\geq 1$, we define 
      \[Q_{m}^{(l)} = \sum_{m_1,\dots,m_l\geq 1,m_1+\cdots+m_l =m}\frac{e^{[m_1+1]}\cdots e^{[m_l+1]}}{e^l}(-1)^lP^l\]
      and
      \[Q_m = \sum_{l=1}^mQ_{m}^{(l)}= \sum_{l=1}^m\sum_{m_1,\dots,m_l\geq 1,m_1+\cdots+m_l =m}(-1)^lP^l\frac{e^{[m_1+1]}\cdots e^{[m_l+1]}}{e^l}.\]
  \end{notation}
  Using above notation, we can write (\ref{Equ-f_V-e=zeta-1-I}) as 
  \begin{equation}\label{Equ-f_V-e=zeta-1-II}
      f_V(\underline v\otimes\sum_{n\geq 0}b_n(\rho X)^{[n]}) = \underline v\otimes e\rho b_0+\underline v\otimes e\sum_{n\geq 1}(\rho b_n
          +\Theta\Theta_0b_{n-1}+\Theta\Theta_0\sum_{m\geq 1}\rho^mQ_mb_{n-1+m})(\rho X)^{[n]}.
  \end{equation}
  Similarly, we have
  \begin{equation}\label{Equ-g_V-e=zeta-1}
      g_V(x) = \underline v\otimes \sum_{n\geq 1}(b_{n-1}+\sum_{m\geq 1}\rho^mQ_mb_{n-1+m})(\rho X)^{[n]}
  \end{equation}

  \begin{lem}\label{Lem-Coho-Principal-e=zeta-1}
      Assume $e = \zeta_p-1$ and $V$ is a log nilpotent $A$-representation of $\Gamma$. Define 
      \[V_{\Theta/\rho} = \{v\in V\mid \forall~ n\geq 1, \Theta^nv\in\rho^nV~\&~\lim_{n\to+\infty}\rho^{-n}\Theta^nv = 0\}.\]
      Then we have $\rH^0(\Gamma,M_0(V))\cong V_{\Theta/\rho}$.
  \end{lem}
  \begin{proof}
      By Lemma \ref{Lem-Coho-Principle}, it suffices to show that if $(x,v)\in M_0(V)\oplus V$ satisfying $f_V(x) = e\Theta v$, then $v\in V_{\Theta/\rho}$ and $x$ is uniquely determined by $v$. 
      
      Fix such a pair $(x,v)$ and write $x = \underline v\otimes\sum_{n\geq 0}b_n(\rho X)^{[n]}$ as above. By (\ref{Equ-f_V-e=zeta-1-II}), we have 
      \begin{equation}\label{Equ-Coho-Principal-e=zeta-1-A}
          \Theta v = \underline v\otimes \rho b_0
      \end{equation}
      and for any $n\geq 0$,
      \begin{equation}\label{Equ-Coho-Principal-e=zeta-1-B}
          \rho b_{n+1}=
          -\Theta\Theta_0b_{n}-\Theta\Theta_0\sum_{m\geq 1}\rho^mQ_mb_{n+m}.
      \end{equation}
      Applying Lemma \ref{Lem-Technique-I} (to $U=-\Theta\Theta_0$ and $a_n = 0$ for all $n\geq 1$), there exists an explicit invertible matrix $R_0\in\GL_r(A)$ such that (\ref{Equ-Coho-Principal-e=zeta-1-B}) holds true if and only if for any $n\geq 0$,
      \[\rho b_{n+1} = -R_0\Theta_0\Theta b_n = \Theta b_n,\]
      where the last equality follows from Example \ref{Exam-Special case for R}. Using (\ref{Equ-Coho-Principal-e=zeta-1-A}), we see that for any $n\geq 0$,
      \[\underline v\otimes \rho^{n+1} b_n = \Theta^{n+1}v.\]
      In particular, for any $n\geq 1$, $b_n$ is uniquely determined by $v$. On the other hand, for any $n\geq 1$, $\Theta^n v = \underline v\otimes\rho^n b_n \in \rho^nV$ and 
      \[\lim_{n\to+\infty}\rho^{-n}\Theta^nv = \lim_{n\to+\infty}\underline v\otimes b_n = 0\]
      as $\lim_{n\to+\infty}b_n = 0$. This implies that $v\in V_{\Theta/\rho}$ as desired.
  \end{proof}
  \begin{rmk}\label{Rmk-Coho-Principal-e=zeta-1}
      Keep notations as in Lemma \ref{Lem-Coho-Principal-e=zeta-1}. Let $v = \underline v\otimes\lambda\in V_{\Theta/\rho}$ and $x = \underline v\otimes\sum_{n\geq 0}b_n(\rho X)^{[n]}\in M_0(V)$ such that $f_V(x) = e\Theta v$. The proof of Lemma \ref{Lem-Coho-Principal-e=zeta-1} tells us that for any $n\geq 0$, 
      \[b_n = (\rho^{-1}\Theta)^{n+1}\lambda.\]
      As a consequence, we have
      \[x = \underline v\otimes\sum_{n\geq 0}(\rho^{-1}\Theta)^{n+1}\lambda(\rho X)^{[n]}=\sum_{n\geq 0}(\rho^{-1}\Theta)^{n+1}(\rho X)^{[n]}v = \rho^{-1}\Theta\exp(\Theta X)v.\]
      Here, $\exp(\Theta X)v:=\sum_{n\geq 0}(\rho^{-1}\Theta)^{n}v(\rho X)^{[n]}$ is well-defined in $M_0(V)$.
  \end{rmk}
  \begin{cor}\label{Cor-Explicit H^0}
      Keep notations as in Lemma \ref{Lem-Coho-Principal-e=zeta-1}. Then as a subset of $M_0(V)$, we have
      \[\rH^0(\Gamma,M_0(V)) = \{\exp(\Theta X)v\mid v\in V_{\Theta/\rho}\}.\]
  \end{cor}
  \begin{proof}
      By Lemma \ref{Lem-Coho-Principle}, $\rH^0(\Gamma,M_0(V)) = \{v+g_V(x)\mid (x,v)\in M_0(V)\oplus V, f_V(x)+e\Theta\Theta_0v = 0\}$. Applying Lemma \ref{Lem-Coho-Principal-e=zeta-1} together with Remark \ref{Rmk-Coho-Principal-e=zeta-1}, we have 
      \[\rH^0(\Gamma,M_0(V)) = \{v-\rho^{-1}\Theta\Theta_0g_V(\exp(\Theta X)v)\mid v\in V_{\Theta/\rho}\}.\]
      By (\ref{Equ-g_V-e=zeta-1}), we see that 
      \begin{equation*}
          \begin{split}
              g_V(\exp(\Theta X)v) = &g_V(\sum_{n\geq 0}(\rho^{-1}\Theta)^nv(\rho X)^{[n]})\\
              = &\sum_{n\geq 1}((\rho^{-1}\Theta)^{n-1}v+\sum_{m\geq 1}Q_m\rho^m(\rho^{-1}\Theta)^{n-1+m}v)(\rho X)^{[n]}\\
              = &\sum_{m\geq 0}Q_m\Theta^m(\sum_{n\geq 1}(\rho^{-1}\Theta)^{n-1}v(\rho X)^{[n]}).
          \end{split}
      \end{equation*}
      Therefore, we have
      \begin{equation*}
          \begin{split}
               &v-\rho^{-1}\Theta\Theta_0g_V(\exp(\Theta X)v)\\
              =&v-\rho^{-1}\Theta\Theta_0\sum_{m\geq 0}Q_m\Theta^m(\sum_{n\geq 1}(\rho^{-1}\Theta)^{n-1}v(\rho X)^{[n]})\\
              =&v-\Theta_0\sum_{m\geq 0}Q_m\Theta^m(\sum_{n\geq 1}(\rho^{-1}\Theta)^{n}v(\rho X)^{[n]})\\
              =&v-\Theta_0\sum_{m\geq 0}Q_m\Theta^m(\exp(\Theta X)v-v)\\
              =&v+\Theta_0\sum_{m\geq 0}Q_m\Theta^mv-\Theta_0\sum_{m\geq 0}Q_m\Theta^m\exp(\Theta X)v\\
              =&\exp(\Theta X)v.
          \end{split}
      \end{equation*}
      Here, the last equality holds true as $-\Theta_0^{-1} = \sum_{m\geq 0}Q_m\Theta^m$ by Example \ref{Exam-Special case for R}.
  \end{proof}
  
  \begin{exam}\label{Exam-V_{Theta/rho}}
  \begin{enumerate}
    
      \item[(1)] If there exists an integer $n\geq 1$ such that $\Theta^n = \rho^n\Theta^{\prime}$ for some topologically nilpotent $\Theta^{\prime}\in M_r(A)$, then $\rho^{n-1}V\subset V_{\Theta/\rho}$. 
      Indeed, for any $d\geq 0$, $0\leq l\leq n-1$ and $v\in V$, 
      \[\Theta^{dn+l}\rho^{n-1}v = \rho^{dn+l}(\Theta^{\prime})^d\rho^{n-1-l}\Theta^lv\in \rho^{dn+l}V.\]
      Moreover, the above formulae also implies that for any $v\in V$, $\lim_{m\to+\infty}\rho^{-m}\Theta^m\rho^{n-1}v = 0$.

      \item[(2)] If $\Theta=\lambda \Theta^{\prime}$ for some $\Theta^{\prime}\in\GL_r(A)$ and some $\lambda\in\calO_C$ with $\nu_p(\lambda)<\nu_p(\rho)$, then $V_{\Theta/\rho} = 0$.
  \end{enumerate}
  \end{exam}
  \begin{rmk}
      If $\Theta$ is as in Example \ref{Exam-V_{Theta/rho}} (1), then the natural inclusion $M_0(V)^{\Gamma}\to M_0(V)$ induces an isomorphism 
      \[(M_0(V)[\frac{1}{p}])^{\Gamma}\otimes_AA[\rho X]^{\wedge}_{\pd}\cong M_0(V)[\frac{1}{p}]\]
      as $\exp(\Theta X)$ is a well-defined automorphism in $\GL_r(A[\rho X]^{\wedge}_{\pd}[\frac{1}{p}])$. If moreover, $\Theta = \rho\Theta^{\prime}$ for some topologically nilpotent $\Theta^{\prime}$, we have an isomorphism 
      \[(M_0(V))^{\Gamma}\otimes_AA[\rho X]^{\wedge}_{\pd}\cong M_0(V).\]
  \end{rmk}
  
      Although $\rH^0(\Gamma,M_0(V))$ has a good form, $\rH^1(\Gamma,M_0(V))$ is too complicated to describe. However, we can show it has bounded $p^{\infty}$-torsion in a special case.
  \begin{lem}\label{Lem-H^1-small case}
      Suppose that $e = \zeta_p-1$ and $V$ is a log nilpotent $A$-representation of $\Gamma$. If there exists some $d\geq 1$ such that $\Theta^d = \rho^d\Theta^{\prime}$ for some topologically nilpotent $\Theta^{\prime}\in M_r(A)$, then 
      \[\rho^deM_0(V)\subset \Ima(f_V).\]
      As a consequence, $\rH^1(\Gamma,M_0(V))$ is killed by $\rho^de$ (cf. Lemma \ref{Lem-Coho-Principle}). If moreover $d = 1$, then the natural inclusion $M_0(V)^{\Gamma}\subset M_0(V)$ induces an isomorphism $M_0(V)^{\Gamma}\otimes_AA[\rho X]_{\pd}^{\wedge} = M_0(V)$ such that
      \[\rH^1(\Gamma,M_0(V))\cong M_0(V)/\rho eM_0(V).\]
  \end{lem}
  \begin{proof}
      We first show that $\rho^deM_0(V)\subset \Ima(f_V)$.
      Fix a $y = \underline v\otimes\sum_{n\geq 0}c_n(\rho X)^{[n]}\in M_0(V)$. We have to construct an $x = \underline v\otimes\sum_{n\geq 0}b_n(\rho X)^{[n]}\in M_0(V)$ such that $f_V(x) = \rho^dey$. By (\ref{Equ-f_V-e=zeta-1-II}), it suffices to construct $b_n$'s such that
      \[\rho eb_0 = \rho^dec_0\]
      and for any $n\geq 0$,
      \[e(\rho b_{n+1}+\Theta\Theta_0b_{n}+\Theta\Theta_0\sum_{m\geq 1}\rho^mQ_mb_{n+m}) = \rho^dec_{n+1}.\]
      Equivalently, we have $b_0 = \rho^{d-1}c_0$ and for any $n\geq 0$,
      \begin{equation}\label{Lem-H^1-small case-A}
          \rho b_{n+1}=\rho^dc_{n+1}-\Theta\Theta_0b_{n}-\Theta\Theta_0\sum_{m\geq 1}\rho^mQ_mb_{n+m}.
      \end{equation}
      Let $R_0$ be the invertible matrix obtained by applying Lemma \ref{Lem-Technique-I} to $ U = -\Theta\Theta_0$ (indeed, $R_0=-\Theta_0^{-1}$ by Example \ref{Exam-Special case for R} and hence $R_0U_0=\Theta$). Apply Lemma \ref{Lem-Technique-I} to $a_{n+1} = \rho^dc_{n+1}$ and then for any $n\geq 0$,
      \[\rho b_{n+1} = R_0U_0b_n+R_0\rho^dc_{n+1}+U\sum_{m\geq 2}\rho^{m+d-1}a_{n+m}=\Theta b_n+R_0\rho^dc_{n+1}+U\sum_{m\geq 2}\rho^{m+d-1}a_{n+m}.\]
      Equivalently, for any $n\geq 0$,
      \[\rho^{n+1}b_{n+1} = \Theta\rho^nb_n+R_0\rho^{n+d}c_{n+1}+U\sum_{m\geq 2}\rho^{n+m+d-1}c_{n+m}.\]
      By iteration, we conclude that for any $n\geq 0$,
      \begin{equation*}
          \begin{split}
              \rho^{n+1}b_{n+1} = &\Theta^{n+1}b_0+\sum_{i=0}^{n}\Theta^i(R_0\rho^{n-i+d}c_{n-i+1}+U\sum_{m\geq 2}\rho^{n+m+d-i-1}c_{n-i+m})\\
              =&\Theta^{n+1}\rho^{d-1}c_0+\sum_{i=0}^{n}\Theta^i(R_0\rho^{n-i+d}c_{n-i+1}+U\sum_{m\geq 2}\rho^{n+m+d-i-1}c_{n-i+m}).
          \end{split}
      \end{equation*}
      As proved in Example \ref{Exam-V_{Theta/rho}} (1) that for any $n\geq 0$, $\Theta^n\rho^{d-1}V\subset \rho^{n}V$. So for any $n\geq 0$,
      \begin{equation}\label{Lem-H^1-small case-B}
          b_{n+1} = \rho^{-n-1}\Theta^{n+1}\rho^{d-1}c_0+\sum_{i=0}^{n}\rho^{-i}\Theta^i(R_0\rho^{d-1}c_{n-i+1}+U\sum_{m\geq 2}\rho^{m+d-2}c_{n-i+m})
      \end{equation}
      is well-defined in $A^r$. The ``moreover'' part of Example \ref{Exam-V_{Theta/rho}} (1) tells us that 
      \[\lim_{n\to+\infty}b_n = 0.\]
      So for any $n\geq 0$, if we define $b_{n+1}$ by (\ref{Lem-H^1-small case-B}) and define $b_0 = \rho^{d-1}c_0$, then $x = \sum_{m\geq 0}b_n\rho^nX^{[m]}$ is well-defined in $M_0(V)$ such that for any $n\geq 0$,
      \begin{equation*}
          \begin{split}
              \rho b_{n+1} = &\rho^{-n}\Theta^{n+1}\rho^{d-1}c_0+\Theta\sum_{i=1}^{n}\rho^{-i+1}\Theta^{i-1}(R_0\rho^{d-1}c_{n-i+1}+U\sum_{m\geq 2}\rho^{m+d-2}c_{n-i+m})\\
              &+R_0\rho^{d-1}c_{n+1}+U\sum_{m\geq 2}\rho^{m+d-2}c_{n+m}\\
              =&\Theta(\rho^{-n}\Theta^{n}\rho^{d-1}c_0+\Theta\sum_{i=0}^{n-1}\rho^{-i}\Theta^{i}(R_0\rho^{d-1}c_{n-i}+U\sum_{m\geq 2}\rho^{m+d-2}c_{n-1-i+m}))\\
              &+R_0\rho^{d-1}c_{n+1}+U\sum_{m\geq 2}\rho^{m+d-2}c_{n+m}\\
              =&R_0U b_n+R_0\rho^{d-1}c_{n+1}+U\sum_{m\geq 2}\rho^{m-1}\rho^{d-1}c_{n+m}.
          \end{split}
      \end{equation*}
      Applying Lemma \ref{Lem-Technique-I} again (to $a_{n+1} = \rho^{n+1}$), we know that (\ref{Lem-H^1-small case-A}) is true for such a choice of $b_n$'s. In other words, we have $f_V(x) = \rho^dey$ as desired.

      Now we assume $d = 1$. By Corollary \ref{Cor-Explicit H^0}, we have
      \[\rH^0(\Gamma,M_0(V)) = \{\exp(\rho \Theta^{\prime}X)v\mid v\in V\}\cong V.\]
      Since $\exp(\rho \Theta^{\prime}X)$ is invertible in $\GL_r(A[\rho X]^{\wedge}_{\pd})$, as mentioned in Example \ref{Exam-V_{Theta/rho}}, we know that 
      \[\rH^0(\Gamma,M_0(V))\otimes_AA[\rho X]^{\wedge}_{\pd} = M_0(V).\]
      As $\rH^0(\Gamma,M_0(V))$ is a finite free $A$-module, we see that
      \[\rH^1(\Gamma,M_0(V)) = M_0(V)^{\Gamma}\otimes\rH^1(\Gamma,A[\rho X]^{\wedge}_{\pd}).\]
      Then we can conclude by using Proposition \ref{Prop-Gamma-cohomology of pd ring}.
  \end{proof}

  At the end of this subsection, we summarize our calculations as follows:
  \begin{prop}\label{Prop-CohomologySummary-general}
      Let $A$ be a $p$-complete $p$-torsion free $\calO_C$-algebra with a topologically nilpotent element $e\in A$ and $\rho\in\calO_C$ satisfied $\nu_p(\rho)\geq\frac{1}{p-1}$.
      Let $V$ be a log nilpotent $A$-representation of $\Gamma$ with $\Theta\in \rM_r(A)$ as above. Assume $\alpha\in\bZ[\frac{1}{p}]\cap[0,1)$.
      \begin{enumerate}
          \item[(1)] When $\alpha\neq 0$, we have $\rH^0(\Gamma,M_{\alpha}(V)) = 0$ and $\rH^1(\Gamma,M_{\alpha}(V)) = M_{\alpha}(V)/(\zeta^{\alpha}-1)$.

          \item[(2)] Assume $\alpha = 0$ and $e=\zeta_p-1$. Then we have 
          \[\rH^0(\Gamma,M_{0}(V)) = \{\exp(\Theta X)v\mid v\in V_{\Theta/\rho}\}\cong V_{\Theta/\rho}.\]
          If there exists $n\geq 1$ such that $\Theta^n = \rho^n\Theta^{\prime}$ for some topologically nilpotent $\Theta^{\prime}$, then we have
          \[\rH^1(\Gamma,M_0(V))\cong \rH^1(\Gamma,M_0(V))[\rho^{n}(\zeta_p-1)].\]
          If moreover $n=1$, then the natural inclusion $M_0(V)^{\Gamma}\subset M_0(V)$ induces a $\Gamma$-equivariant isomorphism 
          \[M_0(V)^{\Gamma}\otimes_AA[\rho X]^{\wedge}_{\pd} = M_0(V)\]
          such that 
          \[\rH^1(\Gamma,M_0(V)) \cong M_0(V)/\rho (\zeta_p-1)M_0(V).\]
      \end{enumerate}
  \end{prop}
  \begin{proof}
      Item (1) is exactly Lemma \ref{Lem-Cogo-Error}. Item (2) follows from Lemma \ref{Lem-Coho-Principal-e=zeta-1}, Corollary \ref{Cor-Explicit H^0} and Lemma \ref{Lem-H^1-small case}.
  \end{proof}
\subsection{Some Technical lemmas}
 We collect some technical lemmas which have been used in previous subsections. Let $A$ be complete topological ring with a topologically nilpotent $\rho\in A$. Let $U\in M_r(A)$ be a topologically nilpotent matrix; that is, $\lim_{n\to+\infty}U^n = 0$.
  \begin{construction}\label{Construction-TechniqueLemma}
      Let $\{Q_m\}_{m\geq 0}$ be matrices in $M_r(A)$ satisfying $Q_0 = I$, $[Q_l,U]$ and $[Q_m,Q_n] = 0$ for any $l,n,m\geq 0$. Then for any $k\geq 0$ and any $m\geq 1$, we define $R_k\in \GL_r(A)$ and $Q_{m,k}, S_{m,k}\in M_r(A)$ as follows:
      \begin{enumerate}
          \item[(1)] $Q_{m,0} = Q_0$, $S_{m,0} = 0$ and $R_0 = I$.

          \item[(2)] For any $k\geq 0$, $Q_{m,k+1} = \sum_{l=1}^{m+1}Q_{l,k}Q_{m+1-l}$ and $R_{k+1} = R_k+U^{k+1}Q_{1,k}$.

          \item[(3)] $S_{m,k+1} = S_{m,k}+U^k Q_{m,k}$.
      \end{enumerate}
      Then $R:=\lim_{k\to+\infty}R_k$ and $S_m=\lim_{k\to+\infty}S_{m,k}$ are well-defined in $\GL_r(A)$ and $M_r(A)$, respectively. More precisely, we have $R = I+U\sum_{k\geq 0}U^kQ_{1,k}$ and $S_m = \sum_{k\geq 0}U^kQ_{m,k}$ such that for any $k\geq 0$, $R-R_k\in U^{k+1}M_r(A)$ and $S_m-S_{m,k}\in U^kM_r(A)$. Clearly, all matrices involved commute with each others.
  \end{construction}
  \begin{rmk}\label{Rmk-TechniqueLemma}
      Let $\calA = \bZ[X_1,X_2,\dots]$ be the polynomial ring generated by the free variables $\{X_n\}_{n\geq 1}$. For any monomial $X_{n_1}^{d_1}\cdots X_{n_r}^{d_r}\in \calA$, we define its height by $d_1n_1+\cdots+d_rn_r$. A polynomial $F\in \calA$ is called homogeneous of height $h$, if it is a sum of monomials of height $h$. Clearly, such an $F$ must belong to $\bZ[X_1,\dots,X_h]$ and has degree no more than $h$. We claim that for any $m\geq 1$ and $k\geq 0$, there is an homogeneous polynomial $F_{m,k}$ of height $m+k$ such that $Q_{m,k} = F_{m,k}(Q_1,\dots,Q_{m+k})$.

      To see the claim, we do induction on $k$. When $k = 0$, we put $F_{m,k} = X_m$ and then the claim holds true in this case. Now assume the claim holds true for some $k\geq 0$. Put $F_{m,k+1} = \sum_{l=1}^mF_{l,k}\cdot X_{m+1-l}+F_{m+1,k}$. Then by inductive hypothesis, $F_{m,k+1}$ is homogeneous of height $m+k+1$ and 
      \begin{equation*}
          \begin{split}
              F_{m,k+1}(Q_1,\dots,Q_{m+k+1}) &= \sum_{l=1}^mF_{l,k}(Q_1,\dots,Q_{l+k})Q_{m+1-l}+F_{m+1,k}(Q_1,\dots,Q_{m+k+1})\\
              &= \sum_{l=1}^{m+1}Q_{l,k}Q_{m+1-l} \\
              &= Q_{m,k+1}.
          \end{split}
      \end{equation*}
      So the claim holds true for $k+1$ as desired. We win! Clearly, $F_{m,k}$'s are independent of the choices of $A,\rho,U$ and $Q_{m}$'s.
  \end{rmk}
  \begin{lem}\label{Lem-Technique-I}
      Suppose that we have commuting matrices $\{Q_m\}_{m\geq 1}$ such that $[Q_m,U] = 0$ and vectors $\{b_n\}_{n\geq 0}$ in $A^r$. Let $R\in I+UM_r(A)(\subset\GL_r(A))$ and $\{S_m\}_{m\geq 1}\subset M_r(A)$ be the matrices given in Construction \ref{Construction-TechniqueLemma}. Then for any given $\{a_n\}_{n\geq 1}$ in $A^r$, the following are equivalent:
      \begin{enumerate}
          \item[(1)] For any $n\geq 0$, $\rho b_{n+1} = RUb_n+Ra_{n+1}+U\sum_{m\geq 2}S_m\rho^{m-1}a_{n+m}$.

          \item[(2)] For any $n\geq 0$, $\rho b_{n+1} = a_{n+1}+Ub_n+U\sum_{m\geq 1}\rho^mQ_mb_{n+m}$.
      \end{enumerate}
  \end{lem}
  \begin{proof}
      Let $R_k$'s and $Q_{m,k}$'s be the matrices introduced in Construction \ref{Construction-TechniqueLemma}. 
      
      We first exhibit how (2) implies (1).
      Assume we have $\rho b_{n+1} = a_{n+1}+U b_n+U\sum_{m\geq 1}\rho^mQ_mb_{n+m}$ for any $n\geq 0$.
      We claim that for any $k\geq 0$,
          \[\rho b_{n+1} = R_kU b_n+R_ka_{n+1}+U\sum_{m\geq 2}S_{m,k}\rho^{m-1}a_{n+m}+U^{1+k}\sum_{m\geq 1}Q_{m,k}\rho^mb_{n+m}.\]
      Granting this, Item (2) can be deduced by letting $k$ go to $+\infty$.%., we can conclude that for any $n\geq 0$, 

      Now, we are going to prove the claim by induction on $k$.
      The $k = 0$ case is trivial. Assume we have confirmed the claim for some $k\geq 0$. Then we have
      \begin{equation}\label{Equ-TechniqueLemma-I}
          \begin{split}
              \rho b_{n+1} = &R_kU b_n+R_ka_{n+1}+U\sum_{m\geq 2}S_{m,k}\rho^{m-1}a_{n+m}+U^{1+k}\sum_{m\geq 1}Q_{m,k}\rho^mb_{n+m}\\
              =& R_kU b_n+R_ka_{n+1}+U\sum_{m\geq 2}S_{m,k}\rho^{m-1}a_{n+m}\\
              &+U^{1+k}\sum_{m_1\geq 1}Q_{m_1,k}\rho^{m_1-1}(a_{n+m_1}+U b_{n+m_1-1}+U\sum_{m_2\geq 1}\rho^{m_2}Q_{m_2}b_{n+m_1-1+m_2})\\
              =& R_kU b_n+R_ka_{n+1}+U\sum_{m\geq 2}S_{m,k}\rho^{m-1}a_{n+m}+U^{1+k}\sum_{m\geq 1}Q_{m,k}\rho^{m-1}a_{n+m}\\
              &+U^{2+k}\sum_{m_1\geq 1}Q_{m_1,k}\rho^{m_1-1}b_{n+m_1-1}
              +U^{2+k}\sum_{m_1\geq 1}Q_{m_1,k}\rho^{m_1-1}\sum_{m_2\geq 1}\rho^{m_2}Q_{m_2}b_{n+m_1-1+m_2}\\
              =& (R_k+U^{1+k}Q_{1,k})U b_n+(R_k+U^{1+k}Q_{1,k})a_{n+1}+U\sum_{m\geq 2}(S_{m,k}+U^kQ_{m,k})\rho^{m-1}a_{n+m}\\
              &+U^{2+k}\sum_{m\geq 1}Q_{m+1,k}\rho^{m}b_{n+m}+U^{2+k}\sum_{m_1,m_2\geq 1}\rho^{m_1+m_2-1}Q_{m_1,k}Q_{m_2}b_{n+m_1+m_2-1}\\
              =& (R_k+U^{1+k}Q_{1,k})U b_n+(R_k+U^{1+k}Q_{1,k})a_{n+1}+U\sum_{m\geq 2}(S_{m,k}+U^kQ_{m,k})\rho^{m-1}a_{n+m}\\
              &+U^{2+k}\sum_{m\geq 1}Q_{m+1,k}\rho^{m}b_{n+m}
              +U^{2+k}\sum_{m\geq 1}\sum_{l=1}^mQ_{l,k}Q_{m+1-l}\rho^{m}b_{n+m}\\
              =& (R_k+U^{1+k}Q_{1,k})U b_n+(R_k+U^{1+k}Q_{1,k})a_{n+1}+U\sum_{m\geq 2}(S_{m,k}+U^kQ_{m,k})\rho^{m-1}a_{n+m}\\
              &+U^{2+k}\sum_{m\geq 1}(Q_{m+1,k}+\sum_{l=1}^mQ_{l,k}Q_{m+1-l})\rho^{m}b_{n+m}\\
              =& R_{k+1}U b_n+R_{k+1}a_{n+1}+U\sum_{m\geq 2}S_{m,k+1}\rho^{m-1}a_{n+m}+U^{2+k}\sum_{m\geq 1}Q_{m,k+1}\rho^mb_{n+m}.
          \end{split}
      \end{equation}
      So the claim holds true for $k+1$ as expected.

      It remains to deduce (2) from (1). We claim that for any $n\geq 0$ and any $d\geq 0$, 
      \[\rho b_{n+1}\equiv a_{n+1}+U b_n+U\sum_{m\geq 1}Q_m\rho^mb_{n+m}\mod U^{d+1}A^r.\]
      Granting this, we can conclude by letting $d$ go to $+\infty$.

      Now, we are going to prove the above claim by induction on $d$. 
      The $d=0$ case is trivial as $R\equiv I\mod U$. Now, we assume the claim holds true for some $d\geq 0$. As $R\equiv R_{d+1}\mod U^{d+2}$ and $S_{m}\equiv S_{m,d+1}\mod U^{d+1}$, we know that
      \[\rho b_{n+1} \equiv U R_{d+1}b_n+R_{d+1}a_{n+1}+U\sum_{m\geq 2}S_{m,d+1}\rho^{m-1}a_{n+m}+U^{d+2}\sum_{m\geq 1}Q_{m,d+1}\rho^{m}b_{n+m}\mod U^{d+2}.\]
      Applying (\ref{Equ-TechniqueLemma-I}) to above formula, we see that
      \begin{equation*}
          \begin{split}
              \rho b_{n+1} \equiv& U R_{d+1}b_n+R_{d+1}a_{n+1}+U\sum_{m\geq 2}S_{m,d+1}\rho^{m-1}a_{n+m}+U^{d+2}\sum_{m\geq 1}Q_{m,d+1}\rho^{m}b_{n+m}\mod U^{d+2}\\
              =& R_dU b_n+R_da_{n+1}+U\sum_{m\geq 2}S_{m,d}\rho^{m-1}a_{n+m}\\
              &+U^{1+d}\sum_{m_1\geq 1}Q_{m_1,d}\rho^{m_1-1}(a_{n+m_1}+U b_{n+m_1-1}+U\sum_{m_2\geq 1}\rho^{m_2}Q_{m_2}b_{n+m_1-1+m_2})\mod U^{d+2}.
          \end{split}
      \end{equation*}
      By inductive hypothesis, more precisely, by 
      \[\rho b_{n+m_1}\equiv a_{n+m_1}+U b_{n+m_1-1}+U\sum_{m_2\geq 1}\rho^{m_2}Q_{m_2}b_{n+m_1-1+m_2}\mod U^{d+1},\]
      we conclude that
      \[\rho b_{n+1}\equiv R_{d}U b_n+R_{d}a_{n+1}+U\sum_{m\geq 2}S_{m,d}\rho^{m-1}a_{n+m}+U^{d+1}\sum_{m\geq 1}Q_{m,d}\rho^mb_{n+m}\mod U^{d+2}.\]
      Since $R_0=I$ and $S_{m,0}=0$, by iteration, we finally conclude that
      \[\rho b_{n+1}\equiv a_{n+1}+U b_n+U\sum_{m\geq 1}Q_{m}\rho^mb_{n+m}\mod U^{d+2}\]
      as desired and therefore confirm the claim. We complete the proof.
  \end{proof}
  \begin{cor}\label{Cor-TechniqueLemma}
      Keep notations as in Lemma \ref{Lem-Technique-I}. Then $R = \sum_{m\geq 0}Q_m(RU)^m$.
  \end{cor}
  \begin{proof}
      Let $\calA$ and $F_{m,k}$'s be as in Remark \ref{Rmk-TechniqueLemma} and define $G_{k+1} = F_{1+k}$ for any $k\geq 0$. Put $G_0=I$. Then we have 
      \[R = \sum_{k\geq 0}U^kG_k(Q_1,\dots,Q_{1+k}).\]
      and therefore 
      \begin{equation*}
          \begin{split}
              R-\sum_{m\geq 0}Q_m(RU)^m=&\sum_{k\geq 0}U^kG_k(\underline Q)-\sum_{m\geq 0}Q_m(\sum_{k\geq 1}U^kG_{k-1}(\underline Q))^m\\
              =&\sum_{k\geq 1}U^kG_k(\underline Q)-\sum_{m\geq 1}Q_m\sum_{k_1,\dots,k_m\geq 1}(G_{k_1}\cdots G_{k_m})(\underline Q)U^{k_1+\cdots+k_m}\\
              =&\sum_{k\geq 1}U^k(G_k-\sum_{m=1}^k\sum_{k_1,\dots,k_m\geq 1,k_1+\cdots+k_m=k}X_mG_{k_1}\cdots G_{k_m})(\underline Q).
          \end{split}
      \end{equation*}
      So it is enough to show that for any $k\geq 1$,
      \[G_k=\sum_{m=1}^k\sum_{k_1,\dots,k_m\geq 1,k_1+\cdots+k_m=k}X_mG_{k_1}\cdots G_{k_m}.\]
      
      Now let $A = \calA[Y][[Z]]$ be the ring of formal power series with variable $Z$ over $\calA[Y]$, the free polynomial ring with variable $Y$ over $\calA$ and equip $A$ with $Z$-adic topology. Choose $\rho = Z$, $U = YZ$ and $Q_m = X_m$ for any $m\geq 1$. Let $R\in \calA[Y][[Z]]^{\times}$ be the element by applying Lemma \ref{Lem-Technique-I} in this situation. Since $\{b_n = (YR)^n\}_{n\geq 0}$ satisfies Lemma \ref{Lem-Technique-I} (1) (with $a_n = 0$ for any $n\geq 1$), we have
      \[UR=ZYR = \rho b_1 =Ub_0+U\sum_{m\geq 1}\rho^mQ_mb_{m} = Ub_0+U\sum_{m\geq 1}Q_m(UR)^mb_0 = U\sum_{m\geq 0}Q_m(UR)^m.\]
      Since $\calA[Y][[Z]]$ is a domain, we have 
      \[R = \sum_{m\geq 0}Q_m(UR)^m.\]
      By the same argument as above, we have
      \[0 = R-\sum_{m\geq 0}Q_m(UR)^m = \sum_{k\geq 1}U^k(G_k-\sum_{m=1}^k\sum_{k_1,\dots,k_m\geq 1,k_1+\cdots+k_m=k}X_mG_{k_1}\cdots G_{k_m}).\]
      Since for any $k\geq 1$, $G_k-\sum_{m=1}^k\sum_{k_1,\dots,k_m\geq 1,k_1+\cdots+k_m=k}X_mG_{k_1}\cdots G_{k_m}\in\calA$ and $U = YZ$, the result follows from comparing the coefficients of $U^k$'s.
  \end{proof}

  \begin{lem}\label{Lem-Uniqueness of R}
      Assume $\cap_{n\geq 0}U^nM_r(A) = 0$.
      Let $\{Q_m\}_{m\geq 0}$ and $U$ be as in Construction \ref{Construction-TechniqueLemma}. Let $\calB$ be the maximal commutative sub-$A$-algebra of $M_r(A)$ containing $U$ and $Q_m$'s. Then the equation
      $R = \sum_{m\geq 0}Q_m(UR)^m$ of $R$ has a unique solution $R = R_0$ in $I+U\calB$.
  \end{lem}
  \begin{proof}
      We consider a function $f:\calB\to \calB$ defined by sending each $X\in M_r(A)$ to $f(X) = \sum_{m\geq 1}Q_m(U(I+X))^m$. Then for any $X_1,X_2\in \calB$, 
      \begin{equation*}
          \begin{split}
              f(X_1)-f(X_2) = & \sum_{m\geq 1}Q_m((U(I+X_1))^m-(U(I+X_2))^m)\\
              =&U(X_1-X_2)\sum_{m\geq 1}Q_mU^{m-1}(\sum_{i=0}^{m-1}(I+X_1)^i(I+X_2)^{m-1-i})\in U(X_1-X_2)\calB
          \end{split}
      \end{equation*}
      Since $U$ is topologically nilpotent, we know $f$ is continuous on $\calB$ with the induced topology. Since $\calB$ is maximal, it is closed in $M_r(A)$.
      We claim that $f$ admits a unique fixed point in $\calB$. 
      
      The uniqueness is trivial: If $f(X_i) = X_i$ for $i=1,2$, then $X_1-X_2 = f(X_1)-f(X_2)\in U(X_1-X_2)\calB$, which implies that $X_1-X_2\in U^nM_r(A)$ for all $n\geq 1$ by iteration. Since $\cap_{n\geq 0}U^nM_r(A) = 0$, we have $X_1=X_2$. To confirm the claim, we are reduced to showing the existence of a fixed point of $f$. Since $f(0) = \sum_{m\geq 1}U^mQ_m\in U\calB$, we know that 
      \[f(f(0))-f(0)\in U(f(0)-0)\calB \subset U^2\calB.\]
      In general, we can prove by induction that for any $n\geq 2$, 
      \[f^{(n)}(0)-f^{(n-1)}(0)\in U^n\calB,\]
      where $f^{(n)}$ denotes the $n$-times self-composition of $f$. So $X_0:=\lim_{n\to +\infty}f^{(n)}(0)$ is a well-defined element in $\calB$. Since $f$ is continuous, we know that 
      \[f(X_0) = \lim_{n\to+\infty}f(f^{(n)}(0)) = \lim_{n\to+\infty}f^{(n+1)}(0) = X_0.\]
      So $X_0$ is a fixed point of $f$ as desired.

      Now we are prepared to prove the lemma. By Corollary \ref{Cor-TechniqueLemma}, there exists an $R\in I+U\calB$ such that $R = \sum_{m\geq 0}Q_m(UR)^m$. However, by the definition of $f$, we know that 
      \[f(R-I) = \sum_{m\geq 1}Q_m(UR)^m = R-I.\]
      So $R-I$ is the unique fixed point of $f$ on $\calB$. This implies the uniqueness of $R$.
  \end{proof}

  Now, we give a typical example we have used in the previous subsections.
  \begin{exam}\label{Exam-Special case for R}
      Let $A$ be a $p$-complete $p$-torsion free $\calO_C$-algebra, $\rho\in\frakm_C$ such that $\nu_p(\rho)\geq \frac{1}{p-1}$ and $e = \zeta_p-1$. Let $\Theta$ be a topologically nilpotent matrix in $M_r(A)$. We define 
      \begin{equation*}
          \begin{split}
              &P=\exp(-e\Theta) = \sum_{l\geq 0}(-1)^l\Theta^le^{[l]},\\
              &\Theta_0 = \frac{\exp(-eX)-I}{eX}\bigg|_{X = \Theta} = -\sum_{l\geq 0}(-1)^l\Theta^l\frac{e^{[l+1]}}{e},~\rm{and}\\
              &Q_m = \sum_{l=1}^m\sum_{m_1,\dots,m_l\geq 1,m_1+\cdots+m_l=m}\frac{e^{[m_1+1]}\cdots e^{[m_l+1]}}{e^l}(-1)^lP^l,~\forall m\geq 1.
          \end{split}
      \end{equation*}
      Put $U = -\Theta\Theta_0$. Then we know that $\cap_{n\geq 0}U^nM_r(A) = 0$ (as $A$ is $p$-adically separated and $\Theta$ is topologically nilpotent). Let $R_0$ be the invertible matrix obtained by applying Lemma \ref{Lem-Technique-I} in this case. Then we prove that $R_0 = -\Theta_0^{-1}$ as follows:

      Thanks to Lemma \ref{Lem-Uniqueness of R}, this amounts to that 
      \[-\Theta_0^{-1} = \sum_{m\geq 0}Q_m(-\Theta_0^{-1}U)^m = \sum_{m\geq 0}Q_m\Theta^m.\]
      Put $G(X) = \frac{\exp(eX)-1-eX}{eX} = \sum_{m\geq 1}\frac{e^{[m+1]}}{e}X^m$ and then $G(\Theta)$ is well-defined in $\Theta M_r(A)$. Note that 
      \begin{equation*}
          \begin{split}
              \sum_{m\geq 0}Q_m\Theta^m = & I+\sum_{m\geq 1}\sum_{l=1}^m\sum_{m_1,\dots,m_l\geq 1,m_1+\cdots+m_l=m}\frac{e^{[m_1+1]}\cdots e^{[m_l+1]}}{e^l}(-1)^lP^l\Theta^m\\
              =&I+\sum_{l\geq 1}\sum_{m_1,\dots,m_l\geq 1}\frac{e^{[m_1+1]}\cdots e^{[m_l+1]}}{e^l}(-1)^lP^l\Theta^{m_1+\cdots+m_l}\\
              =&I+\sum_{l\geq 1}(\sum_{m\geq 1}\frac{e^{[m+1]}}{e}\Theta^m)^l(-1)^lP^l\\
              =&\sum_{l\geq 0}(-PG(\Theta))^l.
          \end{split}
      \end{equation*}
      Since $-PG(\Theta)\in\Theta M_r(A)$, we conclude that 
      \[ \sum_{m\geq 0}Q_m\Theta^m = (I+PG(\Theta))^{-1} = -\Theta_0^{-1}\]
      as desired, where the last equality follows from that $1+\exp(-eX)G(X) = -\frac{\exp(-eX)-1}{eX}$.
  \end{exam}

\section{The $p$-adic Simpson correspondence}\label{Sec-MainResult}
   Let $\frakX$ be a smooth liftable formal scheme of dimension $d$ over $\calO_C$ with rigid analytic generic fiber $X$ and fix an $\rA_2$-lifting $\widetilde \frakX$ of $\frakX$. Put $\rho = \rho_K$. We will establish an integral $p$-adic Simpson correspondence in this section.
   \begin{dfn}\label{Dfn-SmallRep}
       Assume $a\geq \frac{1}{p-1}$. By an {\bf $a$-small $\OXp$-representation of rank $r$} on $X_{\proet}$, we mean a $p$-complete $p$-torsion free $\OXp$-module $\calL$ satisfying the condition that there exist an \'etale covering $\{\frakX_i\to\frakX\}_{i\in I}$ and rationals $b_i>b>a$ such that for any $i$, there exists an isomorphism
       \[(\calL\big|_{X_{i}}/\rho p^{b_i})^{\rm al}\cong ((\OXp\big|_{X_i}/\rho p^{b_i})^r)^{\rm al}\]
       of $(\widehat \calO_X^{+{\rm al}}/\rho p^{b_i})\big|_{X_i}$-modules, where $\widehat \calO_X^{+{\rm al}}$ denotes the almost integral structure sheaf and $X_i$'s are rigid generic fibers of $\frakX_i$'s. Denote by $\Rep^{\geq a}(\OXp)$ the category of $a$-small $\OXp$-representations on $X_{\proet}$. 
   \end{dfn}
   \begin{rmk}
       The existence of $b$ sitting strictly between $a$ and $b_i$'s is necessary, comparing with Definition \ref{Dfn-LocalSmallRep}. This is due to the ``almost" issue at the integral level (cf. Example \ref{Exam-LocalSmallRep}).
   \end{rmk}
   
   \begin{dfn}\label{Dfn-SmallHiggs}
       Assume $a\geq \frac{1}{p-1}$. By an {\bf $a$-small Higgs bundle of rank $r$} on $\frakX_{\et}$, we mean a pair $(\calH,\theta_{\calH})$ of a locally finite free $\calO_{\frakX}$-module $\calH$ of rank $r$ and an $\calO_{\frakX}$-linear morphism $\theta_{\calH}:\calH\to\calH\otimes_{\calO_{\frakX}}\rho\widehat \Omega_{\frakX}^1(-1)$\footnote{It is not hard to see that there exists a canonical isomorphism of $\calO_C$-modules $\rho\calO_C(-1)\otimes_{\calO_C}\xi_K\rA_2\cong\calO_C$, where $\xi_K$ is the generator of the kernel of surjection $\rA_{\inf,K}\to\calO_C$ (cf. \cite[Cor. 2.3 (1)]{Wang}). Therefore, one can understand $\rho\calO_C(-1)$ as a ramified version of Breuil--Kisin twist of $\calO_C$. In particular, when $K$ is absolutely unramified, we may choose $\rho =\zeta_p-1$ and then $\rho\calO_C(-1) = \calO_C\{-1\}$ is the usual Breuil--Kisin twist of $\calO_C$.} such that 
       \[\Ima(\theta_{\calH})\subset p^{b-\frac{1}{p-1}}\calH\otimes_{\calO_{\frakX}}\rho \widehat \Omega_{\frakX}^1(-1)\]
       for some $b>a$. Denote by $\HIG(\calH,\theta_{\calH})$ the Higgs complex
       \[\calH\xrightarrow{\theta_{\calH}}\calH\otimes_{\calO_{\frakX}}\rho\widehat \Omega_{\frakX}^1(-1)\xrightarrow{\theta_{\calH}}\calH\otimes_{\calO_{\frakX}}\rho^2\widehat \Omega_{\frakX}^2(-2)\to\cdots\to\calH\otimes_{\calO_{\frakX}}\rho^d\widehat \Omega_{\frakX}^d(-d)\]
       induced by $(\calH,\theta_{\calH})$. Denote by $\HIG^{\geq a}(\frakX)$ the category of $a$-small Higgs bundles on $\frakX_{\et}$.
   \end{dfn}

   \begin{construction}
       Let $(\calO\widehat \bC_{\pd}^+,\Theta)$ be the period sheaf together with the Higgs field defined in Proposition \ref{Prop-PeriodSheaf}. Assume $a\geq \frac{1}{p-1}$.
       \begin{enumerate}
           \item[(1)] For an $a$-small $\OXp$-representation $\calL$ on $X_{\proet}$, define 
           \[\Theta_{\calL} = \id_{\calL}\otimes\Theta:\calL\otimes_{\OXp}\calO\widehat \bC_{\pd}^+\to\calL\otimes_{\OXp}\calO\widehat \bC_{\pd}^+\otimes_{\calO_{\frakX}}\rho\widehat \Omega^1_{\frakX}(-1).\]
           Then $\Theta_{\calL}\wedge\Theta_{\calL} = 0$ and thus $(\calL\otimes_{\OXp}\calO\widehat \bC_{\pd}^+,\Theta_{\calL})$ induces a Higgs complex which is denoted by $\HIG(\calL\otimes_{\OXp}\calO\widehat \bC_{\pd}^+,\Theta_{\calL})$.

           \item[(2)] For an $a$-small Higgs bundle $(\calH,\theta_{\calH})$ on $\frakX_{\et}$, define 
           \[\Theta_{\calH} = \theta_{\calH}\otimes\id_{\calO\widehat \bC_{\pd}^+}+\id_{\calH}\otimes\Theta:\calH\otimes_{\calO_{\frakX}}\calO\widehat \bC_{\pd}^+\to\calH\otimes_{\calO_{\frakX}}\calO\widehat \bC_{\pd}^+\otimes_{\calO_{\frakX}}\rho\widehat \Omega^1_{\frakX}(-1).\]
           Then $\Theta_{\calH}\wedge\Theta_{\calH} = 0$ and thus $(\calH\otimes_{\calO_{\frakX}}\calO\widehat \bC_{\pd}^+,\Theta_{\calH})$ induces a Higgs complex which is denoted by $\HIG(\calH\otimes_{\calO_{\frakX}}\calO\widehat \bC_{\pd}^+,\Theta_{\calH})$.
       \end{enumerate}
   \end{construction}

   Before moving on, let us recall the $\rL\eta$-functor introduced in\cite[\S 6]{BMS18}. We only state its construction in the ring case for simplicity and the readers are referred to \textit{loc.cit.} for more details.

   Let $A$ be a ring with an invertible ideal $I\subset A$. For any complex $(C^{\bullet},\rd_C)$ of flat $A$-modules, we define a new complex $\eta_IC^{\bullet}$ by letting 
   \[(\eta_IC)^n = \{x\in I^nC^n\mid \rd x\in I^{n+1}C^{n+1}\}\]
   for any $n\in \bZ$. As shown in \cite[\S 6]{BMS18} $\eta_I$  induces a functor $\rL\eta_I:D(A)\to D(A)$ such that for any $K\in D(A)$ and any $n\in\bZ$, 
   \[\rH^n(\rL\eta_IK)\cong \rH^n(K)/\rH^n(K)[I].\]
   We will apply $\rL\eta$-operator when the ring (resp. ring topos) is $R^+$ (resp. $(\frakX_{\et},\calO_{\frakX})$) and $I$ is principally generated by $\rho(\zeta_p-1)$.
   
   Now, we are able to state our main result.

   \begin{thm}\label{Thm-IntegralSimpson}
      Assume $a\geq \frac{1}{p-1}$. Let $\frakX$ be a liftable smooth formal scheme of dimension $d$ over $\calO_C$. Fix an $\rA_2$-lifting of $\frakX$ and let $(\calO\widehat \bC_{\pd}^+,\Theta)$ be the corresponding period sheaf together with Higgs field as above. Let $\nu:X_{\proet}\to\frakX_{\et}$ be the natural projection of sites.
      \begin{enumerate}
          \item[(1)] For any $a$-small $\OXp$-representation $\calL$ of rank $r$ on $X_{\proet}$, we have that $\rR\nu_*(\calL\otimes_{\OXp}\calO\widehat \bC_{\pd}^+)$ is concentrated in degree $[0,d]$, and that the complex  $\rL\eta_{\rho(\zeta_p-1)}\rR\nu_*(\calL\otimes_{\OXp}\calO\widehat \bC_{\pd}^+)$ is concentrated in degree $0$ and coincides with $\nu_*(\calL\otimes_{\OXp}\calO\widehat \bC_{\pd}^+)[0]$ such that $(\nu_*(\calL\otimes_{\OXp}\calO\widehat \bC_{\pd}^+),\nu_*(\Theta_{\calL}))$ is an $a$-small Higgs bundle of rank $r$ on $\frakX_{\et}$.

          \item[(2)] For any $a$-small Higgs bundle $(\calH,\theta_{\calH})$ of rank $r$ on $\frakX_{\et}$, put $\calL=(\calH\otimes_{\calO_{\frakX}}\calO\widehat \bC_{\pd}^+)^{\Theta_{\calH} = 0}$. Then $\calL$ is an $a$-small $\OXp$-representation of rank $r$ on $X_{\proet}$.

          \item[(3)] The functors in Item (1) and Item (2) are quasi-inverses of each other and hence define an equivalence of categories
          \[\Rep^{\geq a}(\OXp)\simeq \HIG^{\geq a}(\frakX),\]
          which preserves tensor products and dualities. Moreover, if $\calL$ and $(\calH,\theta_{\calH})$ are the corresponding $a$-small $\OXp$-representation and $a$-small Higgs bundle via the above equivalence, then there exists an isomorphism of Higgs morphisms
          \[(\calL\otimes_{\OXp}\calO\widehat \bC_{\pd}^+,\Theta_{\calL})\cong (\calH\otimes_{\calO_{\frakX}}\calO\widehat \bC_{\pd}^+,\Theta_{\calH}).\]
      \end{enumerate}
   \end{thm}
   We will prove this theorem in \S\ref{SSec-ProofMainTheorme}. Now, we give a consequence of Theorem \ref{Thm-IntegralSimpson}.
   \begin{cor}\label{Cor-IntegralSimpson}
       Keep notations as in Theorem \ref{Thm-IntegralSimpson}. For any $a$-small $\OXp$-representation $\calL$ on $X_{\proet}$ with induced Higgs bundle $(\calH,\theta_{\calH})$ via the equivalence in Theorem \ref{Thm-IntegralSimpson} (3), we have a natural morphism
       \[\HIG(\calH,\theta_{\calH})\to \rR\nu_*\calL\]
       with cofiber killed by $(\rho(\zeta_p-1))^{\max\{d+1,2(d-1)\}}$.
   \end{cor}
   \begin{proof}
       As $(\calO\widehat \bC_{\pd}^+,\Theta)$ is a resolution of $\OXp$, we get a quasi-isomorphism $\calL\simeq\HIG(\calL\otimes_{\OXp}\calO\widehat \bC_{\pd}^+,\Theta_{\calL})$ and a fortiori a quasi-isomorphism
       \[\rR\nu_*\calL\simeq\rR\nu_*\HIG(\calL\otimes_{\OXp}\calO\widehat \bC_{\pd}^+,\Theta_{\calL}).\]
       Using Theorem \ref{Thm-IntegralSimpson} (3), we see that 
       \[\HIG(\calH,\theta_{\calH})\simeq \nu_*\HIG(\calL\otimes_{\OXp}\calO\widehat \bC_{\pd}^+,\Theta_{\calL}),\]
    which induces a natural morphism 
       \[\HIG(\calH,\theta_{\calH})\to \rR\nu_*\HIG(\calL\otimes_{\OXp}\calO\widehat \bC_{\pd}^+,\Theta_{\calL})\simeq\rR\nu_*\calL\]
       as desired. It remains to show its cofiber is killed by $(\rho(\zeta_p-1))^{\max\{d+1,2(d-1)\}}$

       Note that by the following spectral sequence
\[E_1^{i,j}=R^j\nu_*(\calL\otimes_{\OXp}\calO\widehat \bC^+_{\pd})\otimes_{\calO_{\frakX}}\rho^i\widehat \Omega^i_{\frakX}(-i)\Rightarrow R^{i+j}\nu_*\HIG(\calL\otimes_{\OXp}\calO\widehat \bC_{\pd}^+,\Theta_{\calL}),\]
we know that $\rR\nu_*\HIG(\calL\otimes_{\OXp}\calO\widehat \bC_{\pd}^+,\Theta_{\calL})$ is concentrated in degree $[0,2d]$, which implies the cofiber of the morphism $\HIG(\calH,\theta_{\calH})\to \rR\nu_*\calL$ is also concentrated in degree $[0,2d]$. Moreover for  each $j\geq 1$, we have $E_1^{i,j} = \rR^j\nu_*(\calL\otimes_{\OXp}\calO\widehat \bC^+_{\pd})\otimes_{\calO_{\frakX}}\rho^i\widehat \Omega^i_{\frakX}(-i)$ is killed by $\rho(\zeta_p-1)$. 
Thus for any $r\geq 1$, we have $E_r^{i,j}$ is killed by $\rho(\zeta_p-1)$ for any $j\geq 1$ as well because it is a sub-quotient of $E_1^{i,j}$.
Then by chasing the spectral sequence, we can show that this cofiber is killed by $(\rho(\zeta_p-1))^{\max\{d+1,2(d-1)\}}$. 

To see this, we first claim the kernel $K^i$ of the canonical map
\[\can^i:~E_2^{i,0}=\rH^i(\HIG(\calH,\theta_{\calH}))\to \rR^i\nu_*(\HIG(\calL\otimes_{\OXp}\calO\widehat \bC^+_{\pd},\Theta_{\calL}))\]
is killed by $(\rho(\zeta_p-1))^{i-1}$ when $1\leq i\leq d$ and is $0$ for other $i>d$. The latter is trivial as $\HIG(\calH,\theta_{\calH})$ is concentrated in degree $[0,d]$. For the former, we note that $E_{\infty}^{i,0} = ~E_{i+1}^{i,0}$ and the map $\can^i$ factors through the injection $E_{\infty}^{i,0}\subset \rR^i\nu_*(\HIG(\calL\otimes_{\OXp}\calO\widehat \bC^+_{\pd},\Theta_{\calL}))$. So one can conclude the claim by noting that $\Ker(E_{r}^{i,0}\to~E_{r+1}^{i,0}) = \Ima(E_r^{i-r,r-1}\to E_{r}^{i,0})$ is killed by $\rho (\zeta_p-1)$ for any $2\leq r\leq i$ (as it is a quotient of $E_r^{i-r,r-1}$ and $r-1\geq 1$).

We claim the cokernel $C^i$ of the map $\can^i$ is killed by $(\rho(\zeta_p-1))^i$ when $1\leq i\leq d$ and by $(\rho(\zeta_p-1))^{d+1}$ when $i>d$: For the former, put $\rH^i:=\rR^i\nu_*(\HIG(\calL\otimes_{\OXp}\calO\widehat \bC^+_{\pd},\Theta_{\calL}))$ and then it is equipped with a filtration
\[0 = \Fil^{i+1}(\rH^i)\subset E_{\infty}^{i,0} = 
\Fil^i(\rH^i)\subset\cdots\subset\Fil^0(\rH^i) = \rH^i\]
such that $\Fil^{j}(\rH^i)/\Fil^{j+1}(\rH^i) = E_{\infty}^{i,i-j}$ for any $0\leq j\leq i$. Noting that $E_{\infty}^{i,i-j}$ is killed by $\rho(\zeta_p-1)$ for any $0\leq j<i$, we have $\rho(\zeta_p-1)\Fil^j(\rH^i)\subset \Fil^{j+1}(\rH^i)$ for any $0\leq j<i$. By iteration, we conclude that $(\rho(\zeta_p-1))^i\rH^0\subset E_{\infty}^{i,0}$ which is equivalent to that $C^i$ is killed by $(\rho(\zeta_p-1))^i$ (as $\can^i$ factors through the injection $E_{\infty}^{i,0}\to\rH^i$). For the latter, we can conclude from the same argument above by furthermore noting that $E_{1}^{i,j} = 0$ and a fortiori $E_{\infty}^{i,j} = 0$ for any $i>d$.

Putting $K^i$ and $C^i$ together, we finally conclude that the cofiber of the natural morphism in the statement is killed by $(\rho(\zeta_p-1))^{\max\{2(d-1),d+1\}}$.
   \end{proof}

   \begin{rmk}\label{Rmk-Warnning}
       As $\rL\eta$-operator is not exact, one can not use the usual spectral sequence argument to compare $\HIG(\calH,\theta_{\calH})$ with $\rL\eta_{\rho(\zeta_p-1)}\rR\nu_*\calL \simeq \rL\eta_{\rho(\zeta_p-1)}\rR\nu_*(\HIG(\calL\otimes_{\OXp}\calO\widehat \bC^+_{\pd},\Theta_{\calL}))$ directly. We will discuss the comparison between $\HIG(\calH,\theta_{\calH})$ and $\rL\eta_{\rho(\zeta_p-1)}\rR\nu_*\calL$ in the next section \S \ref{Sec-PerfectComplex}. 
   \end{rmk}
   \begin{rmk}
       Thanks to Remark \ref{Rmk-Small} and Zariski--\'etale comparison theorem, one can replace the \'etale site $\frakX_{\et}$ by the Zariski site $\frakX_{\rm Zar}$ and all results in \S\ref{Sec-MainResult} and \S\ref{Sec-PerfectComplex} are still true.
   \end{rmk}
   \begin{rmk}
     For a semi-stable formal scheme $\frakX$ over $\calO_C$ in the sense of \cite{CK19} which admits a lifting over $\rA_2$ as log schemes, in \cite{SW}, the second author together with Mao Sheng constructed an analogue of $\calE_{\rho}^+$. By applying our approach, one may see that all results in this paper also hold true in the case of \cite{SW}.
   \end{rmk}

 \subsection{Local $p$-adic Simpson correspondence}\label{SSec-Local Simpson}
   This subsection is devoted to giving a local version of Theorem \ref{Thm-IntegralSimpson}.
   We keep notations as in Convention \ref{Convention-small}.
   \begin{dfn}\label{Dfn-LocalSmallRep}
       Let $A\in\{R^+, \widehat R_{\infty}^+\}$ and $a\in \bQ_{>\frac{1}{p-1}}$. By an {\bf $a$-small $A$-representation of $\Gamma$ of rank $r$}, we mean a finite free $A$-module $V$ of rank $r$ together with a continuous semi-linear action of $\Gamma$ such that there exists a $\Gamma$-equivariant isomorphism of $A$-modules
       \[V/\rho p^aV\cong (A/\rho p^aA)^r.\]
       We denote by $\Rep^{\geq a}_A(\Gamma)$ the category of $a$-small $A$-representations of $\Gamma$.
   \end{dfn}
   \begin{rmk}\label{Rmk-LocalSmallRep}
       Assume $b\in\bQ_{>\frac{1}{p-1}}$. Let $V\in \Rep^{\geq b}_{R^+}(\Gamma)$ and  $v_1,\dots, v_r$ form an $R^+$-basis of $V$. Let $P_i$ be the matrix of $\gamma_i\in \Gamma$ on $V$ with respect to the given basis. Then we know that $P_i$'s commute with each other and belong to $I+\rho p^b\rM_r(R^+)$. In particular, there are commuting matrices $\Theta_i$'s such that for any $1\leq i\leq d$,
           \[P_i = \exp(-(\zeta_p-1)\Theta_i).\]
           As $P_i\in I+\rho p^bM_r(R^+)$, we know that $\Theta_i\in \rho p^{b-\frac{1}{p-1}}\rM_r(R^+)$.
    \end{rmk}
    \begin{exam}\label{Exam-LocalSmallRep}
        Let $\calL$ be a $p$-complete $p$-torison free $\widehat \calO_U^+$-module such that there exists an almost isomorphism
        \[(\calL/\rho p^b)^{\rm al}\cong((\widehat \calO_U^+/\rho p^b)^r)^{\rm al}\]
        for some integer $r\geq 0$ and rational $b\in\bQ_{>\frac{1}{p-1}}$.
        Then \cite[Lem. 5.9]{Wang} implies that for any $\frac{1}{p-1}<a<b$, $\calL(U_{\infty})$ is an $a$-small $\widehat R_{\infty}^+$-representation of $\Gamma$ of rank $r$. However, this is not necessarily true if $a=b$. This justifies Definition \ref{Dfn-SmallRep} in the global case.
    \end{exam}
      
   The following decompletion theorem plays an important role in our theory.
   \begin{thm}\label{Thm-Decompletion}
       For any $a\in\bQ_{>\frac{1}{p-1}}$, the base-change along $R^+\to\widehat R_{\infty}^+$ induces an equivalence of categories
       \[\Rep^{\geq a}_{R^+}(\Gamma)\to\Rep^{\geq a}_{\widehat R_{\infty}^+}(\Gamma),\]
       which preserves tensor products and dualities. Moreover, for any $M\in \Rep^{\geq a}_{R^+}(\Gamma)$ with corresponding $M_{\infty}\in \Rep^{\geq a}_{\widehat R_{\infty}^+}(\Gamma)$, the natural inclusion $M\hookrightarrow M_{\infty}$ identifies $\rR\Gamma(\Gamma,M)$ with a direct summand of $\rR\Gamma(\Gamma,M_{\infty})$ whose complement is concentrated in degree $\geq 1$ and killed by $\zeta_p-1$.
   \end{thm}
   \begin{proof}
       Note that for $A\in\{R^+,\widehat R_{\infty}^+\}$, $a$-small $A$-representations of $\Gamma$ are exactly $(a+\nu_p(\rho))$-trivial $A$-representations of $\Gamma$ in the sense of \cite[Def. 3.3]{Wang}. Then \cite[Thm. 3.4 and Prop. 3.5]{Wang} apply.
   \end{proof}

   Before we state and prove our main theorem in this subsection, we need the following definition.
   \begin{dfn}\label{Dfn-LocalHiggs}
       Assume $a\in\bQ_{>\frac{1}{p-1}}$.
       By a {\bf Higgs module over $R^+$ of rank $r$}, we mean a finite free $R^+$-module $H$ of rank $r$ together with an $R^+$-linear morphism $\theta:H\to H\otimes_{R^+}\widehat \Omega_{R^+}^1(-1)$ such that $\theta\wedge\theta = 0$. We denote by $(H,\theta_H)$ the Higgs complex by $(H,\theta_H)$. A Higgs module $(H,\theta)$ is called {\bf $a$-small}, if $\theta$ is divided by $\rho p^{a-\frac{1}{p-1}}$; that is, 
       \[\Ima(\theta)\subset p^{a-\frac{1}{p-1}}H\otimes_{R^+}\rho \widehat \Omega^1_{R^+}(-1).\]
       Denote by ${\rm HIG}^{\geq a}(R^+)$ the category of $a$-small Higgs modules over $R^+$.
   \end{dfn}
   \begin{rmk}\label{Rmk-LocalHiggs}
       The definition of $a$-small Higgs modules is slightly different from that in \cite[Def. 4.2]{Wang}. Indeed, the functor $(H,\theta_H)\mapsto(H,(\zeta_p-1)\theta_H)$ induces an equivalence from the category ${\rm HIG}^{\geq a}(R^+)$ of $a$-small Higgs modules over $R^+$ defined as above to the category ${\rm HIG}^{\geq a}_W(R^+)$ of $a$-small Higgs modules over $R^+$ in \textit{loc.cit.}.
   \end{rmk}
   \begin{construction}\label{Construction-LocalSimpson}
       Let $\widehat S_{\pd}^+$ be as in Notation \ref{Notation-LocalPeriodSheaf} with the induced Higgs field $\Theta$.
       \begin{enumerate}
           \item[(1)] For any $\widehat R_{\infty}^+$-representation $V$ of $\Gamma$, put $\Theta_V = \id_V\otimes\Theta$ and then $\Theta_V\wedge\Theta_V=0$. Denote by $\HIG(V\otimes_{\widehat R_{\infty}^+}\widehat S_{\pd}^+,\Theta_V)$ the Higgs complex induced by $(V\otimes_{\widehat R_{\infty}^+}\widehat S_{\pd}^+,\Theta_V)$.

           \item[(2)] For any Higgs module $(H,\theta_H)$ over $R^+$, put $\Theta_H = \theta_H\otimes\id_{\widehat S_{\pd}^+}+\id_H\otimes\Theta$ and then $\Theta_H\wedge\Theta_H = 0$. Denote by $\HIG(H\otimes_{R^+}\widehat S_{\pd}^+,\Theta_H)$ the Higgs complex induced by $(H\otimes_{R^+}\widehat S_{\pd}^+,\Theta_H)$.
       \end{enumerate}
   \end{construction}
   
   \begin{thm}\label{Thm-LocalSimpson}
       Assume $a>\frac{1}{p-1}$. Then the functor $V\mapsto ((V\otimes_{\widehat R_{\infty}^+}\widehat S_{\pd}^+)^{\Gamma},\Theta_V\big|_{(V\otimes_{\widehat R_{\infty}^+}\widehat S_{\pd}^+)^{\Gamma}})$ induces an equivalence of categories
       \[\Rep^{\geq a}_{\widehat R_{\infty}^+}(\Gamma)\xrightarrow{\simeq}\HIG^{\geq a}(R^+)\]
       with a quasi-inverse given by the functor $(H,\theta_H)\mapsto (H\otimes_{R^+}\widehat S_{\pd}^+)^{\Theta_H = 0}$, which preserves tensor products and dualities. Moreover, if $V$ and $(H,\theta_H)$ correspond to each other under the above equivalence, then we have an isomorphism of Higgs complex,
       \[\HIG(V\otimes_{\widehat R_{\infty}^+}\widehat S_{\pd}^+,\Theta_V)\cong\HIG(H\otimes_{R^+}\widehat S_{\pd}^+,\Theta_H)\]
       and a quasi-isomorphism
       \begin{equation}\label{Equ-LocalSimpson-I}
           H=(V\otimes_{\widehat R_{\infty}^+}S_{\pd}^+)^{\Gamma}\simeq \rL\eta_{\rho(\zeta_p-1)}\rR\Gamma(\Gamma,V\otimes_{\widehat R_{\infty}^+}\widehat S_{\pd}^+).
       \end{equation}
   \end{thm}
   \begin{proof}
       We may argue as in the proof of \cite[Thm. 4.3]{Wang}. Let $V_0$ be the $a$-small $R^+$-representation corresponding to $V$ in the sense of Theorem \ref{Thm-Decompletion}. Then we have a $\Gamma$-equivariant isomorphism 
       \[V_0\otimes_{R^+}\widehat S_{\pd}^+ = V\otimes_{\widehat R_{\infty}^+}\widehat S_{\pd}^+.\]
       Fix an $R^+$-basis $v_1,\dots,v_r$ of $V_0$ and let $\Theta_i$'s be as given in Remark \ref{Rmk-LocalSmallRep} (for $V = V_0$ and $b=a$). Then $\Theta_i$'s belong to $\rho p^{a-\frac{1}{p-1}}\rM_r(R^+)$. For any $0\leq i\leq d$, let $A_i=R^+[\rho Y_1,\dots,\rho Y_i]^{\wedge}_{\pd}$. Then there exists a $\Gamma$-equivariant isomorphism which follows from Equation (\ref{R-decomp}) and Notation \ref{Notation-LocalPeriodSheaf}
       \begin{equation}\widehat S_{\pd}^+\cong \widehat {\bigoplus_{\alpha_1,\dots,\alpha_d\in\bN[1/p]\cap[0,1)}}A_dT_1^{\alpha_1}\cdots T_d^{\alpha_d}.\end{equation}
       
       We first claim that the natural inclusion $V_0\otimes_{R^+}A_d\to V_0\otimes_{R^+}\widehat S_{\pd}^+=V\otimes_{\widehat R_{\infty}^+}\widehat S_{\pd}^+$ identifies $\rR\Gamma(\Gamma,V_0\otimes_{R^+}A_d)$ with a direct summand of $\rR\Gamma(\Gamma,V\otimes_{\widehat R_{\infty}^+}\widehat S_{\pd}^+)$ whose complement is concentrated in positive degrees and killed by $(\zeta_p-1)$. Indeed, for any $\underline \alpha = (\alpha_1,\dots,\alpha_d)\in(\bN[1/p]\cap[0,1))^d$ with $\alpha_i\neq 0$ for some $i$, we deduce from Proposition \ref{Prop-CohomologySummary} (1)  that $\rR\Gamma(\Zp\gamma_i,V_0\otimes_{R^+}A_dT_1^{\alpha_1}\cdots T_d^{\alpha_d})$ is concentrated in positive degrees and killed by $(\zeta_p-1)$. By Hochschild--Serre spectral sequence, we see that $\rR\Gamma(\Gamma,V_0\otimes_{R^+}A_dT_1^{\alpha_1}\cdots T_d^{\alpha_d})$ is also concentrated in positive degrees and killed by $(\zeta_p-1)$. Now write $\widehat S_{\pd}^+=A_d\oplus \widehat S_{\pd}^{+,non-int}$, where 
       \[
       \widehat S_{\pd}^{+,not-int}=\widehat {\bigoplus_{\alpha_1,\dots,\alpha_d\in\bN[1/p]\cap[0,1)\atop\text{not all zero}}}A_dT_1^{\alpha_1}\cdots T_d^{\alpha_d}.\] 
       Then
       \[
       \rR\Gamma(\Gamma,V_0\otimes_{R^+}\widehat S_{\pd}^{+,not-int})=\widehat{\bigoplus_{\alpha_1,\dots,\alpha_d\in\bN[1/p]\cap[0,1)\atop\text{not all zero}}}\rR\Gamma(\Gamma,V_0\otimes_{R^+}A_dT_1^{\alpha_1}\cdots T_d^{\alpha_d}).
       \]
       As $\rR\Gamma(\Gamma,V\otimes_{R^+}A_dT_1^{\alpha_1}\cdots T_d^{\alpha_d})$ is killed by $(\zeta_p-1)$, which implies it is already derived $p$-complete, we see 
       \[
       \rR\Gamma(\Gamma,\widehat S_{\pd}^{+,not-int})=\bigoplus_{\alpha_1,\dots,\alpha_d\in\bN[1/p]\cap[0,1)\atop\text{not all zero}}\rR\Gamma(\Gamma,V_0\otimes_{R^+}A_dT_1^{\alpha_1}\cdots T_d^{\alpha_d}).\]

       Noting that $\rR\Gamma(\Gamma,V_0\otimes_{R^+}\widehat S_{\pd}^{+})\simeq \rR\Gamma(\Gamma,V_0\otimes_{R^+}\widehat S_{\pd}^{+,not-int})\oplus \rR\Gamma(\Gamma,V_0\otimes_{R^+}A_d)$, the claim then follows.

       Now, we are going to prove that
       \begin{equation}\label{Equ-ExplicitCorrespondence}
         (V_0\otimes_{R^+}A_d)^{\Gamma}= \exp(\Theta_1Y_1+\cdots+\Theta_dY_d)V_0,
       \end{equation}
       where the right hand side means the image of $V_0$ in $V_0\otimes_{R^+}A_d$ under the action of the operator 
       \[\exp(\Theta_1Y_1+\cdots+\Theta_dY_d):=\sum_{n_1,\cdots,n_d\geq 0}(\rho^{-1}\Theta_1)^{n_1}\cdots(\rho^{-1} \Theta_d)^{n_d}(\rho Y_1)^{[n_1]}\cdots (\rho Y_d)^{[n_d]}\in \GL_r(A_d)\]
       and that it is a finite free $R^+$-module of rank $r$ such that
       \[(V_0\otimes_{R^+}A_d)^{\Gamma}\otimes_{R^+}A_d\cong V_0\otimes_{R^+}A_d.\] 
       
       Indeed, note that $A_i = A_{i-1}[\rho Y_i]_{\pd}^{\wedge}$ for any $i\geq 1$. Then by Proposition \ref{Prop-CohomologySummary} (2), we have
       \[\begin{split}  
         (V_0\otimes_{R^+}A_d)^{\Gamma} & = (V_0\otimes_{R^+}A_d)^{\gamma_1=\cdots=\gamma_d=1}\\
         &= ((V_0\otimes_{R^+}A_{d-1}[\rho Y_d]^{\wedge}_{\pd})^{\gamma_d=1})^{\gamma_1=\cdots=\gamma_{d-1}=1}\\
         & =(\exp(\Theta_dY_d)V_0\otimes_{R^+}A_{d-1}))^{\gamma_1=\cdots=\gamma_{d-1}=1}\\
         & = \cdots\\
         & = \exp(\Theta_1Y_1+\cdots+\Theta_dY_d)V_0.
       \end{split}\]
       As $\exp(\Theta_1Y_1+\cdots+\Theta_dY_d)\in\GL_r(A_d)$, we have $(V_0\otimes_{R^+}A_d)^{\Gamma}\otimes_{R^+}A_d\cong V_0\otimes_{R^+}A_d$ as desired. 

       As a consequence, we know that 
       \[\rR\Gamma(\Gamma,V_0\otimes_{R^+}A_d)\simeq (V_0\otimes_{R^+}A_d)^{\Gamma}\otimes_{R^+}\rR\Gamma(\Gamma,A_d).\]
       By Lemma \ref{Lem-Kunneth}, we have 
       \[(V_0\otimes_{R^+}A_d)^{\Gamma}\simeq \rL\eta_{\rho(\zeta_p-1)}\rR\Gamma(\Gamma,V_0\otimes_{R^+}A_d).\]
       Combining this with the first claim above, we get the quasi-isomorphism (\ref{Equ-LocalSimpson-I})
       \begin{equation*}
           (V_0\otimes_{R^+}A_d)^{\Gamma}\simeq \rL\eta_{\rho(\zeta_p-1)}\rR\Gamma(\Gamma,V\otimes_{\widehat R_{\infty}^+}\widehat S_{\pd}^+).
       \end{equation*}
       Recall by Corollary \ref{Cor-LocalPeriodSheaf} together with Notation \ref{Notation-LocalPeriodSheaf}, the Higgs field $\Theta:\widehat S_{\pd}^+\to \widehat S_{\pd}^+\otimes_{R^+}\widehat \Omega^1_{R^+}\{-1\}$ is given by 
       \[\Theta = \sum_{i=1}^d\frac{\partial}{\partial Y_i}\otimes\frac{\rd\log T_i}{t}:\widehat R_{\infty}^+[\rho Y_1,\dots,\rho Y_d]^{\wedge}_{\pd}\to\oplus_{i=1}^d(\widehat R_{\infty}^+[\rho Y_1,\dots,\rho Y_d]^{\wedge}_{\pd}\cdot\frac{\rho \dlog T_i}{t})\] 
       via the isomorphisms 
       \[\widehat S_{\pd}^+\cong \widehat R_{\infty}^+[\rho Y_1,\dots,\rho Y_d]^{\wedge}_{\pd} \text{ and } \widehat \Omega^1_{R^+}\{-1\}\cong \oplus_{i=1}^d(R^+\cdot\frac{\rho \dlog T_i}{t}).\]
       Then we know that
       \[\Theta_V\big|_{(V\otimes_{\widehat R_{\infty}^+}\widehat S_{\pd}^+)^{\Gamma}} = \sum_{i=1}^d\Theta_i\otimes\frac{\rd\log T_i}{t},\]
       which is divided by $\rho p^{a-\frac{1}{p-1}}$. Therefore $(H,\theta_H):=((V\otimes_{\widehat R_{\infty}^+}\widehat S_{\pd}^+)^{\Gamma},\Theta_V\big|_{(V\otimes_{\widehat R_{\infty}^+}\widehat S_{\pd}^+)^{\Gamma}})$ is an $a$-small Higgs module of rank $r$ over $R^+$ as desired. Moreover, we have
       \[H\otimes_{R^+}\widehat S_{\pd}^+ \cong (V\otimes_{\widehat R_{\infty}^+}\widehat S_{\pd}^+)^{\Gamma}\otimes_{R^+}A_d\otimes_{A_d}\widehat S_{\pd}^+ \cong V_0\otimes_{R^+}A_d\otimes_{A_d}\widehat S_{\pd}^+\cong V\otimes_{\widehat R_{\infty}^+}\widehat S_{\pd}^+.\]
       Via the above isomorphism, for any $v\in V$ and any $f\in \widehat S_{\pd}^+$, we have
       \[\begin{split}
           &\Theta_H(\exp(\Theta_1Y_1+\cdots+\Theta_dY_d)v\otimes f) \\
           =&\sum_{i=1}^d(\Theta_i \exp(\Theta_1Y_1+\cdots+\Theta_dY_d)v\otimes f+\exp(\Theta_1Y_1+\cdots+\Theta_dY_d)v\otimes \frac{\partial f}{\partial Y_i})\otimes\frac{\rd\log T_i}{t}\\
           =&\sum_{i=1}^d\frac{\partial}{\partial Y_i}(\exp(\Theta_1Y_1+\cdots+\Theta_dY_d)f)v\otimes\frac{\rd\log T_i}{t}\\
           =&\Theta_V(v\otimes\exp(\Theta_1Y_1+\cdots+\Theta_dY_d)f),
       \end{split}\]
       which implies the isomorphism of Higgs complexes
       \[\HIG(V\otimes_{\widehat R_{\infty}^+}\widehat S_{\pd}^+,\Theta_V)\cong\HIG(H\otimes_{R^+}\widehat S_{\pd}^+,\Theta_H).\]
       
       Therefore, the functor $V\mapsto ((V\otimes_{\widehat R_{\infty}^+}\widehat S_{\pd}^+)^{\Gamma},\Theta_V\big|_{(V\otimes_{\widehat R_{\infty}^+}\widehat S_{\pd}^+)^{\Gamma}})$ satisfies all desired conditions. A standard argument shows that this functor preserves tensor products and dualities.

       It remains to prove this functor is an equivalence. For this, we claim that for any $a$-small Higgs module $(H,\theta_H)$ of rank $r$ over $R^+$ with an $R^+$-basis $h_1,\dots,h_r$, if we write 
       \[\theta_H = \sum_{i=1}^d\Theta_i\otimes\frac{\rd\log T_i}{t}\]
       with respect to the given basis, then $V:=(H\otimes_{R^+}\widehat S_{\pd}^+)^{\Theta_H = 0}$ is a finite free $\widehat R_{\infty}^+$-module of rank $r$ and the matrix of $\gamma_i$ is $\exp(-(\zeta_p-1)\Theta_i)$ under a certain basis of $V$.
       
       Granting this, as $\rho p^{a-\frac{1}{p-1}}$ divides $\Theta_i$'s, we know $V$ is an $a$-small $\widehat R_{\infty}^+$-representation. Then by the above discussion, we know that $(H,\theta_H)\cong ((V\otimes_{\widehat R_{\infty}^+}\widehat S_{\pd}^+)^{\Gamma},\Theta_V\big|_{(V\otimes_{\widehat R_{\infty}^+}\widehat S_{\pd}^+)^{\Gamma}})$ and hence the rule $(H,\theta_H)\mapsto (H\otimes_{R^+}\widehat S_{\infty}^+)^{\Theta_H = 0}$ induces the quasi-inverse of the functor constructed above.  
       
       Now, we prove the above claim. As $\rho p^{a-\frac{1}{p-1}}$ divides $\Theta_i$'s, we know that $\exp(-\Theta_1Y_1-\cdots-\Theta_dY_d)\in\GL_r(\widehat S_{\pd}^+)$. Therefore, we have
       \[\exp(-\Theta_1Y_1-\cdots-\Theta_dY_d)(H\otimes_{R^+}\widehat R_{\infty}^+)\otimes_{\widehat R_{\infty}^+}\widehat S_{\pd}^+\cong H\otimes_{R^+}\widehat S_{\pd}^+,\]
       where $\exp(-\Theta_1Y_1-\cdots-\Theta_dY_d)$ is similarly defined as $\exp(\Theta_1Y_1+\cdots+\Theta_dY_d)$ in (\ref{Equ-ExplicitCorrespondence}).
       Then for any $f\in (\widehat S_{\pd}^+)^r$, we have
       \begin{equation*}
           \begin{split}
               &\Theta_H(\underline h\otimes\exp(-\Theta_1Y_1-\cdots-\Theta_dY_d)f)\\
               =&\underline h\otimes\sum_{i=1}^d(\Theta_i\exp(-\Theta_1Y_1-\cdots-\Theta_dY_d)f+\frac{\partial }{\partial Y_i}(\exp(-\Theta_1Y_1-\cdots-\Theta_dY_d)f))\otimes\frac{\rd\log T_i}{t}\\
               =&\underline h\otimes\sum_{i=1}^d\exp(-\Theta_1Y_1-\cdots-\Theta_dY_d)\frac{\partial f}{\partial Y_i}\otimes\frac{\rd\log T_i}{t}.
           \end{split}
       \end{equation*}
       Therefore, $\Theta_H(\underline h\otimes\exp(-\Theta_1Y_1-\cdots-\Theta_dY_d)f) = 0$ amounts to that $\frac{\partial f}{\partial Y_i} = 0$ for all $i$. By Corollary \ref{Cor-LocalPeriodSheaf}, this is equivalent to that $f\in(\widehat R_{\infty}^+)^d$. In other words, we have
       \begin{equation}\label{Equ-LocalSimpson-II}
           (H\otimes_{R^+}\widehat S_{\infty}^+)^{\Theta_H = 0} = \exp(-\Theta_1Y_1-\cdots-\Theta_dY_d)(H\otimes_{R^+}\widehat R_{\infty}^+).
       \end{equation}
       As $\gamma_i(Y_j) = Y_j+\delta_{ij}(\zeta_p-1)$, we know that $\gamma_i$ acts on $V$ via $\exp(-(\zeta_p-1)\Theta_i)$ as desired.
   \end{proof}
   The following lemma was used above.
   \begin{lem}\label{Lem-Kunneth}
       Let $A_d$ be as in the proof of Theorem \ref{Thm-LocalSimpson}. Then for any $i\geq 0$, we have
       \[\rH^i(\Gamma,A_d) = \bigwedge_{R^+}^i((R^+[\rho Y]^{\wedge}_{\pd}/\rho(\zeta_p-1))^d).\]
       Similarly, for any $i\geq 0$, we have
       \[\rH^i(\Gamma,\widehat S_{\pd}^+) = \bigwedge_{R^+}^i((R^+[\rho Y]^{\wedge}_{\pd}/\rho(\zeta_p-1)\oplus\bigoplus_{\alpha\in \bN[1/p]\cap[0,1)}R^+[\rho Y]^{\wedge}_{\pd}/(\zeta^{\alpha}-1))^d).\]
   \end{lem}
   \begin{proof}
       As $A_d$ is a (derived) $p$-complete tensor products of $d$-copies of $R^+[\rho Y]^{\wedge}_{\pd}$ over $R^+$, the first part follows from K\"unneth formula and Proposition \ref{Prop-Gamma-cohomology of pd ring} (2). 
       Similarly, as $\widehat S_{\pd}^+$ is a (derived) $p$-complete tensor products of $\widehat \bigoplus_{\alpha\in\bN[1/p]\cap[0,1)}R^+[\rho Y_i]^{\wedge}_{\pd}T_i^{\alpha}$ over $R^+$ for $1\leq i\leq d$, the second part follows from the same argument (together with Proposition \ref{Prop-Gamma-cohomology of pd ring} (1)) as above.
   \end{proof}
   \begin{cor}\label{Cor-Kunnech}
       Let $V$ be an $a$-small $\widehat R_{\infty}^+$-representation of $\Gamma$. Then for any $i\geq 0$, we have
       \[\rH^i(\Gamma,V\otimes_{\widehat R_{\infty}^+}\widehat S_{\pd}^+) = (V\otimes_{\widehat R_{\infty}^+}\widehat S_{\pd}^+)^{\Gamma}\otimes_{R^+}\bigwedge_{R^+}^i((R^+[\rho Y]^{\wedge}_{\pd}/\rho(\zeta_p-1)\oplus\bigoplus_{\alpha\in \bN[1/p]\cap[0,1)}R^+[\rho Y]^{\wedge}_{\pd}/(\zeta^{\alpha}-1))^d).\]
       In particular, $\rH^i(\Gamma,V\otimes_{\widehat R_{\infty}^+}\widehat S_{\pd}^+)$ has no $\frakm_C$-torsion.
   \end{cor}
   \begin{proof}
       By Theorem \ref{Thm-LocalSimpson}, we have a $\Gamma$-equivariant isomorphism
       \[V\otimes_{\widehat R_{\infty}^+}\widehat S_{\pd}^+\cong (V\otimes_{\widehat R_{\infty}^+}\widehat S_{\pd}^+)^{\Gamma}\otimes_{R^+}\widehat S_{\pd}^+.\]
       As $(V\otimes_{\widehat R_{\infty}^+}\widehat S_{\pd}^+)^{\Gamma}$ is finite free, the desired isomorphism follows from Lemma \ref{Lem-Kunneth} directly. For the last sentence, as $(V\otimes_{\widehat R_{\infty}^+}\widehat S_{\pd}^+)^{\Gamma}$ is a finite free $R^+$-module, it suffices to prove that for any $\lambda\in\frakm_C$, $R^+[\rho Y]^{\wedge}_{\pd}/\lambda$ has no $\frakm_C$-torsion. As $\calO_C/\lambda$ has no $\frakm_C$-torsion, we can conclude by noting that $R^+[\rho Y]^{\wedge}_{\pd}$ is a topologically free $\calO_C$-module, following from \cite[Lem. 8.10]{BMS18}.
   \end{proof}

   \begin{rmk}[Theorem \ref{Thm-LocalSimpson} vs \emph{\rm \cite[Thm. 4.3]{Wang}}]\label{Rmk-CompareLocalSimpson}
       Assume $a>\frac{1}{p-1}$. 
       Thanks to Remark \ref{Rmk-LocalHiggs} and Theorem \ref{Thm-LocalSimpson}, we get an equivalence 
       \[\Rep_{\widehat R_{\infty}^+}^{\geq a}(\Gamma)\to\HIG^{\geq a}(R^+)\xrightarrow{(H,\theta_H)\mapsto(H,(\zeta_p-1)\theta_H)}\HIG_W^{\geq a}(R^+)\]
       from the category of $a$-small $\widehat R_{\infty}^+$-representations of $\Gamma$ to the category of $a$-small Higgs modules in the sense of \cite[Def. 4.2]{Wang}. We claim this functor coincides with that given in \cite[Thm. 4.3]{Wang}.
       
       To see this, let $\widehat S_{\infty,\rho}^+$ and $S_{\infty}^{\dagger,+}$ be the evaluations of $\calO\widehat \bC^+_{\rho}$ and $\calO\bC^{\dagger,+}$ at $U_{\infty}$. Then we have $S_{\infty}^{\dagger,+}\subset \widehat S_{\infty,\rho}^+$ and the morphism $\iota:\widehat S_{\pd}^+\to \widehat S_{\infty,\rho}^+$ induced by $\iota_{PS}$ sends $\rho Y_i\in \widehat S_{\pd}^+$ to $\rho(\zeta_p-1) Y_i\in \widehat S_{\infty,\rho}^+$ (cf. Proposition \ref{Prop-ComparePeriodSheaf}). Let $V$ be an $a$-small $\widehat R_{\infty}^+$-representation of $\Gamma$ with corresponding $a$-small $ R ^+$-representation $V_0$ in the sense of Theorem \ref{Thm-Decompletion}. Let $(H_1,\theta_1)$ and $(H_2,\theta_2)$ be the Higgs modules in $\HIG^{\geq a}_W(R^+)$ induced by the functors defined above and in \cite[Thm. 4.3]{Wang}, respectively. Fix an $R^+$-basis of $V_0$ and assume $\gamma_i\in\Gamma$ acts on $V_0$ via the matrix $\exp(-(\zeta_p-1)\Theta_i)$. By (\ref{Equ-LocalSimpson-I}), we have 
       \[(H_1,\theta_1) = (\exp(\sum_{i=1}^d\Theta_iY_i)V_0,\sum_{i=1}^d(\zeta_p-1)\Theta_i\otimes\frac{\rd\log T_i}{t}),\]
       where $\rho Y_i$'s are contained in $ \widehat S_{\pd}^+$. Applying $\iota$, we finally get that
       \[(H_1,\theta_1) = (\exp(\sum_{i=1}^d(\zeta_p-1)\Theta_iY_i)V_0,\sum_{i=1}^d(\zeta_p-1)\Theta_i\otimes\frac{\rd\log T_i}{t}) = (\prod_{i=1}^d\gamma_i^{-Y_i}V_0,\sum_{i=1}^d-\log\gamma_i\otimes\frac{\rd\log T_i}{t}),\]
       where $\rho Y_i$'s are contained in $\widehat S_{\infty,\rho}^+$. On the other hand, by \cite[Rem. 4.10, Prop. 4.13]{Wang}, we have
       \[(H_2,\theta_2) = (\prod_{i=1}^d\gamma_i^{-Y_i}V_0,\sum_{i=1}^d-\log\gamma_i\otimes\frac{\rd\log T_i}{t}),\]
       where $\rho Y_i$'s are contained in $\widehat S_{\infty,\rho}^+$. Therefore, we have $(H_1,\theta_1) = (H_2,\theta_2)$ as desired.
   \end{rmk}
   
  In the local case, we also have the following generalisation of local calculations in \cite[\S 8]{BMS18}.
   \begin{lem}\label{Lem-Local L-eta}
       Assume $\lambda\in\calO_C$ satisfying $0\leq\nu_p(\lambda)\leq \nu_p(\rho)$ and $a\geq\frac{1}{p-1}$.
       
    For any $a$-small $\widehat R_{\infty}^+$-representation $V$ of $\Gamma$, the complex $\rL\eta_{(\zeta_p-1)\lambda}\rR\Gamma(\Gamma,V)$ of $R^+$-modules is perfect and concentrated in degree $[0,d]$ with $p$-torsion free $\rH^0$.
   \end{lem}
   \begin{proof}
       Let $V_0$ be the corresponding $\Gamma$-representation over $R^+$ as above. Then the ``moreover'' part of Theorem \ref{Thm-Decompletion} induces a quasi-isomorphism
       \[\rL\eta_{\lambda}\rR\Gamma(\Gamma,V_0)\xrightarrow{\simeq}\rL\eta_{\lambda}\rR\Gamma(\Gamma,V).\]
       Assume $\gamma_i -1= \exp(-(\zeta_p-1)\Theta_i)-1= (\zeta_p-1)\Theta_iF(\Theta_i)$, where 
       \begin{equation}\label{Equ-F}
           F(X) := \frac{\exp(-(\zeta_p-1)X)-1}{(\zeta_p-1)X} = \sum_{n\geq 0}(-1)^{n+1}\frac{(\zeta_p-1)^n}{(n+1)!}X^n.
       \end{equation}
       As $\Theta_i$'s are topologically nilpotent, we know $F(\Theta_i)\in\GL_r(R^+)$. Therefore, we have quasi-isomorphisms
       \[\rR\Gamma(\Gamma,V_0)\simeq \rK(V_0;\gamma_1-1,\dots,\gamma_d-1)\simeq \rK(V_0; (\zeta_p-1)\Theta_1,\dots,(\zeta_p-1)\Theta_d).\]
       Now, it follows from \cite[Lem. 7.9]{BMS18} that
       \[\rL\eta_{(\zeta_p-1)\lambda}\rR\Gamma(\Gamma,V_0)\simeq \rK(V_0;\lambda^{-1}\Theta_1,\dots,\lambda^{-1}\Theta_d)\]
       is a perfect complex (note that $a$-smallness of $V$ implies that $\Theta_i$'s belong to $\rho\rM_d(R^+)$).
   \end{proof}
   \begin{rmk}
       The quasi-isomorphism in Lemma \ref{Lem-Local L-eta} depends on the local datum (more precisely, the choice of chart on $R^+$).
   \end{rmk}

 \subsection{Proof of Theorem \ref{Thm-IntegralSimpson}}\label{SSec-ProofMainTheorme}
  We focus on the proof of Theorem \ref{Thm-IntegralSimpson} in this section. The idea is similar to that in the proof of \cite[Thm. 4.3]{Wang}. We start with the following lemma:
  \begin{lem}\label{Lem-Key}
      Let $\frakU\in\frakX_{\et}$ be affine small with $U$ and $U_{\infty}$ as in Convention \ref{Convention-small}. Let $\calL$ be a $p$-complete $p$-torsion free $\OXp$-module such that 
      \[(\calL/p^c)^{\rm al}\big|_U \cong ((\OXp/p^c)^l)^{\rm al}\big|_U\]
      for some $c>0$ and $l\in \bN$.
      Then for any $\calP\in \{\calO\bC^+_{\pd}, \calO\widehat \bC^+_{\pd}\}$, there is a natural map from the group cohomology to the pro-\'etale cohomology
      \[\rR\Gamma(\Gamma,(\calL\otimes_{\OXp}\calP)(U_{\infty}))\to \rR\Gamma(U_{\proet},(\calL\otimes_{\OXp}\calP)\big|_{U})\]
      which is an almost quasi-isomorphism, and furthermore an isomorphism in degree $0$.
  \end{lem}
  \begin{proof}
      We only construct the natural map 
      \[\rR\Gamma(\Gamma,(\calL\otimes_{\OXp}\calP)(U_{\infty}))\to \rR\Gamma(U_{\proet},(\calL\otimes_{\OXp}\calP)\big|_{U})\]
      here, and the rest can be deduced from the same argument in the proof of \cite[Lem. 5.11]{Wang}.

      First, we claim that for any affinoid perfectoid $V\in U_{\infty,\proet}$ and any $i\geq 1$, the 
      $\rH^i(V,\calL\otimes_{\OXp}\calP)$ is almost vanishing. Note that $(\calO\bC_{\pd}^+)\big|_{U_{\infty}}\cong (\OXp[\rho Y_1,\dots,\rho Y_d]_{\pd})\big|_{U_{\infty}}$ is free over $(\OXp)\big|_{U_{\infty}}$. Then one can conclude the claim from the same argument in the proof of \cite[Lem. 5.7]{Wang}.

      Now we follow the proof of \cite[Lem. 5.11]{Wang}. For any $m\geq 1$, let $U_{\infty}^{m/U}$ be the $m$-fold self-product of $U_{\infty}$ over $U$. As $U_{\infty}\to U$ is a pro-\'etale Galois covering with the Galois group $\Gamma$, we have $U_{\infty}^{m/U}\cong U_{\infty}\times\Gamma^{m-1}$ and a quasi-isomorphism
      \[\rR\Gamma(U_{\proet},(\calL\otimes_{\OXp}\calP)\big|_{U})\simeq \rR\Gamma(U_{\infty,\proet}^{\bullet/U},\calL\otimes_{\OXp}\calP).\]
      As argued in \emph{loc.cit.}, we then have an almost isomorphism
      \[\Hom_{\cts}(\Gamma^{m-1},\rH^i(U_{\infty},\calL\otimes_{\OXp}\calP))\to \rH^i(U_{\infty}^{m/U},\calL\otimes_{\OXp}\calP)\]
      for any $i\geq 0$ and any $m\geq 1$. By letting $i = 0$, we get a natural map of complexes
      \[\begin{split}
          \Hom_{\cts}(\Gamma^{\bullet-1},(U_{\infty},\calL\otimes_{\OXp}\calP)(U_{\infty}))&\to(\calL\otimes_{\OXp}\calP)(U_{\infty}^{\bullet/U})\\
          &\to \rR\Gamma(U_{\infty,\proet}^{\bullet/U},\calL\otimes_{\OXp}\calP)\simeq \rR\Gamma(U_{\proet},(\calL\otimes_{\OXp}\calP)\big|_{U}),
      \end{split}\]
      yielding the desired natural map via the quasi-isomorphism
      \[\rR\Gamma(\Gamma,(\calL\otimes_{\OXp}\calP)(U_{\infty}))\simeq\Hom_{\cts}(\Gamma^{\bullet-1},\calL\otimes_{\OXp}\calP)(U_{\infty}))\]
  \end{proof}
  \begin{cor}\label{Cor-Key}
      Keep notations as in Lemma \ref{Lem-Key}. If moreover $c>\nu_p(\rho)+\frac{1}{p-1}$, then the above natural map induces a quasi-isomorphism 
      \[\rL\eta_{\rho(\zeta_p-1)}\rR\Gamma(\Gamma,(\calL\otimes_{\OXp} \calO\widehat \bC^+_{\pd})(U_{\infty}))\simeq \rL\eta_{\rho(\zeta_p-1)}\rR\Gamma(U_{\proet},(\calL\otimes_{\OXp} \calO\widehat \bC^+_{\pd})\big|_{U}).\]
      Moreover, we have isomorphisms of $R^+$-modules
      \[\begin{split}\rH^0(U_{\proet},(\calL\otimes_{\OXp} \calO\widehat \bC^+_{\pd})\big|_{U}) &\cong\rH^0(\Gamma,(\calL\otimes_{\OXp} \calO\widehat \bC^+_{\pd})(U_{\infty}))\\
      &\cong \rH^0(\rL\eta_{\rho(\zeta_p-1)}\rR\Gamma(\Gamma,(\calL\otimes_{\OXp} \calO\widehat \bC^+_{\pd})(U_{\infty})))\\
      &\cong\rH^0(\rL\eta_{\rho(\zeta_p-1)}\rR\Gamma(U_{\proet},(\calL\otimes_{\OXp} \calO\widehat \bC^+_{\pd})\big|_{U})).
      \end{split}\]
  \end{cor}
  \begin{proof}
      By Example \ref{Exam-LocalSmallRep}, $\calL(U_{\infty})$ is an $a$-small $\widehat R_{\infty}^+$-representation of $\Gamma$ for some $a>\frac{1}{p-1}$. Thus, we deduce from Corollary \ref{Cor-Kunnech} (or Lemma \ref{Lem-coherent} below) that $\rR\Gamma(\Gamma,(\calL\otimes_{\OXp} \calO\widehat \bC^+_{\pd})(U_{\infty}))$ has no $\frakm_C$-torsion. Then the desired quasi-isomorphism comes from Lemma \ref{Lem-Key} combined with \cite[Lem. 8.11 (2)]{BMS18}. Finally, the ``moreover'' part follows from Lemma \ref{Lem-Key} together with the $p$-torsion freeness of $\rH^0(\Gamma,(\calL\otimes_{\OXp} \calO\widehat \bC^+_{\pd})(U_{\infty}))$, again by Corollary \ref{Cor-Kunnech}.
  \end{proof}

   Now, we are prepared to prove Theorem \ref{Thm-IntegralSimpson}.
   \begin{proof}[Proof of Theorem \ref{Thm-IntegralSimpson}]
       We prove Item (1) as follows: Fix an $a$-small $\OXp$-representation $\calL$ and let $\{\frakX_i\to\frakX\}_{i\in I}$, $b_i$ and $b$ be as in Definition \ref{Dfn-SmallRep}. 
       
       Note that $\rL\eta_{\rho(\zeta_p-1)}\rR\nu_*(\calL\otimes_{\OXp}\calO\widehat \bC_{\pd}^+)$ is the sheafification of the presheaf 
       \[\frakU\in\frakX_{\et}\mapsto \rL\eta_{\rho(\zeta_p-1)}\rR\Gamma(U_{\proet},\calL\otimes_{\OXp}\calO\widehat \bC_{\pd}^+).\]
       In fact, for any $x\in \frakX$, the stalk at $x$ of the sheaf $\rL\eta_{\rho(\zeta_p-1)}\rR\nu_*(\calL\otimes_{\OXp}\calO\widehat \bC_{\pd}^+)$ is given by 
       \[
       \rL\eta_{\rho(\zeta_p-1)}\varinjlim_{x\in \frakU}\rR\Gamma(\frakU_{\et},\rR\nu_*(\calL\otimes_{\OXp}\calO\widehat \bC_{\pd}^+)) =
       \rL\eta_{\rho(\zeta_p-1)}\varinjlim_{x\in \frakU}\rR\Gamma(U_{\proet},\calL\otimes_{\OXp}\calO\widehat \bC_{\pd}^+).
       \] 
       as the $\rL\eta$-functor commutes with taking stalks by \cite[Lemma 6.14]{BMS18}. Note the stalk at $x$ of the above presheaf is given by \[\varinjlim_{x\in \frakU}\rL\eta_{\rho(\zeta_p-1)}\rR\Gamma(U_{\proet},\calL\otimes_{\OXp}\calO\widehat \bC_{\pd}^+).\] Then the claim follows from that the $\rL\eta$-functor commutes with filtered colimit by \cite[Corollary 6.5]{BMS18} (see also the proof of \cite[Cor. 8.13(iv)]{BMS18}).

       Also, the sheaf $\nu_*(\calL\otimes_{\OXp}\calO\widehat \bC_{\pd}^+)$ is the sheafification of the presheaf
       \[\frakU\in\frakX_{\et}\mapsto \rH^0(U_{\proet},\calL\otimes_{\OXp}\calO\widehat \bC_{\pd}^+).\]
       We know that the problem is local on $\frakX_{\et}$ and hence we are reduced to showing that for a small affine $\frakU = \Spf(R^+)$ as in Convention \ref{Convention-small} which is \'etale over $\frakX_i$ for some $i\in I$, 
       \begin{enumerate}
           \item[(1)] $\rL\eta_{\rho(\zeta_p-1)}\rR\Gamma(U_{\proet},\calL\otimes_{\OXp}\calO\widehat \bC_{\pd}^+)\cong \rH^0(U_{\proet},\calL\otimes_{\OXp}\calO\widehat \bC_{\pd}^+)[0]$, and

           \item[(2)] $\rH^0(U_{\proet},\calL\otimes_{\OXp}\calO\widehat \bC_{\pd}^+)$ is a free $R^+$-module of rank $r$ such that the restriction of $\nu_*(\Theta_{\calL})$ is a $b$-small Higgs field.
       \end{enumerate}

       Indeed, in this case, we know that $\calL\big|_{U}\cong ((\OXp\big|_{U}/\rho p^{b_i})^r)^{\rm al}$. By Example \ref{Exam-LocalSmallRep}, $\calL(U_{\infty})$ is a $b$-small $\widehat R_{\infty}^+$-representation of $\Gamma$.
       So the condition (1) above can be confirmed by combining Corollary \ref{Cor-Key} together with (\ref{Equ-LocalSimpson-I}) while the condition (2) follows from Theorem \ref{Thm-LocalSimpson}.

       For Item (2), let us fix an $a$-small Higgs bundle $(\calH,\theta_{\calH})$ of rank $r$ on $\frakX_{\et}$ such that $\Ima(\theta_{\calH})\subset \rho p^{c-\frac{1}{p-1}}$ for some rational $c>a$. Define $\calL = (\calH\otimes_{\calO_{\frakX}}\calO\widehat \bC^+_{\pd})^{\Theta_{\calH} = 0}$. As $\calH\otimes_{\calO_{\frakX}}\calO\widehat \bC^+_{\pd}$ is $p$-complete $p$-torsion free, so is $\calL$. Now, let $\{\frakX_i\to\frakX\}_{i\in I}$ be an \'etale covering of small affines such that for any $i$, the restriction of $(\calH,\theta_{\calH})$ to $\frakX_i$ is induced by a $c$-small Higgs module over $\calO_{\frakX}(\frakX_i)$. Let $X_i$ be the rigid analytic generic fiber of $\frakX_i$. It remains to show that there exists an almost isomorphism
       \[\calL\big|_{X_i}\cong (\OXp\big|_{X_i}/\rho p^c)^r)^{\rm al}.\]
       (Then one can conclude that $\calL$ is an $a$-small $\OXp$-representation by letting $b_i = c$ and $b = \frac{a+c}{2}$.)

       Fix an $i\in I$ as above, let $\frakU = \frakX_i = \Spf(R^+)$ and keep notations as in Convention \ref{Convention-small} By Theorem \ref{Thm-LocalSimpson} (2), we know that $\calL(U_{\infty})$ is a $c$-small $\widehat R_{\infty}^+$-representation of $\Gamma$ of rank $r$ and hence get a $\Gamma$-equivariant isomorphism
       \[\calL(U_{\infty})/p^c\rho \cong (\widehat R_{\infty}^+/\rho p^c)^r.\]
       As $\Gamma$ is the Galois group of the pro-\'etale Galois covering $U_{\infty}\to U$, we deduce the desired almost isomorphism by the proof of \cite[Lem. 4.10 (i)]{Sch-Pi}. 

       We finally prove Item (3). Let $\calL\mapsto (\calH(\calL),\theta_{\calH(\calL)})$ and $(\calH,\theta_{\calH})\mapsto \calL(\calH,\theta_{\calH})$ be the functor defined in Items (1) and (2), respectively. By adjunction, there exists a natural morphism of Higgs fields
       \[(\calH(\calL)\otimes_{\calO_{\frakX}}\calO\widehat \bC_{\pd}^+,\Theta_{\calH(\calL)})\to(\calL\otimes_{\OXp}\calO\widehat \bC_{\pd}^+,\Theta_{\calL}).\]
       It is clear that this is an isomorphism. Namely, since the problem is local, we are reduced to Theorem \ref{Thm-LocalSimpson}. As a consequence, we get isomorphisms
       \[\calL(\calH(\calL),\theta_{\calH(\calL)}) = (\calH(\calL)\otimes_{\calO_{\frakX}}\calO\widehat \bC_{\pd}^+)^{\Theta_{\calH(\calL)} = 0}\cong (\calL\otimes_{\OXp}\calO\widehat \bC_{\pd}^+)^{\Theta_{\calL} = 0} = \calL,\]
       where the last equality is due to that $\HIG(\calO\widehat \bC_{\pd}^+,\Theta)$ is a resolution of $\OXp$. On the other hand, by the definition of $\calL(\calH,\theta_{\calH})$ and by linear extension, there is a natural morphism of Higgs fields
       \[(\calL(\calH,\theta_{\calH})\otimes_{\OXp}\calO\widehat \bC_{\pd}^+,\Theta_{\calL(\calH,\theta_{\calH})})\to (\calH\otimes_{\calO_{\frakX}}\calO\widehat \bC_{\pd}^+,\Theta_{\calH}).\]
       By working in the local case and applying Theorem \ref{Thm-LocalSimpson}, we know this is also an isomorphism. Applying $\nu_*$ to both sides, we get an isomorphism of Higgs bundles
       \[(\calH(\calL(\calH,\theta_{\calH})),\theta_{\calH(\calL(\calH,\theta_{\calH}))})\to (\calH,\theta_{\calH}).\]
       Therefore, the functors in Items (1) and (2) are equivalences and are quasi-inverses of each other. The above argument also implies that the isomorphism stated in Item (3) is true. To finish the proof, we have to show these equivalences preserve tensor products and dualities. But this is a local problem and hence follows from Theorem \ref{Thm-LocalSimpson}.
   \end{proof}

     Finally, we compare Theorem \ref{Thm-IntegralSimpson} with \cite[Thm. 5.3]{Wang}. Let us recall some definitions in \textit{loc.cit.}. Assume $a\geq \frac{1}{p-1}$. An {\bf $a$-small generalised representation of rank $r$} in the sense of \cite[Def. 5.1]{Wang} is a locally finite free $\OX$-module of the form $\calL[\frac{1}{p}]$ with $\calL$ an $a$-small $\OXp$-representation of rank $r$ on $X_{\proet}$. In \cite[Def. 5.2]{Wang}, an {\bf $a$-small Higgs bundle of rank $r$} denotes a Higgs bundle with coefficients in $\calO_{\frakX}[\frac{1}{p}]$ of the form $(\calH[\frac{1}{p}],(\zeta_p-1)\theta_{\calH})$ with $(\calH,\theta_{\calH})$ an $a$-small Higgs bundle in the sense of Definition \ref{Dfn-SmallHiggs}. To distinguish notions, we define rational $a$-small Higgs bundles of rank $r$ as $a$-small Higgs bundles of rank $r$ in the sense of \cite[Def. 5.2]{Wang}. Let $\Vect^{\geq a}(\OX)$ and $\HIG^{\geq a}_W(\calO_{\frakX}[\frac{1}{p}])$ be the category of $a$-small generalised representations and the category of rational $a$-small Higgs bundles respectively.
   \begin{cor}\label{Cor-CompareWithWang}
       Keep assumptions as in Theorem \ref{Thm-IntegralSimpson}. Then the functors
       \[\calL\in\Vect^{\geq a}(\OX)\mapsto (\calH(\calL),\theta_{\calH(\calL)}):=(\nu_*(\calL\otimes_{\OX}\calO\widehat \bC_{\pd}),(\zeta_p-1)\nu_*(\Theta_{\calL}))\]
       defines an equivalence of categories $ \Vect^{\geq a}(\OX)\xrightarrow{\simeq}\HIG^{\geq a}(\calO_{\frakX}[\frac{1}{p}])$, which preserves tensor product and dualities. Moreover, the complex $\rR\nu_*(\calL\otimes_{\OX}\calO\widehat \bC_{\pd})$ is concentrated in degree $0$ and there exists a quasi-isomorphism
       \[\rR\nu_*(\calL)\simeq \HIG(\calH(\calL),\theta_{\calH(\calL)}).\]
       This equivalence coincides with that in \cite[Thm. 5.3]{Wang}
   \end{cor}
   \begin{proof}
       All assertions can be concluded by the same arguments for the proof of Theorem \ref{Thm-IntegralSimpson} except the last sentence. 
       Let $(\calO\bC^{\dagger},\widetilde \Theta)$ be the period sheaf appearing in \cite[thm. 5.3]{Wang}. By Proposition \ref{Prop-ComparePeriodSheaf}, there exists a natural morphism of Higgs fields
       \[(\calO\bC^{\dagger},\widetilde \Theta) \to (\calO\widehat \bC_{\pd},(\zeta_p-1)\Theta).\]
       Therefore, for any $a$-small generalised representation $\calL$, we have a natural morphism
       \[(\calL\otimes_{\OX}\calO\bC^{\dagger},\id_{\calL}\otimes\widetilde \Theta) \to (\calL\otimes_{\OX}\calO\widehat \bC_{\pd},(\zeta_p-1)\id_{\calL}\otimes\Theta)\]
       and a fortiori a morphism
       \[(\nu_*(\calL\otimes_{\OX}\calO\bC^{\dagger}),\nu_*(\id_{\calL}\otimes\widetilde \Theta))=:(\calH_W(\calL),\theta_{\calH_W(\calL)}) \to (\calH(\calL),\theta_{\calH}).\]
       By \cite[Thm. 5.3]{Wang}, $(\calH_W(\calL),\theta_{\calH_W(\calL)})$ is exactly the rational $a$-small Higgs bundles corresponding to $\calL$ in the sense of \cite[Thm. 5.3]{Wang}. So it suffices that the above map $(\calH_W(\calL),\theta_{\calH_W(\calL)}) \to (\calH(\calL),\theta_{\calH})$ is an isomorphism. Since the problem is local, we are reduced to Remark \ref{Rmk-CompareLocalSimpson}.
   \end{proof}

\section{Cohomological comparison and a conjectural analogue of Deligne--Illusie decomposition}\label{Sec-PerfectComplex}
  Throughout this section, we always assume $\lambda\in\calO_C$ such that $\nu_p(\lambda)\leq \nu_p(\rho)$.

  Recall that in Theorem \ref{Thm-IntegralSimpson}, it is hard to compare $\HIG(\calH,\theta_{\calH})$ with $\rL\eta_{\rho(\zeta_p-1)}\rR\nu_*\calL$
  for corresponding $\calL$ and $(\calH,\theta_{\calH})$ directly (cf. Remark \ref{Rmk-Warnning}). So the following questions appear naturally.
  \begin{ques}\label{Ques-PerfectComplex}
      \begin{enumerate}
          \item[(1)] Is $\rL\eta_{\rho(\zeta_p-1)}\rR\nu_*\calL$ a perfect complex in $D(\frakX)$?

          \item[(2)] If the answer of (1) is ``yes'', does the natural morphism in Corollary \ref{Cor-IntegralSimpson} upgrade to a quasi-isomorphism
          $\HIG(\calH,\theta_{\calH})\simeq\rL\eta_{\rho(\zeta_p-1)}\rR\nu_*\calL$?
      \end{enumerate}
  \end{ques}
  Note that if the answer to (1) is negative, $\HIG(\calH,\theta_{\calH})$ will never be isomorphic to $\rL\eta_{\rho(\zeta_p-1)}\rR\nu_*\calL$ as the former is a perfect complex while the latter is not. We try to investigate these questions in this section.

\subsection{The $\rL\eta_{(\zeta_p-1)\lambda}\rR\nu_*\calL$ is a perfect complex}\label{SSec-PerfectComplex}
  In this subsection, we do {\bf Not} assume $\frakX$ is liftable and are going to show the following result:
  \begin{thm}\label{Thm-PerfectComplex}
      Assume $a\geq \frac{1}{p-1}$.
      For any $a$-small $\OXp$-representation $\calL$, $\rL\eta_{(\zeta_p-1)\lambda}\rR\nu_*\calL$ is a perfect complex in $D(\frakX)$ concentrated in degree $[0,d]$ with $p$-torsion free $\rH^0$. As a consequence, by applying \cite[Lem. 6.10]{BMS18}, we obtain a canonical map $\rL\eta_{(\zeta_p-1)\lambda}\rR\nu_*\calL\to\rR\nu_*\calL$.
  \end{thm}
  \begin{rmk}
      By Theorem \ref{Thm-PerfectComplex}, the answer to Question \ref{Ques-PerfectComplex}(1) is ``yes'' (by letting $\lambda = \rho$) even when $\frakX$ is not liftable.
      When $\lambda = 1$ and $\calL = \OXp$, the result was proved in \cite[\S 8]{BMS18}.
  \end{rmk}
  \begin{proof}[Proof of Theorem \ref{Thm-PerfectComplex}:]
      Similar to the proof of Theorem \ref{Thm-IntegralSimpson} (1), it suffices to show that for small affine $\frakX = \Spf(R^+)$\footnote{In this case, it follows from the smoothness that $\frakX$ admits a unique lifting over $\rA_{\inf}$ up to isomorphisms.}, as a complex of $R^+$-modules, $\rL\eta_{\zeta_p-1}\rR\Gamma(X_{\proet},\calL)$ is perfect and concentrated in degree $[0,d]$. Since $\calL$ is $a$-small, by \cite[Lem. 5.11]{Wang}, we get an almost quasi-isomorphism
      \[\rR\Gamma(\Gamma,V)\rightarrow\rR\Gamma(X_{\proet},\calL),\]
      where $V = \calL(X_{\infty})$ is a $b$-small $\widehat R^+_{\infty}$-representation of $\Gamma$ for some $b>a$.

      By Lemma \ref{Lem-Local L-eta}, it is enough to show that after applying $\rL\eta_{(\zeta_p-1)\lambda}$, we have a quasi-isomorphism
      \[\rL\eta_{(\zeta_p-1)\lambda}\rR\Gamma(\Gamma,V)\simeq\rL\eta_{(\zeta_p-1)\lambda}\rR\Gamma(X_{\proet},\calL).\]
      Thanks to \cite[Lem. 8.11(2)]{BMS18}, it suffices to show that for any $n\geq 0$, $\rH^n(\Gamma,V)$ and $\rH^n(\Gamma,V)/\lambda(\zeta_p-1)$ have no almost zero elements.

      Let $V_0$ be the $b$-small $R^+$-representation of $\Gamma$ corresponding to $V$ in the sense of Theorem \ref{Thm-Decompletion}. In particular, we have a $\Gamma$-equivariant isomorphism $V\cong V_0\otimes_{R^+}\widehat R_{\infty}^+$. Now, write
      \[\widehat R_{\infty}^+ = \widehat {\bigoplus_{\underline \alpha = (\alpha_1,\dots,\alpha_d)\in([0,1)\cap\bZ[1/p])^d}}R^+T_1^{\alpha_1}\cdots T_d^{\alpha_d}.\]
      The ``moreover'' part of Theorem \ref{Thm-Decompletion} implies that 
      \[\rR\Gamma(\Gamma,V) = \bigoplus_{\underline \alpha\in([0,1)\cap\bZ[1/p])^d}\rR\Gamma(\Gamma,V_0\underline T^{\underline \alpha})\]
      such that $\rR\Gamma(\Gamma,V_0\underline T^{\underline \alpha})$ is killed by $\zeta_p-1$ as long as $\underline \alpha\neq 0$. Since $\rR\Gamma(\Gamma,V_0\underline T^{\underline \alpha})$ is represented by the Koszul complex
      \[V_0\underline T^{\underline \alpha}\xrightarrow{\gamma_1-1,\dots,\gamma_d-1}(V_0\underline T^{\underline \alpha})^d\to\cdots,\]
      which is a perfect complex of $R^+$-modules. By the coherence of $R^+$ (cf. \cite[\S 7.3 Cor. 6]{Bos}), we know for any $n\geq 0$, $\rH^n(\Gamma,V_0\underline T^{\underline \alpha})$ is a coherent $R^+$-module. Therefore, $\rH^n(\Gamma,V_0\underline T^{\underline \alpha})/(\zeta_p-1)\lambda$ is also coherent. Now we can conclude by applying Lemma \ref{Lem-coherent} below.
  \end{proof}

  \begin{lem}\label{Lem-coherent}
      Let $\frakX = \Spf(R^+)$ be small smooth. Then every coherent $R^+$-module $M$ has no non-trivial almost zero elements.
  \end{lem}
  \begin{proof}
      Otherwise, let $x\in M\setminus\{0\}$ be a non-trivial almost zero element and $A$ be the sub-$R^+$-module of $M$ generated by $x$. As $M$ is coherent, $A$ is of finite presentation. Therefore, there exists a finite generated ideal $I\subset R^+$ fitting into the exact sequence
      \begin{equation}\label{SES-Coherent-I}
          0\to I\to R^+\to A\to 0
      \end{equation}
      Since $x$ is almost zero, we have $\frakm_CR^+\subset I$. So $A$ is a non-zero $\overline \kappa$-algebra, where $\overline \kappa = \calO_C/\frakm_C$. 
      
      Let $\overline t = (\overline t_1,\dots,\overline t_d): A\to \overline \kappa$ be a closed point in $\Spec (A)$, where $\overline t_i$ denotes the image of $T_i$ in $\overline \kappa$. Let $t_i\in\calO_C$ be a fixed lifting of $\overline t_i$. Then we know that $t_i$'s are invertible. So we get a surjective morphism $t:\calO_C\za\underline T^{\pm 1}\ya\to\calO_C$ sending each $T_i$ to $t_i$.

      Let $\Box:\calO_C\za\underline T^{\pm 1}\ya\to R^+$ be the chosen frame on $R^+$. Applying $-\otimes_{\calO_C\za\underline T^{\pm 1}\ya,t}\calO_C$ to the short exact sequence (\ref{SES-Coherent-I}), we get an exact sequence
      \begin{equation}\label{SES-Coherent-II}
          I\otimes_{\calO_C\za\underline T^{\pm 1}\ya,t}\calO_C\to R^+\otimes_{\calO_C\za\underline T^{\pm 1}\ya,t}\calO_C\to A\otimes_{\calO_C\za\underline T^{\pm 1}\ya,t}\calO_C\to 0.
      \end{equation}
      Since $I$ is a finite generated ideal of $R^+$, we know that $A\otimes_{\calO_C\za\underline T^{\pm 1}\ya,t}\calO_C$ is a finite presented $(R^+\otimes_{\calO_C\za\underline T^{\pm 1}\ya,t}\calO_C)$-module and is killed by $\frakm_C$. By the construction of $t$, the $A\otimes_{\calO_C\za\underline T^{\pm 1}\ya,t}\calO_C$ is non-zero. As $\Box$ is an \'etale morphism, we know that 
      \[R^+\widehat \otimes_{\calO_C\za\underline T^{\pm 1}\ya,t}\calO_C \cong \calO_C^r\] 
      is \'etale over $\calO_C$ and hence a finite product of $\calO_C$ for some $r\geq 1$. Therefore, 
      \[A\otimes_{\calO_C\za\underline T^{\pm 1}\ya,t}\calO_C\cong \overline \kappa^s\]
      for some $1\leq s\leq r$, which is not of finite presentation (as $\frakm_C$ is not finitely generated). So we get a contradiction and complete the proof.
  \end{proof}
  \begin{rmk}\label{Rmk-Coherent}
    Let $R^+$ be an $\calO_C$-algebra of topologically finite presentation. If it satisfies the condition that any closed point $\overline t:\Spec(\overline \kappa)\to\Spf(R^+)$ lifts to an $\calO_C$-point $t:\Spf(\calO_C)\to \Spf(R^+)$\footnote{Indeed, the above condition on lifting points holds for any $R^+$ which is of topologically finite presentation. To see this, let $\frakX = \Spf(R^+)$ with the adic generic fiber $X$, and $\overline s\in\frakX(\overline \kappa)$ be the given point. Then $\overline s$ lifts to a point $\widetilde s\in X(C)$, as the specialisation map $X(C)\to \frakX(\overline \kappa)$ is always surjective (\cite[\S 8.3 Prop. 8]{Bos}). Now we can conclude by noticing that $\widetilde s$ extends to a point $s:\Spf(\calO_C)\to \frakX$ as the set of power-bounded elements of $C$ is exactly $\calO_C$. We learn the above argument from Fei Liu and thank him for allowing us to include it here.}, then by the same argument for the proof of Lemma \ref{Lem-coherent}, one can show that any coherent $R^+$-module has no non-trivial almost zero elements.
  \end{rmk}

\subsection{Cohomological comparison: Truncation one case}\label{SSec-TruncationOne}
  In this subsection, we always assume $\frakX$ is liftable and fix a lifting.
  We give a partial answer to Question \ref{Ques-PerfectComplex}(2). 

\begin{thm}\label{Thm-TruncationOne}
    Assume $a\geq\frac{1}{p-1}$. Then for any $a$-small $\OXp$-representation $\calL$ with induced Higgs bundle $(\calH,\theta_{\calH})$, the composite
    $\tau^{\leq 1}\HIG(\calH,\theta_{\calH})\to\HIG(\calH,\theta_{\calH})\xrightarrow{{\rm Cor.} \ref{Cor-IntegralSimpson}}\rR\nu_*\calL$ uniquely factors through $\rL\eta_{\rho(\zeta_p-1)}\rR\nu_*\calL\xrightarrow{{\rm Thm.} \ref{Thm-PerfectComplex}}\rR\nu_*\calL$ and induces a quasi-isomorphism 
    \[\tau^{\leq 1}\HIG(\calH,\theta_{\calH})\xrightarrow{\simeq} \tau^{\leq 1}\rL\eta_{\rho(\zeta_p-1)}\rR\nu_*\calL.\]
\end{thm}
  We can then obtain the following corollary.
\begin{cor}\label{Cor-TruncationOne}
    Assume $\frakX$ is a smooth curve (i.e. $d=1$) and $a\geq \frac{1}{p-1}$. Then for any $a$-small $\OXp$-representation with induced Higgs bundle $(\calH,\theta_{\calH})$ in the sense of Theorem \ref{Thm-IntegralSimpson}, the canonical map in Corollary \ref{Cor-IntegralSimpson} induces a quasi-isomorphism 
    \[\HIG(\calH,\theta_{\calH})\xrightarrow{\simeq}\rL\eta_{\rho(\zeta_p-1)}\rR\nu_*\calL.\]
\end{cor}
\begin{proof}
    In this case, $\rL\eta_{\rho(\zeta_p-1)}\rR\nu_*\calL$ is concentrated in degree $[0,1]$. So the result follows from Theorem \ref{Thm-TruncationOne} immediately.
\end{proof}
  Note that Corollary \ref{Cor-TruncationOne} gives a positive answer to Question \ref{Ques-PerfectComplex} when $\frakX$ is a smooth curve. The rest of \S\ref{SSec-TruncationOne} is devoted to proving Theorem \ref{Thm-TruncationOne}. Similar to the proof of Theorem \ref{Thm-IntegralSimpson}, we start with the local case. 
 
 From now on, we keep and freely use the notations in the proof of Theorem \ref{Thm-LocalSimpson}. For the convenience of the reader, let us recall what we have proved in (the proof of) Theorem \ref{Thm-LocalSimpson}:
 
 Let $V$ be an $a$-small $\widehat R_{\infty}^+$-representation of $\Gamma$ of rank $r$, $V_0$ the induced $R^+$-representation of $\Gamma$ in the sense of Theorem \ref{Thm-Decompletion}, and $(H,\theta_H = \sum_{i=1}^d\Theta_i\otimes\frac{\dlog T_i}{t})$ the induced Higgs module over $R^+$ in the sense of Theorem \ref{Thm-LocalSimpson}. Note that the $a$-smallness of $V$ implies that $\rho^{-1}\Theta_i$'s are well-defined topologically nilpotent matrices in $\rM_r(R^+)$. Let $S^+ = R^+[\rho Y_1,\dots,\rho Y_d]^{\wedge}_{\pd}$ denote the $A_d$ in the proof of Theorem \ref{Thm-LocalSimpson} with induced $\Gamma$-action and Higgs field $\Theta = \sum_{i=1}^d\frac{\partial}{\partial Y_i}\otimes\frac{\dlog T_i}{t}$. Put $M:=H\otimes_{R^+}S^+$ and then $M= V_0\otimes_{R^+}S^+$ such that
 \[\Theta_M = \id_{V_0}\otimes\Theta = \theta_H\otimes\id_{S^+}+\id_H\otimes\Theta.\]
Moreover, we have that $M^{\Gamma} = H$, that
\[V_0 = \exp(-\sum_{k=1}^d\Theta_kY_k)H = \{\sum_{n_1,\dots,n_d\geq 0}(-\frac{\Theta_1}{\rho})^{n_1}\cdots(-\frac{\Theta_d}{\rho})^{n_d}(\rho Y_1)^{[n_1]}\cdots(\rho Y_d)^{[n_d]}h\mid h\in H\}\] which follows from (\ref{Equ-ExplicitCorrespondence}), and that a quasi-isomorphism $V_0\simeq\HIG(M,\Theta_M)$ such that the canonical map 
\[\HIG(H,\theta_H)\to \rR\Gamma(\Gamma,\HIG(V\otimes_{\widehat R^+_{\infty}}\widehat S_{\pd}^+,\Theta_V))\simeq \rR\Gamma(\Gamma,V)\]
factors through 
\[\HIG(H,\theta_H)\to\rR\Gamma(\Gamma,\HIG(M,\Theta_M))\simeq\rR\Gamma(\Gamma,V_0).\]

 \begin{prop}\label{Prop-TruncationOne}
     \begin{enumerate}
         \item[(1)] The canonical morphism 
         \[\HIG(H,\theta_H)\to\rR\Gamma(\Gamma,\HIG(V_0\otimes_{R^+}\widehat S_{\pd}^+,\Theta_V))\simeq\rR\Gamma(\Gamma,V)\]
         induces an isomorphism $\rH^1(\HIG(H,\theta_H))\simeq\rho(\zeta_p-1)\rH^1(\Gamma,V)$.

         \item[(2)] The canonical morphism $\HIG(H,\theta_H)\to\rR\Gamma(\Gamma,\HIG(V_0\otimes_{R^+}\widehat S_{\pd}^+,\Theta_V))\simeq\rR\Gamma(\Gamma,V)$ induces a quasi-isomorphism 
         \[\tau^{\leq 1}\HIG(H,\theta_H)\simeq\tau^{\leq 1}\rL\eta_{\rho(\zeta_p-1)}\rR\Gamma(\Gamma,V).\]
     \end{enumerate}
 \end{prop}
 \begin{proof}
     As $\rR\Gamma(\Gamma,V_0)$ is a direct summand of $\rR\Gamma(\Gamma,V)$ with complement being killed by $\zeta_p-1$ and concentrated in degree $\geq 1$, we have $\rho(\zeta_p-1)\rH^1(\Gamma,V_0) = \rho(\zeta_p-1)\rH^1(\Gamma,V)$ and $\rL\eta_{\rho(\zeta_p-1)}\rR\Gamma(\Gamma,V_0)\simeq\rL\eta_{\rho(\zeta_p-1)}\rR\Gamma(\Gamma,V)$. So it suffices to show that
     \begin{enumerate}
         \item[(i)] the induced map $\HIG(H,\theta_H)\to\rR\Gamma(\Gamma,\HIG(M,\Theta_M))\simeq\rR\Gamma(\Gamma,V_0)$ induces an isomorphism
         \[\rH^1(\HIG(H,\theta_H))\cong\rho(\zeta_p-1)\rH^1(\Gamma,V_0)\]
         and that

         \item[(ii)] the induced map $\HIG(H,\theta_H)\to\rR\Gamma(\Gamma,\HIG(M,\Theta_M))\simeq\rR\Gamma(\Gamma,V_0)$ induces a quasi-isomorphism 
         \[\tau^{\leq 1}\HIG(H,\theta_H)\simeq\tau^{\leq 1}\rL\eta_{\rho(\zeta_p-1)}\rR\Gamma(\Gamma,V_0).\]
     \end{enumerate}
     In what follows, to simplify the notation, for any $\underline n = (n_1,\dots,n_d)\in\bN^d$, we put
     \[\rho^{|\underline n|}\underline Y^{[\underline n]} = (\rho Y_1)^{[n_1]}\cdots(\rho Y_d)^{[n_d]}.\]

     We first prove (i). Consider the following commutative diagram:
\begin{equation}\label{Diag-DoubleComplex}
    \xymatrix@C=0.5cm{
      && H\ar[rr] \ar[d]&& H\otimes\rho\widehat\Omega^1(-1)\ar[rr]\ar[d] && H\otimes\rho^2\widehat\Omega^2(-2)\ar[r]\ar[d] &\dots\\
      V_0\ar[rr]\ar[d]&& M\ar[rr]\ar[d]&& M\otimes\rho\widehat\Omega^1(-1)\ar[rr]\ar[d] && M\otimes\rho^2\widehat\Omega^2(-2)\ar[r]\ar[d] &\dots\\
      \wedge^1V_0^d\ar[d]\ar[rr]&& \wedge^1M^d\ar[rr]\ar[d]&& \wedge^1M^d\otimes\rho\widehat\Omega^1(-1)\ar[rr]\ar[d] && \wedge^1M^d\otimes\rho^2\widehat\Omega^2(-2)\ar[r]\ar[d] &\dots\\
      \wedge^2V_0^d\ar[d]\ar[rr]&& \wedge^2M^d\ar[rr]\ar[d]&& \wedge^2M^d\otimes\rho\widehat\Omega^1(-1)\ar[rr]\ar[d] && \wedge^2M^d\otimes\rho^2\widehat\Omega^2(-2)\ar[r]\ar[d] &\dots\\
      \vdots&& \vdots&& \vdots &&\vdots &
    }
\end{equation}
where the horizontal arrows are induced by Higgs fields and the vertical arrows are induced by Koszul complexes associated to $\Gamma$-actions. Then we have the following commutative diagram
\begin{equation}\label{Diag-TotalComplex}
    \xymatrix@C=0.5cm{
      H\ar[r] \ar[d]& H\otimes\rho\widehat \Omega^1(-1)\ar[rr]\ar[d]&&H\otimes\rho^2\widehat \Omega^2(-2)\ar[r]\ar[d]&\dots\\
      M\ar[r]&M\otimes\rho\widehat \Omega^1(-1)\oplus\wedge^1M^d\ar[rr]&&M\otimes\rho^2\widehat \Omega^2(-2)\oplus\wedge^1M^d\otimes\rho\widehat \Omega^1(-1)\oplus\wedge^2M^d\ar[r]&\dots\\
      V_0\ar[r]\ar[u]&\wedge^1V_0^d\ar[rr]\ar[u]&&\wedge^2V_0^d\ar[u]\ar[r]&\dots
    }
\end{equation}
such that the arrows from the bottom to the middle induce the quasi-isomorphism $\rR\Gamma(\Gamma,V_0)\simeq\rR\Gamma(\Gamma,\HIG(M,\Theta_M))$. We have to deduce the relation between $\rH^1(\HIG(H,\theta_H))$ and $\rH^1(\Gamma,V_0)$ from the diagram (\ref{Diag-TotalComplex}).

Let $x_1,\dots,x_d\in H$ such that $\omega = \sum_{i=1}^dx_i\otimes\frac{\rho\dlog T_i}{t}$ represents an element in $\rH^1(\HIG(H,\theta_H))$. Equivalently, we have that for any $1\leq i,j\leq d$, $\Theta_i(x_j) = \Theta_j(x_i)$. We want to determine the element in $\rH^1(\Gamma,V_0)$ induced by $\omega$. To do so, we have to solve the equation 
\begin{equation}\label{Equ-Bridge}
    \Theta_M(\sum_{\underline n}h_{\underline n}\rho^{|\underline n|}\underline Y^{[\underline n]})=\omega,
\end{equation}
where $\sum_{\underline n}h_{\underline n}\rho^{|\underline n|}\underline Y^{[\underline n]}=:m\in M$. Note that
\begin{equation*}
    \begin{split}
        \Theta_M(\sum_{\underline n}h_{\underline n}\rho^{|\underline n|}\underline Y^{[\underline n]})
        =&\sum_{i=1}^d(\sum_{\underline n}\Theta_i(h_{\underline n})\rho^{|\underline n|}\underline Y^{[\underline n]}+\sum_{\underline n}h_{\underline n}\rho^{|\underline n|}\underline Y^{[\underline n-\underline 1_i]})\otimes\frac{\dlog T_i}{t}\\
        =&\sum_{i=1}^d\sum_{\underline n}(\rho^{-1}\Theta_i(h_{\underline n})+h_{\underline n+\underline 1_i})\rho^{|\underline n|}\underline Y^{[\underline n]}\otimes\frac{\rho\dlog T_i}{t}.
    \end{split}
\end{equation*}
So (\ref{Equ-Bridge}) holds true if and only if for any $1\leq i\leq d$ and $\underline n$ satisfying $|\underline n| \geq 1$,
\begin{equation}\label{Equ-Iteration-I}
    h_{\underline n+\underline 1_i} = -\rho^{-1}\Theta_i( h_{\underline n} )
\end{equation}
and 
\begin{equation}\label{Equ-Iteration-II}
    h_{\underline 1_i} = -\rho^{-1}\Theta_i(h_0)+x_i.
\end{equation}
As $\Theta_i(x_j) = \Theta_j(x_i)$ for any $1\leq i,j\leq d$, it is easy to see that for any $h\in H$, one can put $h=h_0$ and use (\ref{Equ-Iteration-I}) and (\ref{Equ-Iteration-II}) to achieve an element $m(h)\in M$ satisfying $\Theta_M(m(h)) = \omega$. Moreover, if we put $m(\omega):=m(0)$, then 
  \[m(h) = m(\omega)+\exp(-\sum_{k=1}^d\Theta_kY_k)h.\]
As a consequence, the image of $\omega$ in $\rH^1(\Gamma,V_0) = \rH^1(\rK(V_0;\gamma_1-1,\dots,\gamma_d-1))$ is represented by 
  \[v(\omega):=(\gamma_1(m(\omega))-m(\omega),\dots,\gamma_d(m(\omega))-m(\omega))\in \wedge^1 V_0^d.\]

  On the other hand, as $\gamma_1$ acts on $Y_2,\dots,Y_d$ trivially (cf. Notation \ref{Notation-LocalPeriodSheaf}), we have 
  \begin{equation*}
      \begin{split}
          &\gamma_1(m(\omega))-m(\omega)\\
          =&\gamma_1\left(\sum_{n_1\geq 1,n_2,\dots,n_d\geq 0}(-\frac{\Theta_1}{\rho})^{n_1-1}(-\frac{\Theta_2}{\rho})^{n_2}\cdots(-\frac{\Theta_d}{\rho})^{n_d}x_1(\rho Y_1)^{[n_1]}\cdots(\rho Y_d)^{[n_d]}\right)\\
          &-\sum_{n_1\geq 1,n_2,\dots,n_d\geq 0}(-\frac{\Theta_1}{\rho})^{n_1-1}(-\frac{\Theta_2}{\rho})^{n_2}\cdots(-\frac{\Theta_d}{\rho})^{n_d}x_1(\rho Y_1)^{[n_1]}\cdots(\rho Y_d)^{[n_d]}\\
          =&\sum_{n_1\geq 1,n_2,\dots,n_d\geq 0}(-\frac{\Theta_1}{\rho})^{n_1-1}(-\frac{\Theta_2}{\rho})^{n_2}\cdots(-\frac{\Theta_d}{\rho})^{n_d}x_1(\rho Y_1+\rho(\zeta_p-1))^{[n_1]}\cdots(\rho Y_d)^{[n_d]}\\
          &-\sum_{n_1\geq 1,n_2,\dots,n_d\geq 0}(-\frac{\Theta_1}{\rho})^{n_1-1}(-\frac{\Theta_2}{\rho})^{n_2}\cdots(-\frac{\Theta_d}{\rho})^{n_d}x_1(\rho Y_1)^{[n_1]}\cdots(\rho Y_d)^{[n_d]}\\
          =&\sum_{n_1\geq 1,n_2,\dots,n_d\geq 0,0\leq l\leq n_1}(-\frac{\Theta_1}{\rho})^{n_1-1}(-\frac{\Theta_2}{\rho})^{n_2}\cdots(-\frac{\Theta_d}{\rho})^{n_d}x_1\rho^{n_1-l}(\zeta_p-1)^{[n_1-l]}(\rho Y_1)^{[l]}(\rho Y_2)^{[n_2]}\cdots(\rho Y_d)^{[n_d]}\\
          &-\sum_{n_1\geq 1,n_2,\dots,n_d\geq 0}(-\frac{\Theta_1}{\rho})^{n_1-1}(-\frac{\Theta_2}{\rho})^{n_2}\cdots(-\frac{\Theta_d}{\rho})^{n_d}x_1(\rho Y_1)^{[n_1]}\cdots(\rho Y_d)^{[n_d]}\\
          =&\sum_{n_1\geq 1,n_2,\dots,n_d\geq 0,0\leq l\leq  n_1-1}(-\frac{\Theta_1}{\rho})^{n_1-1}(-\frac{\Theta_2}{\rho})^{n_2}\cdots(-\frac{\Theta_d}{\rho})^{n_d}x_1\rho^{n_1-l}(\zeta_p-1)^{[n_1-l]}(\rho Y_1)^{[l]}(\rho Y_2)^{[n_2]}\cdots(\rho Y_d)^{[n_d]}\\
          =&\sum_{n_1,n_2,\dots,n_d\geq 0,l\geq 0}(-\frac{\Theta_1}{\rho})^{n_1+l}(-\frac{\Theta_2}{\rho})^{n_2}\cdots(-\frac{\Theta_d}{\rho})^{n_d}x_1\rho^{l+1}(\zeta_p-1)^{[l+1]}(\rho Y_1)^{[n_1]}\cdots(\rho Y_d)^{[n_d]}\\
          =&\sum_{l\geq 0}\rho(-\Theta_1)^{l}(\zeta_p-1)^{[l+1]}\exp(-\sum_{k=1}^d\Theta_kY_k)x_1\\
          =&-(\zeta_p-1)\rho F(\Theta_1)\exp(-\sum_{k=1}^d\Theta_kY_k)x_1,
      \end{split}
  \end{equation*}
  where $F(X) = \frac{\exp(-(\zeta_p-1)X)-1}{(\zeta_p-1)X}$ was defined in (\ref{Equ-F}). Similarly, for any $1\leq i\leq d$, we have
  \[(\gamma_i-1)(m(\omega)) = -\rho(\zeta_p-1)F(\Theta_i)\exp(-\sum_{k=1}^d\Theta_kY_k)x_i.\]
  As a consequence, the image of $\omega$ in $\rH^1(\Gamma,V_0)$ is represented by
  \begin{equation}\label{Equ-v(omega)}
      v(\omega)=-(\zeta_p-1)\rho(F(\Theta_1)\exp(-\sum_{k=1}^d\Theta_kY_k)x_1,\dots,F(\Theta_d)\exp(-\sum_{k=1}^d\Theta_kY_k)x_d).
  \end{equation}
  Since for any $1\leq i,j\leq d$,
  \begin{equation}\label{Equ-Bridge-II}
      \begin{split}
          (\gamma_j-1)(F(\Theta_i)\exp(-\sum_{k=1}^d\Theta_kY_k)x_i)=&F(\Theta_i)\exp(-\sum_{k=1}^d\Theta_kY_k)(\exp(-(\zeta_p-1)\Theta_j)-1)x_i\\
          =&(\zeta_p-1)F(\Theta_i)F(\Theta_j)\exp(-\sum_{k=1}^d\Theta_kY_k)\Theta_j(x_i),
      \end{split}
  \end{equation}
  and $\Theta_j(x_i) = \Theta_i(x_j)$, we deduce that
  \[(\gamma_j-1)(F(\Theta_i)\exp(-\sum_{k=1}^d\Theta_kY_k)x_i) = (\gamma_i-1)(F(\Theta_i)\exp(-\sum_{k=1}^d\Theta_kY_k)x_j)\]
  and hence that
  \[v'(\omega):=(F(\Theta_1)\exp(-\sum_{k=1}^d\Theta_kY_k)x_1,\dots,F(\Theta_d)\exp(-\sum_{k=1}^d\Theta_kY_k)x_d)\in\wedge^1V_0^d\]
  represents an element in $\rH^1(\Gamma,V_0)$. 
  
  Therefore, as a cohomological class, we have 
  \[v(\omega) = \rho(\zeta_p-1)v'(\omega)\in\rho(\zeta_p-1)\rH^1(\Gamma,V_0).\]
  In other words, the map $\HIG(H,\theta_H)\to\rR\Gamma(\Gamma,V_0)$ carries $\rH^1(\HIG(H,\theta_H))$ into $\rho(\zeta_p-1)\rH^1(\Gamma,V_0)$. We have to show it induces an isomorphism $\rH^1(\HIG(H,\theta_H))\simeq\rho(\zeta_p-1)\rH^1(\Gamma,V_0)$.

  The injectivity is obvious. Indeed, let $T$ be the total complex of the double complex in (\ref{Diag-DoubleComplex}) representing $\rR\Gamma(\Gamma,\HIG(M,\Theta_M))$. By the spectral sequence argument, we have
  \[\rH^1(\HIG(H,\Theta_H)) = E_2^{1,0} = E_{\infty}^{1,0}\subset \rH^1(T) \cong \rH^1(\Gamma,V_0)\]
  and the desired injectivity follows.

  It remains to prove the surjectivity. For this, let $y_1,\dots,y_d\in H$ such that 
  \[v' = (F(\Theta_1)\exp(-\sum_{j=1}^d\Theta_jY_j)y_1,\dots,F(\Theta_d)\exp(-\sum_{j=1}^d\Theta_jY_j)y_d)\in \wedge^1V_0^d\]
  represents an element in $\rH^1(\rK(V_0;\gamma_1-1,\dots,\gamma_d-1))\cong\rH^1(\Gamma,V_0)$. Equivalently, we have that for any $1\leq i,j\leq d$,
  \begin{equation*}
      \begin{split}
          (\gamma_j-1)(F(\Theta_i)\exp(-\sum_{k=1}^d\Theta_kY_k)y_i)=(\gamma_i-1)(F(\Theta_j)\exp(-\sum_{k=1}^d\Theta_kY_k)y_j).
      \end{split}
  \end{equation*}
  By (\ref{Equ-Bridge-II}), this amounts to that
  \begin{equation*}
      \begin{split}
          (\zeta_p-1)F(\Theta_i)F(\Theta_j)\exp(-\sum_{k=1}^d\Theta_kY_k)\Theta_j(y_i)=(\zeta_p-1)F(\Theta_i)F(\Theta_j)\exp(-\sum_{k=1}^d\Theta_kY_k)\Theta_i(y_j).
      \end{split}
  \end{equation*}
  By noting that $F(\Theta_i)$'s are invertible (as $\Theta_i$'s are topologically nilpotent and $F(X)\equiv 1\mod X$), we conclude that $v'$ represents an element in $\rH^1(\Gamma,V_0)$ if and only if $\Theta_i(y_j) = \Theta_j(y_i)$ for any $1\leq i,j\leq d$. As a consequence, $\omega' = \sum_{i=1}^dy_i\otimes\frac{\rho\dlog T_i}{t}$ represents an element in $\rH^1(\HIG(H,\theta_H))$.

  Now for any $v'\in \rH^1(\Gamma,V_0)$, by (\ref{Equ-v(omega)}), we know that $v(\omega') = \rho(\zeta_p-1)v'$. That is, $\rho(\zeta_p-1)v'$ is contained in the image of 
  \[\rH^1(\HIG(H,\theta_H))\to\rho(\zeta_p-1)\rH^1(\Gamma,V_0).\]
  As $v'$ is arbitrary in $\rH^1(\Gamma,V_0)$, this proves the desired surjectivity and we complete the proof of (i).

  Now, we are going to prove (ii). By (i) and \cite[Lem. 8.16]{BMS18}, the composite
  \[\tau^{\leq 1}\HIG(H,\theta_H)\to \HIG(H,\theta_H)\to \rR\Gamma(\Gamma,V_0)\]
  uniquely factors through 
  \[\tau^{\leq 1}\rL\eta_{\rho(\zeta_p-1)}\rR\Gamma(\Gamma,V_0)\to\rL\eta_{\rho(\zeta_p-1)}\rR\Gamma(\Gamma,V_0)\to\rR\Gamma(\Gamma,V_0).\]
  So it is enough to show that $\tau^{\leq 1}\HIG(H,\theta_H)\simeq\tau^{\leq 1}\rL\eta_{\rho(\zeta_p-1)}\rR\Gamma(\Gamma,V_0)$.

  However, by the construction in \cite[Lem. 6.10]{BMS18}, the map on $\rH^1$ induced by $\rL\eta_{\rho(\zeta_p-1)}\rR\Gamma(\Gamma,V_0)\to\rR\Gamma(\Gamma,V_0)$ is given by
  \[\rH^1(\rL\eta_{\rho(\zeta_p-1)}\rR\Gamma(\Gamma,V_0))\cong\rH^1(\Gamma,V_0)/\rH^1(\Gamma,V_0)[\rho(\zeta_p-1)]\xrightarrow{\times(\zeta_p-1)\rho}\rH^1(\Gamma,V_0).\]
  In particular, it is an injection and identifies $\rH^1(\rL\eta_{\rho(\zeta_p-1)}\rR\Gamma(\Gamma,V_0))$ with $\rho(\zeta_p-1)\rH^1(\Gamma,V_0)$. Then the desired quasi-isomorphism follows from (i).
 \end{proof}

\begin{proof}[Proof of Theorem \ref{Thm-TruncationOne}:]
  We first show that the composite
    $\tau^{\leq 1}\HIG(\calH,\theta_{\calH})\to\HIG(\calH,\theta_{\calH})\xrightarrow{{\rm Cor.} \ref{Cor-IntegralSimpson}}\rR\nu_*\calL$ uniquely factors through $\rL\eta_{\rho(\zeta_p-1)}\rR\nu_*\calL\xrightarrow{{\rm Thm.} \ref{Thm-PerfectComplex}}\rR\nu_*\calL$. 
    By \cite[Lem. 8.16]{BMS18}, it suffices to show that the induced map
    \[\rH^1(\HIG(\calH,\theta_{\calH}))\to\rH^1(\rR\nu_*\calL)\]
    factors through $\rho(\zeta_p-1)\rH^1(\rR\nu_*\calL)$. Since this is a local problem, we may assume $\frakX = \Spf(R^+)$ is small affine such that $(\calH,\theta_{\calH})$ is induced by a $b$-small Higgs module over $R^+$ for some $b>a$. Then we are reduced to Proposition \ref{Prop-TruncationOne}(1).

    It remains to prove $\tau^{\leq 1}\HIG(\calH,\theta_{\calH})\to\tau^{\leq 1}\rL\eta_{\rho(\zeta_p-1)}\rR\nu_*\calL$ is a quasi-isomorphism. As it is still a local problem, we are reduced to Proposition \ref{Prop-TruncationOne}(2).
\end{proof}

\subsection{A conjectural analogue of Deligne--Illusie decomposition}

  In \cite{Min21}, the first named author proved the following decomposition of Deligne--Illusie type:
   \begin{thm}[\emph{\cite[Thm. 4.1]{Min21}}]\label{Thm-Min}
        Let $\frakX$ be the base change of a smooth $p$-adic formal scheme $\frakX_0$ over $\calO_K=W(\kappa)$ (in particular, $\rho=\zeta_p-1)$ along the inclusion $W(\kappa)\to \calO_C$ so that $\frakX$ has a natural lifting induced by the base change of $\frakX_0$ along $W(\kappa)\to \rA_2$. Then there exists a quasi-isomorphism
       \begin{equation}\label{Equ-Min}
           \gamma:\tau^{\leq p-1}\HIG(\calO_{\frakX},0)\to \tau^{\leq p-1}\rL\eta_{\zeta_p-1}\rR\nu_*\OXp.
       \end{equation}
   \end{thm}
   \begin{proof}
       Let us sketch the argument in the proof of \cite[Thm. 4.1]{Min21} here. The key point is \cite[Prop. 8.15]{BMS18}, which claims that there exists a canonical quasi-isomorphism
       \[\widehat \rL_{\frakX/\calO_K}\{-1\}[-1]\xrightarrow{\simeq}\tau^{\leq 1}\rL\eta_{\zeta_p-1}\rR\nu_*\OXp,\]
       where $\widehat \rL_{\frakX/\calO_K}$ denotes the $p$-adic completion of the cotangent complex $\rL_{\frakX/\calO_K}$. 

       The natural lifting of $\frakX$ then induces a quasi-isomorphism 
       \[\widehat \rL_{\frakX/\calO_K}\simeq \calO_{\frakX}\{1\}[1]\oplus\widehat \Omega^1_{\frakX}.\]
       So we get a quasi-isomorphism
       \[\gamma_1:\calO_{\frakX}[0]\oplus(\zeta_p-1)\widehat \Omega_{\frakX}^1(-1)[-1]\simeq\tau^{\leq 1}\HIG(\calO_{\frakX},0)\to\tau^{\leq 1}\rL\eta_{\zeta_p-1}\rR\nu_*\OXp.\]
       Then one can construct $\gamma$ by taking anti-symmetrization of $\gamma_1$, which is a standard argument of Deligne--Illusie (cf. \cite[Thm. 2.1]{DI}).
   \end{proof}
   Applying Theorem \ref{Thm-TruncationOne} to $\calL = \OXp$, we get another quasi-isomorphism 
   \begin{equation}\label{equ:gamma_1^prime}
       \gamma_1':\calO_{\frakX}[0]\oplus\rho\widehat \Omega_{\frakX}^1(-1)[-1]\simeq\tau^{\leq 1}\HIG(\calO_{\frakX},0)\to\tau^{\leq 1}\rL\eta_{\rho(\zeta_p-1)}\rR\nu_*\OXp.
   \end{equation}
   Using this, we have the following analogue of the Deligne--Illusie decomposition:
   \begin{thm}\label{Thm-DI}
       Let $\frakX$ be a smooth formal scheme over $\calO_C$ with a fixed lifting $\calO_{\widetilde \frakX}$ over $\rA_2$. Then there exists a quasi-isomorphism
       \[\gamma^{\prime}:\tau^{\leq p-1}\HIG(\calO_{\frakX},0)\to \tau^{\leq p-1}\rL\eta_{(\zeta_p-1)\rho}\rR\nu_*\OXp.\]
   \end{thm}
   \begin{proof}
       Apply Theorem \ref{Thm-TruncationOne} (with respect to the given lifting $\calO_{\widetilde \frakX}$), we have a quasi-isomorphism (\ref{equ:gamma_1^prime}). Then the desired $\gamma^{\prime}$ can be deduced from the same argument as in \cite[Th. 4.1]{Min21}.
   \end{proof}

   It is a natural question to compare $\gamma^{\prime}$ with $\gamma$ in Theorem \ref{Thm-Min} . Note that by \cite[Lem. 6.10 and Lem. 6.11]{BMS18}, we have a canonical map
   \[\alpha:\rL\eta_{(\zeta_p-1)\rho}\rR\nu_*\OXp\to \rL\eta_{\zeta_p-1}\rR\nu_*\OXp\]
   making the following diagram commute
   \[\xymatrix@C=0.5cm{
     \rL\eta_{(\zeta_p-1)\rho}\rR\nu_*\OXp\ar[rr]^{\alpha}\ar[rd]&&\ar[ld]\rL\eta_{\zeta_p-1}\rR\nu_*\OXp\\
     &\rR\nu_*\OXp.
   }\]

   \begin{thm}\label{Thm-compare DI with Min}
       Let $\frakX$ be the base change of a smooth $p$-adic formal scheme over $\calO_K = \rW(\kappa)$ (and thus $\rho = \zeta_p-1$) and the fixed lifting is the natural one induced by $W(\kappa)\to A_2 = \Ainf/(\xi^2)$. Then the above $\gamma'$ is compatible with $\gamma$ in Theorem \ref{Thm-Min} in the following sense: There is a commutative diagram
          \begin{equation}\label{Diag-DI}
              \xymatrix@C=0.5cm{
              \bigoplus_{i=0}^{p-1}(\zeta_p-1)^i\widehat \Omega_{\frakX}^i(-i)[-i]\ar[d]^{\iota_{\zeta_p-1}}\ar[rr]^{\gamma'}&&\tau^{\leq p-1}\rL\eta_{(\zeta_p-1)^2}\rR\nu_*\OXp\ar[r]\ar[d]^{\alpha}&\rR\nu_*\OXp\ar@{=}[d]\\
              \bigoplus_{i=0}^{p-1}(\zeta_p-1)^i\widehat \Omega_{\frakX}^i(-i)[-i]\ar[rr]^{\gamma}&&\tau^{\leq p-1}\rL\eta_{\zeta_p-1}\rR\nu_*\OXp\ar[r]&\rR\nu_*\OXp,
              }
          \end{equation}
          where $\iota_{\zeta_p-1}$ is induced by multiplication $(\zeta_p-1)^i$ at each degree $0\leq i\leq p-1$.
   \end{thm}
   \begin{proof}
       To compare $\gamma'$ with $\gamma$, it is enough to show the following diagram is commutative
       \begin{equation}\label{Diag-DI-II}
              \xymatrix@C=0.5cm{
              (\zeta_p-1)\widehat \Omega_{\frakX}^1(-1)[-1]\ar[d]^{\times(\zeta_p-1)}\ar[rr]^{\gamma^{\prime}}&&\tau^{\leq 1}\rL\eta_{(\zeta_p-1)^2}\rR\nu_*\OXp\ar[r]&\rR\nu_*\OXp\ar@{=}[d]\\
              (\zeta_p-1)\widehat \Omega_{\frakX}^1(-1)[-1]\ar[rr]^{\gamma}&&\tau^{\leq 1}\rL\eta_{\zeta_p-1}\rR\nu_*\OXp\ar[r]&\rR\nu_*\OXp.
              }
          \end{equation}
Unraveling the definition of $\gamma$ and $\gamma'$, we need to show the composite
\[
F: \widehat \Omega_{\frakX}^1\{-1\}[-1]\to \widehat \rL_{\frakX/\bZ_p}\{-1\}[-1]\xrightarrow{\cite{BMS18}} \rL\eta_{\zeta_p-1}\rR\nu_*(\OXp)\to\rR\nu_*(\OXp)\xrightarrow{\times (\zeta_p-1)}\rR\nu_*(\OXp).
\]
and the composite 
\[
F^{\prime}:\widehat \Omega^1_{\frakX}\{-1\}[-1]\to\HIG(\calO_{\frakX},0)\xrightarrow{\simeq}\nu_*(\HIG(\calO\widehat \bC_{\pd}^+,\Theta))\to\rR\nu_*(\HIG(\calO\widehat \bC_{\pd}^+,\Theta))\xleftarrow{\simeq}\rR\nu_*(\OXp)
\]
coincide in the derived category $D(\calO_{\frakX})$, where $\widehat \Omega^1_{\frakX}\{-1\}=(\zeta_p-1)\Omega^1_{\frakX}(-1)$.

Since the problem is local, we may assume $\frakX = \Spf(R^+)$ is small affine. In this case, $\widehat \Omega^1_{\frakX}\{-1\}$ is the finite free $R^+$-module $\bigoplus_{i=1}^dR^+\cdot\frac{\zeta_p-1}{t}\dlog T_i$. 

%By construction, both $F$ and $F'$ are induced by two maps from $\widehat \Omega^1_{R^+}\{-1\}[-1]$ to $\rR\Gamma(\Gamma, \widehat R_{\infty}^+$. In fact, after composing with the natural map $\rR\Gamma(\Gamma, \widehat R_{\infty}^+)\to \rR\Gamma(\Spf(R^+),\rR\nu_*(\OXp))$, these two maps give back $F$ and $F'$. By abuse of notations, we still write $F,F':\widehat \Omega^1_{R^+}\{-1\}[-1]\to \rR\Gamma(\Gamma, \widehat R_{\infty}^+$. Now it suffices to prove t$F,F':\widehat \Omega^1_{R^+}\{-1\}[-1]\to \rR\Gamma(\Gamma, \widehat R_{\infty}^+$ coincide in $D(R^+)$.

We first claim that to compare $F,F'$ in $D(R^+)$, we only need to compare their cohomology groups. Indeed, we have 
\begin{equation*}
 \begin{split}  
\Hom_{D(R^+)}(\widehat \Omega^1_{R^+}\{-1\}[-1],\rR\Gamma(\Spf(R^+),\rR\nu_*(\OXp))&=\rH^0(\rR\Hom(\widehat \Omega^1_{R^+}\{-1\}[-1],\rR\Gamma(\Spf(R^+),\rR\nu_*(\OXp))\\
& =\rH^1((\widehat \Omega^1_{R^+}\{-1\})^{\vee}\otimes \rR\Gamma(\Spf(R^+),\rR\nu_*(\OXp))\\
&=(\widehat \Omega^1_{R^+}\{-1\})^{\vee}\otimes \rH^1(\rR\Gamma(\Spf(R^+),\rR\nu_*(\OXp))\\
&=\Hom(\widehat \Omega^1_{R^+}\{-1\},\rH^1(\rR\Gamma(\Spf(R^+),\rR\nu_*(\OXp)).
\end{split}
\end{equation*}
So it suffices to show $F,F':\widehat \Omega^1_{R^+}\{-1\}\to \rH^1(\Spf(R^+),\rR\nu_*(\OXp))$ are the same. By construction, both $F$ and $F'$ factor through $\rH^1(\Gamma,\widehat R_{\infty}^+)\to \rH^1(\Spf(R^+),\rR\nu_*(\OXp))$ (for the map $F$, see \cite[Section 8.3]{BMS18}). So we only need to show the induced two maps $F,F':\widehat \Omega^1_{R^+}\{-1\}\to H^1(\Gamma,\widehat R_{\infty}^+)$ are the same.

For any $\omega = \sum_{i=1}^dx_i\otimes\frac{\zeta_p-1}{t}\dlog T_i\in \widehat \Omega^1_{R^+}\{-1\}$, by the proof of Proposition \ref{Prop-TruncationOne} (recall that $\rho = \zeta_p-1$ as $\calO_K=\rW(\kappa)$), we have
\[F^{\prime}(\omega) = (\zeta_p-1)^2v(\omega) = ((\zeta_p-1)^2x_1,\dots,(\zeta_p-1)^2x_d)\in\rH^1(\Gamma,R^+)\subset\rH^1(\Gamma,\widehat R_{\infty}^+).\]
On the other hand, as $\gamma_i(T_j^{1/p^n}) = \zeta_{p^n}^{\delta_{ij}}T_j^{1/p^n}$ (and thus $\gamma_i$ plays the same role as $g$ in the proof of \cite[Prop. 8.17]{BMS18}), it follows from \cite[Prop. 8.17]{BMS18} that 
\[F(\omega)=(\zeta_p-1)\cdot((\zeta_p-1)x_1,\dots,(\zeta_p-1)x_d) = (\zeta_p-1)^2x_1,\dots,(\zeta_p-1)^2x_d).\]

So we conclude $F,F':\widehat \Omega_{\frakX}^1\{-1\}[-1]\to \rR\nu_*(\OXp)$ coincide in $D(\calO_{\frakX})$. As a consequence, we obtain the following commutative diagram
\begin{equation*}
    \xymatrix@C=0.5cm{
      \widehat\Omega^1_{\frakX}\{-1\}[-1]\ar[r]\ar@{=}[d]& \rL\eta_{\zeta_p-1}\rR\nu_*(\OXp)\ar[r]&\rR\nu_*(\OXp)\ar[d]^{\times(\zeta_p-1)}\\
    \widehat\Omega^1_{\frakX}\{-1\}[-1]\ar[r]&\rL\eta_{(\zeta_p-1)^2}\rR\nu_*(\OXp)\ar[r]&\rR\nu_*(\OXp).
    }
\end{equation*}
and equivalently the following commutative diagram
\begin{equation*}
    \xymatrix@C=0.5cm{
      \Omega^1_{\frakX}\{-1\}[-1]\ar[r]^-{\gamma^{\prime}}\ar[d]^{\times(\zeta_p-1)}& \rL\eta_{(\zeta_p-1)^2}\rR\nu_*(\OXp)\ar[r]&\rR\nu_*(\OXp)\ar@{=}[d]\\
    \Omega^1_{\frakX}\{-1\}[-1]\ar[r]^-{\gamma}&\rL\eta_{\zeta_p-1}\rR\nu_*(\OXp)\ar[r]&\rR\nu_*(\OXp)
    }
\end{equation*}
    as desired. This completes the proof.
\end{proof}

   \begin{rmk}\label{Rmk-Min}
       In Theorem \ref{Thm-Min}, the range of truncation $p-1$ is optimal. Namely, the quasi-isomorphism $\gamma$ can not extend to a quasi-isomorphism between $\HIG(\calO_{\frakX},0)$ and $\widetilde\Omega_{\frakX}:=\rL\eta_{\zeta_p-1}\rR\nu_*\OXp$. Indeed, if $\rL\eta_{\zeta_p-1}\rR\nu_*\OXp$ is formal, i.e. quasi-isomorphic to the trivial Higgs complex $\HIG(\calO_{\frakX},0)$, then \cite[Theorem 1.1]{Pet23} shows that there will be a contradiction. More precisely, \cite[Theorem 1.1]{Pet23} proves that there exists a smooth projective variety $\overline \frakX_0$ over a finite field $k$ of dimension $p+1$ that lifts to $\rW(\kappa)$, whose de Rham cohomology is not formal. Let $\frakX$ be the base change along $\rW(\kappa)\to \calO_C$ of this lift. Then by the base change property of the prismatic cohomology as well as the Hodge-Tate and de Rham comparisons, the formality of the complex $\tilde\Omega_{\frakX}$ will imply the formality of the de Rham cohomology of $\overline \frakX_0$, which is contradictory to \cite[Theorem 1.1]{Pet23}.
   \end{rmk}

   Inspired by Theorem \ref{Thm-DI}, we make the following conjecture, which can be viewed as an analogue of Deligne--Illusie decomposition with coefficients in small $\OXp$-representations.
  \begin{conj}\label{Conj-CompareCohomology}
     Keep assumptions in Theorem \ref{Thm-IntegralSimpson}. Then for any $a$-small $\OXp$-representation $\calL$ with induced Higgs bundle $(\calH,\theta_{\calH}:\calH\to \calH\otimes\rho\widehat \Omega^1_{\frakX}(-1))$, denote by $r$ the nilpotent length of $(\zeta_p-1)\theta_{\calH}$ modulo $p$, and then the natural morphism in Corollary \ref{Cor-IntegralSimpson} induces a quasi-isomorphism
     \[\tau^{\leq p-r}\HIG(\calH,\theta_{\calH})\simeq \tau^{\leq p-r}\rL\eta_{\rho(\zeta_p-1)}\rR\nu_*\calL.\]
  \end{conj}
  Note that the smallness condition implies that $r\leq p-1$ as $\frac{(\zeta_p-1)^{p-1}}{p}\in\calO_C^{\times}$, so Conjecture \ref{Conj-CompareCohomology} will not violate Theorem \ref{Thm-TruncationOne}. We remark that Theorem \ref{Thm-TruncationOne} and Theorem \ref{Thm-DI} show that Conjecture \ref{Conj-CompareCohomology} holds true for curves as well as when $p=2$. 

 It is worth pointing out that a similar phenomenon already appears in the non-abelian Hodge theory in characteristic $p$ \cite{OV} which the authors are informed by Mao Sheng.
     
 At the end of this paper, we ask the following question inspired by \cite{DI},\cite[Thm. 4.1]{Min21},:
  
  \begin{ques}\label{Ques-Upgrade}
     Let $\frakX$ be liftable and fix such a lifting. Let $a\geq\frac{1}{p-1}$ and $\calL$ be $a$-small $\OXp$-representation with associated Higgs bundle $(\calH,\theta_{\calH})$. Can we upgrade $\rL\eta_{\rho(\zeta_p-1)}$ to $\rL\eta_{\zeta_p-1}$ in Conjecture \ref{Conj-CompareCohomology}? Namely, do we have a quasi-isomorphism
     \[\tau^{\leq p-r}\HIG(\calH,\theta_{\calH})\simeq \tau^{\leq p-r}\rL\eta_{\zeta_p-1}\rR\nu_*\calL?\]
  \end{ques}
  \begin{rmk}\label{Rmk-Tsuji}
     Recall that in \cite{AGT}, Tsuji constructed an equivalence between the category of crystals on his Higgs site and the category of Higgs bundles satisfying certain convergence conditions. Both of the categories can be embedded full-faithfully into the category of $\OXp$-vector bundles. In \cite{Tsu21}, Tsuji gave a canonical quasi-isomorphism between the $(p-1-r)$-th truncations of Higgs complex and Hodge--Tate complex associated with a crystal on his Higgs site when the corresponding Higgs field is nilpotent of length $r$ modulo $p$. This is also one of our evidences to ask Question \ref{Ques-Upgrade}.
  \end{rmk}

%\bibliographystyle{alpha}
%\bibliography{IntegralSimpson}

\begin{thebibliography}{IntegralSimpson}

  \bibitem[AG22]{AG22} Ahmed Abbes, Michel Gros: {\it Correspondance de Simpson p-adique II : fonctorialit\'e par image directe propre et syst\`emes locaux de Hodge-Tate}, arXiv preprint arxiv:2210.10580, (2022).

 \bibitem[AGT16]{AGT} Ahmed Abbes, Michel Gros, Takeshi Tsuji: {\it The $p$-adic Simpson Correspondence}, Annals of Mathematics Studies. 193, Princeton University Press, (2016).

 \bibitem[AHLB23]{AHLB} Johannes Ansch\"utz, Ben Heuer, Arthur-C\'esar Le Bras: {\it Hodge-Tate stacks and non-abelian $p$-adic Hodge theory of v-perfect complexes on smooth rigid spaces}, arXiv preprint arxiv:2302.12747, (2023).

 \bibitem[Bha18]{Bha} Bhargav Bhatt: {Specializing varieties and their cohomology from characteristic $0$ to characteristic $p$}, Algebraic geometry: Salt Lake City 2015, 97:43–88, (2018).

% \bibitem[BL22a]{BL22a} Bhargav Bhatt, Jacob Lurie: {\it Absolute prismatic cohomology}, arxiv:2201.06120, (2022).

 \bibitem[BL22]{BL22b} Bhargav Bhatt, Jacob Lurie: {\it The prismatization of $p$-adic formal schemes}, arXiv preprint arxiv:2201.06124, (2022).
  
 \bibitem[BMS18]{BMS18} Bhargav Bhatt, Matthew Morrow, Peter Scholze: {\it Integral $p$-adic Hogde Theory}, Publ. math. de la IH\'ES 128, pp. 219-395, (2018).

\bibitem[Bos14]{Bos} Siegfried Bosch: {\it Lectures on formal and rigid geometry}, Lecture Notes in Mathematics 2105, Springer International Publishing Switzerland, (2014). %\href{https://doi.org/10.1007/978-3-319-04417-0}{https://doi.org/10.1007/978-3-319-04417-0}

 \bibitem[BS22]{BS22} Bhargav Bhatt, Peter Scholze: {\it Prisms and prismatic cohomology},  Ann. of Math. (2) 196, no. 3, (2022).

  \bibitem[CK19]{CK19} Kestutis \v Cesnavi\v cius, Teruhisa Koshikawa: {\it The $\Ainf$-cohomology in the semistable case}, Compos. Math. 155, no. 11, pp. 2039–2128, (2019).

 \bibitem[DI87]{DI} Pierre Deligne, Luc Illusie: {\it Rel\`evements modulo $p^2$ et d\'ecomposition du complexe de de Rham}, Inventiones Mathematicae 89(2), 247-270, (1987).
 
 \bibitem[DW05]{DW05} Christopher Deninger, Annette Werner: {\it Vector bundles on $p$-adic curves and parallel transport}, Ann. Sci. \'Ecole Norm. Sup. (4) 38, no. 4, pp. 553–597, (2005)

 \bibitem[Fal05]{Fal05}Gerd Faltings: {\it A $p$-adic Simpson correspondence}, Advances in Mathematics. 198, pp. 847-862, 2005.

 \bibitem[Heu22a]{Heu22a} Ben Heuer: {\it Line bundles on rigid spaces in the $v$-topology}, Forum Math. Sigma 10, Paper No. e82, 36 pp., (2022).

 \bibitem[Heu22b]{Heu22b} Ben Heuer: {\it Moduli spaces in $p$-adic non-abelian Hodge theory}, arXiv preprint arxiv:2207.13819, (2022).

 \bibitem[Heu23]{Heu23} Ben Heuer: {\it A $p$-adic Simpson correspondence for smooth proper rigid varieties}, arXiv preprint arxiv:2307.01303, (2023).

  \bibitem[HMW]{HMW} Ben Heuer, Lucas Mann, Annette Werner: {\it The p-adic Corlette–Simpson correspondence for abeloids}, Mathematische Annalen, 385(3-4):1639–1676, (2023).

%\bibitem[Li22]{Li22} 
%  Shizhang Li: {\it Integral $p$-adic Hodge filtrations in low dimension and ramification}, J. Eur. Math. Soc. 24, no. 11, pp. 3801–3827, (2022).
 %\href{https://doi.org/10.4171/JEMS/1237}{https://doi.org/10.4171/JEMS/1237}

 \bibitem[LZ17]{LZ17} Ruochuan Liu, Xinwen Zhu: {\it Rigidty and a Riemman-Hilbert correspondence for $p$-adic local systems}, Inv. Math. 207, pp. 291-343, (2017).

 \bibitem[Min21]{Min21} Yu Min: {\it Integral p-adic Hodge theory of formal schemes in low ramification}, Algebra Number theory, Vol. 15, No. 4, (2021).
 %\href{https://doi.org/10.2140/ant.2021.15.1043}{https://doi.org/10.2140/ant.2021.15.1043}

 \bibitem[MT20]{MT} Matthew Morrow, Takeshi Tsuji: {\it Generalized representations as $q$-connections in integral $p$-adic Hodge theory}, arXiv preprint arXiv: 2010.04059v2, (2020).

 \bibitem[MW22]{MW22} Yu Min, Yupeng Wang: {\it $p$-adic Simpson correpondence via prismatic crystals}, arxiv:2201.08030, (2022). Accepted by J. Eur. Math. Soc..

 \bibitem[OV07]{OV} A. Ogus, V. Vologodsky: {\it Nonabelian Hodge theory in characteristic $p$}, Publ. math. de l'IH\'ES, 106(1), 1-138, (2007).

 \bibitem[Pet23]{Pet23} Alexander Petrov: {\it Non-decomposability of the de Rham complex and non-semisimplicity of the Sen operator}, arXiv preprint arXiv:2302.11389, (2023).

% \bibitem[Sch12]{Sch-IHES}Peter Scholze: {\it Perfectoid spaces}, Publ. math. de l'IH\'ES 116, no. 1, pp. 245-313, (2012).

 \bibitem[Sch13]{Sch-Pi} Peter Scholze: {\it $p$-adic Hodge theory for rigid-analytic varieties},  Forum of Mathematics, Pi, 1, e1, 77 pages, (2013).

 %\bibitem[Sch13b]{Sch-Survey} Peter Scholze: {\it Perfectoid spaces: a survey}, Current developments in mathematics, Volume 2012, Issue 1, pp. 193-227, (2013).

 \bibitem[Sim92]{Sim}C.T. Simpson: {\it Higgs bundles and local systems}, Publ. math. de l'IH\'ES 75, pp. 5-95, (1992).


 \bibitem[SW]{SW} Mao Sheng, Yupeng Wang: {\it The small $p$-adic Simpson correspondence in semi-stable reduction case}, In preparation.

 \bibitem[SZ18]{SZ} Tam\'as Szameuly, Gergely Z\'abr\'adi: {\it A $p$-adic {H}odge decomposition according to {B}eilinson}, Algebraic geometry: Salt Lake City 2015, 75:495–572, Proc. Sympos. Pure Math., 97.2, Amer. Math. Soc., Providence, RI, (2018).
 
 \bibitem[Tia23]{Tia23} Yichao Tian: {\it Finiteness and duality for the cohomology of prismatic crystals}, J. Reine. Angew. Math., 800:217–257, (2023).

 \bibitem[Tsu18]{Tsu18} Takeshi Tsuji: {\it Notes on the local $p$-adic {S}impson correspondence}, Mathematische Annalen, 371(1-2):795–881, (2018).

 \bibitem[Tsu21]{Tsu21} Takeshi Tsuji: {\it Integral cohomologies in the $p$-adic {S}impson correspondence}, talk in the conference on the occasion of Takeshi Saito's 60th birthday, \href{https://www.youtube.com/watch?v=cZQRPSwaOJM}{https://www.youtube.com/watch?v=cZQRPSwaOJM}, (2021).
  
 \bibitem[Wan23]{Wang} Yupeng Wang: {\it A $p$-adic {S}impson correspondence for rigid analytic varieties}, Algebra Number Theory, 17(8):1453–1499, (2023).

 \bibitem[Xu17]{Xu17} Daxin Xu: {\it Transport parall{\`e}le et correspondance de {S}impson $p$-adique}, Forum of Mathematics, Sigma. Vol. 5, page e13. Cambridge University Press, (2017).

 \bibitem[Xu22]{Xu22} Daxin Xu: {\it Parallel transport for {H}iggs bundles over $p$-adic curves}, arXiv preprint arXiv:2201.06697, (2022).


\end{thebibliography}

\end{document}